\documentclass[11pt, reqno]{amsart}
\usepackage{amsfonts,latexsym,enumerate}
\usepackage{amsmath}
\usepackage{amscd} 
\usepackage{float,amsmath,amssymb,mathrsfs,bm,multirow,graphics}
\usepackage[dvips]{graphicx}
\usepackage[percent]{overpic}
\usepackage{amsaddr}
\usepackage[numbers,sort&compress]{natbib}
\usepackage{xcolor}
\usepackage{todonotes}

\newcommand{\R}{{\Bbb R}}

\newcommand{\C}{{\Bbb C}}
\newcommand{\D}{{\Bbb D}}

\newcommand{\diag}{\text{\upshape diag\,}}
\newcommand{\re}{\text{\upshape Re\,}}
\newcommand{\im}{\text{\upshape Im\,}}
\newcommand{\ntlim}{\lim^\angle}

\allowdisplaybreaks

\newtheorem{theorem}{Theorem}[section]

\newtheorem{lemma}[theorem]{Lemma}

\newtheorem{remark}[theorem]{Remark}

\newtheorem{RHproblem}[theorem]{RH problem}
\newtheorem{figuretext}{Figure}

\numberwithin{equation}{section}

\usepackage[colorlinks=true]{hyperref}
\hypersetup{urlcolor=blue, citecolor=red, linkcolor=blue}

\usepackage{tikz}
\usepackage{pict2e}
\usetikzlibrary{arrows,calc,chains, positioning, shapes.geometric,shapes.symbols,decorations.markings,arrows.meta}
\usetikzlibrary{patterns,angles,quotes}
\usetikzlibrary{decorations.markings}
\tikzset{middlearrow/.style={
			decoration={markings,
				mark= at position 0.6 with {\arrow{#1}} ,
			},
			postaction={decorate}
		}
	}
\tikzset{->-/.style={decoration={
				markings,
				mark=at position #1 with {\arrow{latex}}},postaction={decorate}}}
	
\tikzset{-<-/.style={decoration={
				markings,
				mark=at position #1 with {\arrowreversed{latex}}},postaction={decorate}}}
				
				\tikzset{
	master/.style={
		execute at end picture={
			\coordinate (lower right) at (current bounding box.south east);
			\coordinate (upper left) at (current bounding box.north west);
		}
	},
	slave/.style={
		execute at end picture={
			\pgfresetboundingbox
			\path (upper left) rectangle (lower right);
		}
	}
}
\tikzset{cross/.style={cross out, draw, 
         minimum size=2*(#1-\pgflinewidth), 
         inner sep=0pt, outer sep=0pt}}

\def\Xint#1{\mathchoice
{\XXint\displaystyle\textstyle{#1}}%
{\XXint\textstyle\scriptstyle{#1}}%
{\XXint\scriptstyle\scriptscriptstyle{#1}}%
{\XXint\scriptscriptstyle\scriptscriptstyle{#1}}%
\!\int}
\def\XXint#1#2#3{{\setbox0=\hbox{$#1{#2#3}{\int}$ }
\vcenter{\hbox{$#2#3$ }}\kern-.59\wd0}}

\def\dashint{\Xint-}

\title[Boussinesq's equation: asymptotics for $\MakeLowercase{\frac{x}{t}}\in (0,\frac{1}{\sqrt{3}})$]
{Boussinesq's equation for water waves: \\ asymptotics in sector V}

\author{C. Charlier$^{1}$ and J. Lenells$^{2}$}

\address{$^{1}$Centre for Mathematical Sciences, Lund University,
22100 Lund, Sweden. \\
$^{2}$Department of Mathematics, KTH Royal Institute of Technology, \\
10044 Stockholm, Sweden.}
\email{christophe.charlier@math.lu.se}
\email{jlenells@kth.se}

\begin{document}

\begin{abstract}
We consider the Boussinesq equation on the line for a broad class of Schwartz initial data for which (i) no solitons are present, (ii) the spectral functions have generic behavior near $\pm 1$, and (iii) the solution exists globally. In a recent work, we identified ten main sectors describing the asymptotic behavior of the solution, and for each of these sectors we gave an exact expression for the leading asymptotic term. In this paper, we give a proof for the formula corresponding to the sector $\frac{x}{t}\in (0,\frac{1}{\sqrt{3}})$. 
\end{abstract}

\maketitle



\setcounter{tocdepth}{1}

\section{Introduction}

Because of the great complexity of the water wave problem, several approximate models have been considered. One famous such model is the Boussinesq equation \cite{B1872}
\begin{align}\label{boussinesq}
u_{tt} = u_{xx} + (u^2)_{xx} + u_{xxxx}.
\end{align}
Here $u(x,t)$ is a real-valued function and subscripts denote partial derivatives. This nonlinear equation models small-amplitude dispersive waves in shallow water that can propagate in both the right and left directions \cite{J1997}. The Boussinesq equation is ill-posed, and is therefore often referred to as the ``bad" Boussinesq equation. 
Partly due to this ill-posedness, the mathematical literature on \eqref{boussinesq} is limited:
we refer to \cite{H1973, B1976, BZ2002} for results on soliton solutions, to \cite{Z1974} for a Lax pair, and to \cite{LS1985, Y2002} for nonexistence results of global solutions for certain initial-boundary value problems. Until recently \cite{CLmain}, the problem of obtaining long-time asymptotics for the solution of \eqref{boussinesq} was still open, see e.g. Deift's list of open problems in \cite{D2008}.

In \cite{CLmain}, we developed an inverse scattering approach to the initial value problem for \eqref{boussinesq}. Among other things, we proved in \cite{CLmain} that for a large class of initial data, the solution is unique and exists globally. The asymptotic behavior of such solutions is described by ten main asymptotic sectors, and for each of these sectors the leading asymptotic term was given in \cite[Theorem 2.14]{CLmain}. For reasons of length, the proofs of some of these asymptotic formulas were omitted in \cite{CLmain}. The purpose of this paper is to give the proof for the formula corresponding to the sector $\frac{x}{t} \in (0,\frac{1}{\sqrt{3}})$ (this sector was denoted ``sector V" in \cite{CLmain}).

Results on long-time asymptotics for other integrable equations can be found in e.g. \cite{BM2017, BIK2009, BJM2018, BKS2022, BKST2009, CF2022, DZ1993, DIZ1993, DMM2019, FL2010, GM2020, GT2009, HXF2015, LGWW2019, RS2019, TVZ2006, XF2020, YF2023}.

\section{main results}\label{section:main results}

We consider real-valued initial data in the Schwartz class,
\begin{align}\label{initial data}
u_{0}(x):=u(x,0) \in \mathcal{S}(\mathbb{R}), \qquad u_{1}(x):=u_{t}(x,0) \in \mathcal{S}(\mathbb{R}).
\end{align}
We further assume that $\int_{\mathbb{R}}u_{1}(x)dx = 0$. As noted in \cite{CLmain}, this assumption is natural because it ensures that the total mass $\int_{\mathbb{R}}u(x,t)dx$ is conserved in time. The approach of \cite{CLmain} provides a representation for the solution to the initial value problem in terms of the solution $n$ of a row-vector Riemann–Hilbert (RH) problem. This RH problem depends on two scalar reflection coefficients, $r_{1}(k)$ and $r_{2}(k)$, which are defined as follows (see \cite{CLmain} for details).

\subsection{Definition of $r_{1}$ and $r_{2}$.} Let $\omega := e^{\frac{2\pi i}{3}}$ and define $\{l_j(k), z_j(k)\}_{j=1}^3$ by
\begin{align}\label{lmexpressions intro}
& l_{j}(k) = i \frac{\omega^{j}k + (\omega^{j}k)^{-1}}{2\sqrt{3}}, \qquad z_{j}(k) = i \frac{(\omega^{j}k)^{2} + (\omega^{j}k)^{-2}}{4\sqrt{3}}, \qquad k \in \C\setminus \{0\}.
\end{align}
Let $P(k)$ and $\mathsf{U}(x,k)$ be given by
\begin{align*}
P(k) = \begin{pmatrix}
1 & 1 & 1  \\
l_{1}(k) & l_{2}(k) & l_{3}(k) \\
l_{1}(k)^{2} & l_{2}(k)^{2} & l_{3}(k)^{2}
\end{pmatrix}, \quad \mathsf{U}(x,k) = P(k)^{-1} \begin{pmatrix}
0 & 0 & 0 \\
0 & 0 & 0 \\
-\frac{u_{0x}}{4}-\frac{iv_{0}}{4\sqrt{3}} & -\frac{u_{0}}{2} & 0
\end{pmatrix} P(k),
\end{align*} 
where $v_{0}(x) = \int_{-\infty}^{x}u_{1}(x')dx'$. Define $X(x,k), X^A(x,k)$ as the unique solutions to the Volterra integral equations
\begin{align*}  
 & X(x,k) = I - \int_x^{\infty} e^{(x-x')\mathcal{L}(k)} (\mathsf{U}X)(x',k) e^{-(x-x')\mathcal{L}(k)} dx',
	\\
 & X^A(x,k) = I + \int_x^{\infty} e^{-(x-x')\mathcal{L}(k)} (\mathsf{U}^T X^A)(x',k)e^{(x-x')\mathcal{L}(k)} dx',	
\end{align*}
where $(\cdot)^T$ is the transpose operation and $\mathcal{L} = \diag(l_1 , l_2 , l_3)$. Define $s(k)$ and $s^A(k)$ by 
\begin{align*}
& s(k) = I - \int_\R e^{-x \mathcal{L}(k)}(\mathsf{U}X)(x,k)e^{x \mathcal{L}(k)}dx, & & s^A(k) = I + \int_\R e^{x \mathcal{L}(k)}(\mathsf{U}^T X^A)(x,k)e^{-x \mathcal{L}(k)}dx.
\end{align*}
The two spectral functions $\{r_j(k)\}_1^2$ are defined by
\begin{align}\label{r1r2def}
\begin{cases}
r_1(k) = \frac{(s(k))_{12}}{(s(k))_{11}}, & k \in \hat{\Gamma}_{1}\setminus \hat{\mathcal{Q}},
	\\ 
r_2(k) = \frac{(s^A(k))_{12}}{(s^A(k))_{11}}, \quad & k \in \hat{\Gamma}_{4}\setminus \hat{\mathcal{Q}},
\end{cases}
\end{align}	
where the contours $\hat{\Gamma}_j = \Gamma_{j} \cup \partial \D$ and the set $\hat{\mathcal{Q}}$ are defined in Section \ref{notationsubsec}.

\subsection{Assumptions}\label{assumptionssubsec}
As in \cite[Theorem 2.14]{CLmain}, our results will be valid under the following assumptions:
\begin{enumerate}[$(i)$]
\item There are no solitons: we suppose that $s(k)_{11}$ is nonzero for $k \in (\bar{D}_2 \cup \partial \D) \setminus \hat{\mathcal{Q}}$, where $D_{2}$ is the open set shown in Figure \ref{fig: Dn} and $\partial \mathbb{D}$ is the unit circle.
\item The spectral functions $s$ and $s^{A}$ have generic behavior near $k=1$ and $k=-1$: we suppose for $k_{\star} =1$ and $k_{\star}=-1$ that
\begin{align*}
& \lim_{k \to k_{\star}} (k-k_{\star}) s(k)_{11} \neq 0, & & \hspace{-0.1cm} \lim_{k \to k_{\star}} (k-k_{\star}) s(k)_{13} \neq 0, & & \hspace{-0.1cm} \lim_{k \to k_{\star}} s(k)_{31} \neq 0, & & \hspace{-0.1cm} \lim_{k \to k_{\star}} s(k)_{33} \neq 0, \\
& \lim_{k \to k_{\star}} (k-k_{\star}) s^{A}(k)_{11} \neq 0, & & \hspace{-0.1cm} \lim_{k \to k_{\star}} (k-k_{\star}) s^{A}(k)_{31} \neq 0, & & \hspace{-0.1cm} \lim_{k \to k_{\star}} s^{A}(k)_{13} \neq 0, & & \hspace{-0.1cm} \lim_{k \to k_{\star}} s^{A}(k)_{33} \neq 0.
\end{align*}
\item There exists a global solution to the initial value problem \eqref{boussinesq}--\eqref{initial data}: we suppose that $r_{1}(k)=0$ for all $k \in [0,i]$, where $[0,i]$ is the vertical segment from $0$ to $i$.
\end{enumerate}

\begin{figure}
\begin{center}
\begin{tikzpicture}[scale=0.7]
\node at (0,0) {};
\draw[black,line width=0.45 mm,->-=0.4,->-=0.85] (0,0)--(30:4);
\draw[black,line width=0.45 mm,->-=0.4,->-=0.85] (0,0)--(90:4);
\draw[black,line width=0.45 mm,->-=0.4,->-=0.85] (0,0)--(150:4);
\draw[black,line width=0.45 mm,->-=0.4,->-=0.85] (0,0)--(-30:4);
\draw[black,line width=0.45 mm,->-=0.4,->-=0.85] (0,0)--(-90:4);
\draw[black,line width=0.45 mm,->-=0.4,->-=0.85] (0,0)--(-150:4);

\draw[black,line width=0.45 mm] ([shift=(-180:2.5cm)]0,0) arc (-180:180:2.5cm);
\draw[black,arrows={-Triangle[length=0.2cm,width=0.18cm]}]
($(3:2.5)$) --  ++(90:0.001);
\draw[black,arrows={-Triangle[length=0.2cm,width=0.18cm]}]
($(57:2.5)$) --  ++(-30:0.001);
\draw[black,arrows={-Triangle[length=0.2cm,width=0.18cm]}]
($(123:2.5)$) --  ++(210:0.001);
\draw[black,arrows={-Triangle[length=0.2cm,width=0.18cm]}]
($(177:2.5)$) --  ++(90:0.001);
\draw[black,arrows={-Triangle[length=0.2cm,width=0.18cm]}]
($(243:2.5)$) --  ++(330:0.001);
\draw[black,arrows={-Triangle[length=0.2cm,width=0.18cm]}]
($(297:2.5)$) --  ++(210:0.001);

\draw[black,line width=0.15 mm] ([shift=(-30:0.55cm)]0,0) arc (-30:30:0.55cm);

\node at (0.8,0) {$\tiny \frac{\pi}{3}$};

\node at (-1:2.9) {\footnotesize $\Gamma_8$};
\node at (60:2.9) {\footnotesize $\Gamma_9$};
\node at (120:2.9) {\footnotesize $\Gamma_7$};
\node at (181:2.9) {\footnotesize $\Gamma_8$};
\node at (240:2.83) {\footnotesize $\Gamma_9$};
\node at (300:2.83) {\footnotesize $\Gamma_7$};

\node at (105:1.45) {\footnotesize $\Gamma_1$};
\node at (138:1.45) {\footnotesize $\Gamma_2$};
\node at (223:1.45) {\footnotesize $\Gamma_3$};
\node at (-104:1.45) {\footnotesize $\Gamma_4$};
\node at (-42:1.45) {\footnotesize $\Gamma_5$};
\node at (43:1.45) {\footnotesize $\Gamma_6$};

\node at (97:3.3) {\footnotesize $\Gamma_4$};
\node at (144:3.3) {\footnotesize $\Gamma_5$};
\node at (217:3.3) {\footnotesize $\Gamma_6$};
\node at (-96:3.3) {\footnotesize $\Gamma_1$};
\node at (-35:3.3) {\footnotesize $\Gamma_2$};
\node at (36:3.3) {\footnotesize $\Gamma_3$};
\end{tikzpicture}
\hspace{1.7cm}
\begin{tikzpicture}[scale=0.7]
\node at (0,0) {};
\draw[black,line width=0.45 mm] (0,0)--(30:4);
\draw[black,line width=0.45 mm] (0,0)--(90:4);
\draw[black,line width=0.45 mm] (0,0)--(150:4);
\draw[black,line width=0.45 mm] (0,0)--(-30:4);
\draw[black,line width=0.45 mm] (0,0)--(-90:4);
\draw[black,line width=0.45 mm] (0,0)--(-150:4);

\draw[black,line width=0.45 mm] ([shift=(-180:2.5cm)]0,0) arc (-180:180:2.5cm);
\draw[black,line width=0.15 mm] ([shift=(-30:0.55cm)]0,0) arc (-30:30:0.55cm);

\node at (120:1.6) {\footnotesize{$D_{1}$}};
\node at (-60:3.7) {\footnotesize{$D_{1}$}};

\node at (180:1.6) {\footnotesize{$D_{2}$}};
\node at (0:3.7) {\footnotesize{$D_{2}$}};

\node at (240:1.6) {\footnotesize{$D_{3}$}};
\node at (60:3.7) {\footnotesize{$D_{3}$}};

\node at (-60:1.6) {\footnotesize{$D_{4}$}};
\node at (120:3.7) {\footnotesize{$D_{4}$}};

\node at (0:1.6) {\footnotesize{$D_{5}$}};
\node at (180:3.7) {\footnotesize{$D_{5}$}};

\node at (60:1.6) {\footnotesize{$D_{6}$}};
\node at (-120:3.7) {\footnotesize{$D_{6}$}};

\node at (0.8,0) {$\tiny \frac{\pi}{3}$};

\draw[fill] (0:2.5) circle (0.1);
\draw[fill] (60:2.5) circle (0.1);
\draw[fill] (120:2.5) circle (0.1);
\draw[fill] (180:2.5) circle (0.1);
\draw[fill] (240:2.5) circle (0.1);
\draw[fill] (300:2.5) circle (0.1);

\node at (0:2.9) {\footnotesize{$\kappa_1$}};
\node at (60:2.85) {\footnotesize{$\kappa_2$}};
\node at (120:2.85) {\footnotesize{$\kappa_3$}};
\node at (180:2.9) {\footnotesize{$\kappa_4$}};
\node at (240:2.85) {\footnotesize{$\kappa_5$}};
\node at (300:2.85) {\footnotesize{$\kappa_6$}};

\draw[dashed] (-6.3,-3.8)--(-6.3,3.8);

\end{tikzpicture}
\end{center}
\begin{figuretext}\label{fig: Dn}
The contour $\Gamma = \cup_{j=1}^9 \Gamma_j$ in the complex $k$-plane (left) and the open sets $D_{n}$, $n=1,\ldots,6$, together with the sixth roots of unity $\kappa_j$, $j = 1, \dots, 6$ (right).
\end{figuretext}
\end{figure}

We emphasize that Assumption $(ii)$ is generic. Let us also comment on Assumptions $(i)$ and $(iii)$:
\begin{enumerate}[\hspace{15pt}$\bullet$]
\item In \cite{CLsolitons}, the long-time asymptotics for the solution to the initial value problem \eqref{boussinesq}--\eqref{initial data} is obtained in the sector $\frac{x}{t}\in (\smash{\frac{1}{\sqrt{3}}},1)$ in the case when solitons are present. The solitonless assumption $(i)$ is made here for simplicity; we believe that the case when solitons are present can be handled using similar ideas as in \cite{CLsolitons}. 
\item The Boussinesq equation is ill-posed and admits solutions that blow up at any given positive time \cite{CLmain}. Assumption $(iii)$ ensures that the solution belongs to the physically relevant class of global solutions. It follows from \cite[Theorems 2.6 and 2.12]{CLmain} that this class is large. More precisely, associated to any functions $\mathsf{r}_{1},\mathsf{r}_{2}$ satisfying Assumption $(iii)$ and the conditions stated in \cite[Theorem 2.3]{CLmain}, there exist Schwartz class initial data $u_{0},v_{0}$ such that $r_{j}(k;u_{0},v_{0})=\mathsf{r}_{j}(k)$, $j=1,2$, and such that the initial value problem \eqref{boussinesq}--\eqref{initial data} with $u_{1}:=\partial_{x}v_{0}$ admits a global Schwartz class solution.
\end{enumerate}

\subsection{Statement of main result}
We now introduce the necessary material to present our main result. 
For easy comparison, we use the same notation as in \cite{CLmain}.
Let $\zeta := x/t$ and assume that $\zeta \in (0,\frac{1}{\sqrt{3}})$. For $k \in \mathbb{C}\setminus\{0\}$, define $\Phi_{ij}(\zeta, k)$ for $1 \leq j<i \leq 3$ by
\begin{align}\label{def of Phi ij}
\Phi_{ij}(\zeta,k) = (l_{i}(k)-l_{j}(k))\zeta + (z_{i}(k)-z_{j}(k)).
\end{align}
The function $k\mapsto \Phi_{21}(\zeta,k)$ has four saddle points $\{k_{j}=k_{j}(\zeta)\}_{j=1}^{4}$ given by
\begin{subequations}\label{def of kj}
\begin{align}
& k_{2} = \frac{1}{4}\bigg( \zeta - \sqrt{8+\zeta^{2}} - i \sqrt{2}\sqrt{4-\zeta^{2}+\zeta\sqrt{8+\zeta^{2}}} \bigg), \label{def of k2} \\
& k_{4} = \frac{1}{4}\bigg( \zeta + \sqrt{8+\zeta^{2}} - i\sqrt{2}\sqrt{4-\zeta^{2}-\zeta\sqrt{8+\zeta^{2}}} \bigg), \label{def of k4}
\end{align}
\end{subequations} 
$k_{1} = \overline{k_{2}}$, and $k_{3}=\overline{k_{4}}$. Moreover, $|k_{2}|=|k_{4}|=1$, $\arg k_{2} \in (-\frac{3\pi}{4},-\frac{2\pi}{3})$, and $\arg k_{4} \in (-\frac{\pi}{4},-\frac{\pi}{6})$. Define also $\delta_j(\zeta, k)$, $j = 1, \dots, 5$, by
\begin{align}
& \delta_{1}(\zeta, k) = \exp \bigg\{ \frac{1}{2\pi i} \int_{\omega k_{4}}^{i} \frac{\ln (1+r_{1}(\omega^{2}s)r_{2}(\omega^{2}s))}{s - k} ds \bigg\}, \nonumber \\
& \delta_{2}(\zeta, k) = \exp \bigg\{ \frac{1}{2\pi i} \int_{i}^{\omega^{2}k_{2}} \frac{\ln (1+r_{1}(s)r_{2}(s))}{s - k} ds \bigg\}, \quad \delta_{3}(\zeta, k) = \exp \bigg\{ \frac{1}{2\pi i} \int_{i}^{\omega^{2}k_{2}} \frac{\ln f(s)}{s - k} ds \bigg\}, \nonumber \\
& \delta_{4}(\zeta, k) = \exp \bigg\{ \frac{1}{2\pi i} \int_{\omega^{2}k_{2}}^{\omega} \frac{\ln f(s)}{s - k} ds \bigg\}, \quad \delta_{5}(\zeta, k) = \exp \bigg\{ \frac{1}{2\pi i} \int_{\omega^{2}k_{2}}^{\omega} \frac{\ln f(\omega^{2}s)}{s - k} ds \bigg\}, \label{def of deltaj}
\end{align}
where the paths follow the unit circle in the counterclockwise direction, the principal branch is used for the logarithms, and 
\begin{align}\label{def of f}
f(k) := 1+r_{1}(k)r_{2}(k) + r_{1}(\tfrac{1}{\omega^{2}k})r_{2}(\tfrac{1}{\omega^{2}k}), \qquad k \in \partial \mathbb{D}.
\end{align}
Various properties of $r_{1},r_{2}$, and $f$ have been established in \cite[Theorem 2.3 and Lemma 2.13]{CLmain}. In particular, $r_{1},r_{2},f$ are well-defined on $\partial \D \setminus \hat{\mathcal{Q}}$ and the arguments of all logarithms appearing in the above definitions of $\{\delta_j(\zeta, k)\}_{j=1}^{5}$ are $>0$. Define 
\begin{align*}
& \mathcal{D}_{1}(k) = \frac{\delta_{1}(\omega^{2} k)\delta_{1}(\frac{1}{\omega k})^{2}\delta_{1}(\omega k)}{\delta_{1}(\frac{1}{k})\delta_{1}(\frac{1}{\omega^{2}k})} \frac{\delta_{2}(k)\delta_{2}(\omega^{2}k)\delta_{2}(\frac{1}{k})^{2}}{\delta_{2}(\omega k)^{2} \delta_{2}(\frac{1}{\omega^{2}k})\delta_{2}(\frac{1}{\omega k})} \frac{\delta_{3}(\omega k) \delta_{3}(\omega^{2} k) \delta_{3}(\frac{1}{\omega k})^{2}}{\delta_{3}(k)^{2}\delta_{3}(\frac{1}{k})\delta_{3}(\frac{1}{\omega^{2} k})} \nonumber \\
& \hspace{1.5cm} \times \frac{\delta_{4}(\omega^{2}k)^{2} \delta_{4}(\frac{1}{k})\delta_{4}(\frac{1}{\omega k})}{\delta_{4}(k)\delta_{4}(\omega k)\delta_{4}(\frac{1}{\omega^{2} k})^{2}} \frac{\delta_{5}(\omega k)^{2} \delta_{5}(\frac{1}{\omega k})\delta_{5}(\frac{1}{\omega^{2}k})}{\delta_{5}(k) \delta_{5}(\frac{1}{k})^{2}\delta_{5}(\omega^{2}k)}, \\
& \mathcal{D}_{2}(k) = \frac{\delta_{1}(\omega^{2} k)^{2}  \delta_{1}(\frac{1}{\omega k})\delta_{1}(\frac{1}{\omega^{2} k})}{\delta_{1}(k)\delta_{1}(\frac{1}{k})^{2}\delta_{1}(\omega k)} \frac{\delta_{2}(\frac{1}{k})\delta_{2}(\frac{1}{\omega k})}{\delta_{2}(\omega k) \delta_{2}(\frac{1}{\omega^{2}k})^{2}\delta_{2}(\omega^{2} k)} \frac{\delta_{3}(\omega^{2} k)^{2}  \delta_{3}(\frac{1}{\omega k})\delta_{3}(\frac{1}{\omega^{2} k})}{\delta_{3}(\frac{1}{k})^{2}\delta_{3}(\omega k)} \nonumber \\
& \hspace{1.5cm} \times \frac{\delta_{4}(\omega^{2}k) \delta_{4}(\frac{1}{\omega k})^{2}}{\delta_{4}(\omega k)^{2}\delta_{4}(\frac{1}{\omega^{2} k})\delta_{4}(\frac{1}{k})} \frac{\delta_{5}(\omega k) \delta_{5}(\frac{1}{\omega^{2}k})^{2}\delta_{5}(\omega^{2} k)}{\delta_{5}(\frac{1}{k}) \delta_{5}(\frac{1}{\omega k})}.
\end{align*}
Define $\chi_{j}(\zeta,k)$, $j = 1, 2, 3$, and $\tilde{\chi}_{j}(\zeta,k)$, $j = 4,5$, by
\begin{align}
\chi_{1}(\zeta,k) =&\; \frac{1}{2\pi i} \int_{\omega k_{4}}^{i}  \ln_{s}(k-s) d\ln(1+r_1(\omega^{2}s)r_{2}(\omega^{2}s)), \nonumber \\
\chi_{2}(\zeta,k) =&\; \frac{1}{2\pi i} \int_{i}^{\omega^{2}k_{2}}  \ln_{s}(k-s) d\ln(1+r_1(s)r_{2}(s)), \nonumber \\
\chi_{3}(\zeta,k) =&\; \frac{1}{2\pi i} \int_{i}^{\omega^{2}k_{2}}  \ln_{s}(k-s) d\ln(f(s)), \nonumber \\
\tilde{\chi}_{4}(\zeta,k) =&\;\frac{1}{2\pi i} \dashint_{\omega^{2}k_{2}}^{\omega}  \tilde{\ln}_{s}(k-s) d\ln(f(s)) \nonumber \\
:= &\; \frac{1}{2\pi i}\lim_{\epsilon \to 0_{+}} \bigg(  \int_{\omega^{2}k_{2}}^{e^{i(\frac{2\pi}{3}-\epsilon)}} \tilde{\ln}_{s}(k-s) d\ln(f(s)) - \tilde{\ln}_{\omega}(k-\omega)\ln(f(e^{i(\frac{2\pi}{3}-\epsilon)})) \bigg), \nonumber \\
\tilde{\chi}_{5}(\zeta,k) = &\; \frac{1}{2\pi i} \dashint_{\omega^{2}k_{2}}^{\omega} \tilde{\ln}_{s}(k-s) d\ln(f(\omega^{2}s)) \nonumber \\
:= &\; \frac{1}{2\pi i}\lim_{\epsilon \to 0_{+}} \bigg(  \int_{\omega^{2}k_{2}}^{e^{i(\frac{2\pi}{3}-\epsilon)}} \tilde{\ln}_{s}(k-s) d\ln(f(\omega^{2}s)) - \tilde{\ln}_{\omega}(k-\omega)\ln(f(e^{-i\epsilon})) \bigg), \label{def of chij}
\end{align}
where the paths follow the unit circle in the counterclockwise direction. For $s \in \{e^{i\theta}\, | \, \theta \in [\frac{\pi}{3},\frac{2\pi}{3}]\}$, $k \mapsto \ln_{s}(k-s)=\ln |k-s|+i \arg_{s}(k-s)$ has a cut along $\{e^{i \theta} \, | \,  \theta \in [\frac{\pi}{2},\arg s]\}\cup(i,i\infty)$ if $\arg s> \frac{\pi}{2}$, and a cut along $\{e^{i \theta} \, | \,  \theta \in [\arg s,\frac{\pi}{2}]\}\cup(i,i\infty)$ if $\arg s< \frac{\pi}{2}$, and the branch is such that $\arg_{s}(1)=2\pi$. Also, for $s \in \{e^{i\theta}: \theta \in [\frac{\pi}{2},\frac{2\pi}{3}]\}$, $k \mapsto \tilde{\ln}_{s}(k-s):=\ln |k-s|+i \tilde{\arg}_{s}(k-s)$ has a cut along $\{e^{i \theta}: \theta \in [\arg s,\pi]\}\cup(-\infty,0)$, and satisfies $\tilde{\arg}_{s}(1)=0$. In the definitions of $\tilde{\chi}_{4},\tilde{\chi}_{5}$, a regularized integral is needed because $f(1)=f(\omega)=0$ (see \cite[Lemma 2.13 $(i)$]{CLmain}). Let the functions $\{\nu_{j}\}_{j=1}^{4}$ and $\hat{\nu}_2$ be given by
\begin{align*}\nonumber
& \nu_1(k) = - \frac{1}{2\pi}\ln(1+r_{1}(\omega k)r_{2}(\omega k)), 
\qquad \nu_2(k) = - \frac{1}{2\pi}\ln(1+r_{1}(\omega^{2} k)r_{2}(\omega^{2} k)), 
	\\ 
& \nu_3(k) = - \frac{1}{2\pi}\ln(f(\omega k)),  \qquad
\nu_4(k) = - \frac{1}{2\pi}\ln(f(\omega^{2} k)), \qquad \hat{\nu}_2(k) = \nu_2(k) + \nu_3(k) - \nu_4(k).
\end{align*}
The following functions $\tilde{d}_{1,0}$ and $\tilde{d}_{2,0}$ appear in our final result:\begin{align}
& \tilde{d}_{1,0}(\zeta,t) = e^{-4\pi\nu_{1}(\omega^{2}k_{4})}e^{2\chi_{1}(\zeta,\omega k_{4})} e^{-2i \nu_{3}(\omega^{2}i) \ln_{i}(\omega k_{4}-i)} t^{-i \nu_{1}(\omega^{2}k_{4})} z_{1,\star}^{-2i\nu_{1}(\omega^{2}k_{4})} \mathcal{D}_{1}(\omega k_{4}), \label{def of dt10} \\
& \tilde{d}_{2,0}(\zeta,t) = e^{-2\chi_{2}(\zeta,\omega^{2} k_{2}) + \chi_{3}(\zeta,\omega^{2} k_{2}) - \tilde{\chi}_{4}(\zeta,\omega^{2} k_{2}) + 2\tilde{\chi}_{5}(\zeta,\omega^{2} k_{2}) } \nonumber \\
& \hspace{1.6cm} \times e^{i\nu_{3}(\omega^{2}i)\ln_{i}(\omega^{2} k_{2}-i)} t^{i (\nu_{4}(k_{2})-\nu_{3}(k_{2})-\nu_{2}(k_{2}))} z_{2,\star}^{2i(\nu_{4}(k_{2})-\nu_{3}(k_{2})-\nu_{2}(k_{2}))} \mathcal{D}_{2}(\omega^{2} k_{2}), \label{def of dt20}
\end{align}
where $z_{1,\star}$, $z_{2,\star}$ are given by
\begin{align*}
& z_{1,\star} = z_{1,\star}(\zeta) := \sqrt{2}e^{\frac{\pi i}{4}} \sqrt{\omega \frac{4-3k_{4} \zeta - k_{4}^{3} \zeta}{4k_{4}^{4}}}, \qquad -i\omega k_{4}z_{1,\star}>0, \\
& z_{2,\star} = z_{2,\star}(\zeta) := \sqrt{2}e^{\frac{\pi i}{4}} \sqrt{-\omega^{2} \frac{4-3k_{2} \zeta - k_{2}^{3} \zeta}{4k_{2}^{4}}}, \qquad -i\omega^{2} k_{2}z_{2,\star}>0,
\end{align*}
and the branches for the complex powers $z_{j,\star}^{\alpha} = e^{\alpha \ln z_{j,\star}}$ are fixed by
\begin{align*}
& \ln z_{1,\star} = \ln |z_{1,\star}| + i \arg z_{1,\star} = \ln |z_{1,\star}| + i \big( \tfrac{\pi}{2}-\arg (\omega k_{4}) \big), & &  \arg (\omega k_{4}) \in (\tfrac{\pi}{3},\tfrac{\pi}{2}),
	\\
& \ln z_{2,\star} = \ln |z_{2,\star}| + i \arg z_{2,\star} = \ln |z_{2,\star}| + i \big( \tfrac{\pi}{2}-\arg (\omega^{2} k_2) \big), & &  \arg (\omega^{2} k_2) \in (\tfrac{\pi}{2},\tfrac{2\pi}{3}).
\end{align*}
Finally, define also
\begin{align*}
& \tilde{q}_1 = |\tilde{r}(k_{4})|^{\frac{1}{2}}r_{1}(k_{4}), & & q_{2} = \tilde{r}(\omega^{2}k_{2})^{\frac{1}{2}}r_{1}(\omega^{2}k_{2}), & & q_{5} = |\tilde{r}(\omega k_{2})|^{\frac{1}{2}}r_{1}(\omega k_{2}), & & q_{6} = |\tilde{r}(\tfrac{1}{k_{2}})|^{\frac{1}{2}}r_{1}(\tfrac{1}{k_{2}}),
\end{align*}
where
\begin{align}\label{def of tilde r}
\tilde{r}(k):=\frac{\omega^{2}-k^{2}}{1-\omega^{2}k^{2}}, \qquad k \in \mathbb{C}\setminus \{\omega^{2},-\omega^{2}\}.
\end{align}

We now state our main result, which establishes the long-time behavior of $u(x,t)$ in the sector $\zeta \in (0,1/\sqrt{3})$. The statement involves the Gamma function $\Gamma(k)$, as well as the square roots of $\hat{\nu}_{2}(k_{2})$ and $\nu_{1}(\omega^{2}k_{4})$. These square roots are well-defined and $\geq 0$ thanks to the inequalities $\hat{\nu}_{2}(k_{2}) \geq 0$ and $\nu_{1}(\omega^{2}k_{4})\geq 0$ established in Lemma \ref{lemma: nuhat lemma IIbis} below.

\begin{theorem}\label{asymptoticsth}
Let $u_0,u_1 \in \mathcal{S}(\R)$ be real-valued and such that $\int_{\mathbb{R}}u_{1} dx =0$. Let $v_{0}(x) = \int_{-\infty}^{x}u_{1}(x')dx'$ and suppose that $u_{0},v_{0}$ are such that Assumptions (i), (ii), and (iii) are fulfilled. Let $\mathcal{I}$ be a fixed compact subset of $(0,\smash{\frac{1}{\sqrt{3}}})$. Then the global solution $u(x,t)$ of the initial value problem for (\ref{boussinesq}) with initial data $u_0, u_1$ enjoys the following asymptotics as $t \to \infty$:
\begin{align}\label{asymp for u sector V}
& u(x,t) =
 \frac{\tilde{A}_{1}(\zeta)}{\sqrt{t}}   \cos \tilde{\alpha}_{1}(\zeta,t) +\frac{A_{2}(\zeta)}{\sqrt{t}} \cos \tilde{\alpha}_{2}(\zeta,t) + O\bigg(\frac{\ln t}{t}\bigg), 
\end{align}
uniformly for $\zeta:= \frac{x}{t} \in \mathcal{I}$, where 
\begin{align}\nonumber
& \tilde{A}_{1}(\zeta) := \frac{4\sqrt{3}\sqrt{\nu_{1}(\omega^2 k_4)}\,\im k_{4}}{-i\omega k_{4} z_{1,\star}|\tilde{r}(\frac{1}{k_{4}})|^{\frac{1}{2}}}\sin(\arg(\omega k_{4})),  
	\\ \nonumber
& A_{2}(\zeta) := \frac{-4\sqrt{3}\sqrt{\hat{\nu}_2(k_2)}|\tilde{r}(\frac{1}{k_{2}})|^{\frac{1}{2}} \im k_{2}}{-i\omega^{2} k_{2} z_{2,\star}}\sin(\arg(\omega^{2} k_{2})), 
	\\ \nonumber 
& \tilde{\alpha}_1(\zeta, t) := \frac{3\pi}{4}-\arg \tilde{q}_{1}+\arg\Gamma(i\nu_{1}(\omega^2 k_4))+\arg \tilde{d}_{1,0}-t\, \im \Phi_{31}(\zeta,\omega k_{4}),
	\\ \nonumber
& \tilde{\alpha}_{2}(\zeta,t) := \frac{3\pi}{4}-\arg (q_{6}-q_{2}q_{5})+\arg\Gamma(i\hat{\nu}_2(k_2))+\arg \tilde{d}_{2,0}-t \, \im \Phi_{32}(\zeta,\omega^{2} k_{2}).
\end{align}
\end{theorem}

\subsection{Notation}\label{notationsubsec}
We use the following notation throughout the paper.

\begin{enumerate}[$-$]
\item $C>0$ and $c>0$ denote generic constants that may change within a computation.

\item $[A]_1$, $[A]_2$, and $[A]_3$ denote the first, second, and third columns of a $3 \times 3$ matrix $A$.

\item If $A$ is an $n \times m$ matrix, we define $|A| \ge 0$ by $|A|^2=\Sigma_{i,j}|A_{ij}|^2$. 
For a piecewise smooth contour $\gamma \subset \C$ and $1 \le p \le \infty$, if $|A|$ belongs to $L^p(\gamma)$, we write $A \in L^p(\gamma)$ and define $\|A\|_{L^p(\gamma)} := \| |A|\|_{L^p(\gamma)}$.

\item $\D = \{k \in \C \, | \, |k| < 1\}$ denotes the open unit disk and $\partial \D = \{k \in \C \, | \, |k| = 1\}$ denotes the unit circle. 

\item $f^*(k):= \overline{f(\bar{k})}$ denotes the Schwarz conjugate of a function $f(k)$.

\item $D_\epsilon(k)$ denotes the open disk of radius $\epsilon$ centered at a point $k \in \C$.

\item $\mathcal{S}(\R)$ denotes the Schwartz space of all smooth functions $f$ on $\R$ such that $f$ and all its derivatives have rapid decay as $x \to \pm \infty$. 

\item $\kappa_{j} = e^{\frac{\pi i(j-1)}{3}}$, $j=1,\ldots,6$, denote the sixth roots of unity, see Figure \ref{fig: Dn}.

\item We let $\mathcal{Q} := \{\kappa_{j}\}_{j=1}^{6}$ and $\hat{\mathcal{Q}} := \mathcal{Q} \cup \{0\}$.

\item $D_n$, $n = 1, \dots, 6$, denote the open subsets of the complex plane shown in Figure \ref{fig: Dn}.

\item $\Gamma = \cup_{j=1}^9 \Gamma_j$ denotes the contour shown and oriented as in Figure \ref{fig: Dn}.

\item $\hat{\Gamma}_{j} = \Gamma_{j} \cup \partial \D$ denotes the union of $\Gamma_j$ and the unit circle. 

\end{enumerate}

\section{The RH problem for $n$}\label{overviewsec}

The proof of Theorem \ref{asymptoticsth} is based on a careful analysis of a RH problem derived in \cite{CLmain}. We now recall this RH problem, whose solution is denoted by $n$. Let $\theta_{ij}(x,t,k) = t \, \Phi_{ij}(\zeta,k)$. For $j = 1, \dots, 6$, let us write $\Gamma_j = \Gamma_{j'} \cup \Gamma_{j''}$, where $\Gamma_{j'} = \Gamma_j \setminus \D$ and $\Gamma_{j''} := \Gamma_j \setminus \Gamma_{j'}$ with $\Gamma_j$ as in Figure \ref{fig: Dn}. Because $r_{1}(k)=0$ for all $k \in [0,i]$ (by Assumption $(iii)$), the jump matrix $v(x,t,k)$ is given for $k \in \Gamma$ by
\begin{align}
& v_{1'} = \begin{pmatrix}
1 & -r_{1}(k)e^{-\theta_{21}} & 0 \\
0 & 1 & 0 \\
0 & 0 & 1
\end{pmatrix}, \; v_{1''} = \begin{pmatrix}
1 & 0 & 0 \\
r_{1}(\frac{1}{k})e^{\theta_{21}} & 1 & 0 \\
0 & 0 & 1
\end{pmatrix}, \; v_{2'} = \begin{pmatrix}
1 & 0 & 0 \\
0 & 1 & -r_{2}(\frac{1}{\omega k})e^{-\theta_{32}} \\
0 & 0 & 1
\end{pmatrix}, \nonumber \\
& v_{2''} = \begin{pmatrix}
1 & 0 & 0 \\
0 & 1 & 0 \\
0 & r_{2}(\omega k)e^{\theta_{32}} & 1
\end{pmatrix}, \;  v_{3'} = \begin{pmatrix}
1 & 0 & 0 \\
0 & 1 & 0 \\
-r_{1}(\omega^{2}k)e^{\theta_{31}} & 0 & 1
\end{pmatrix}, \; v_{3''} = \begin{pmatrix}
1 & 0 & r_{1}(\frac{1}{\omega^{2}k})e^{-\theta_{31}} \\
0 & 1 & 0 \\
0 & 0 & 1
\end{pmatrix}, \nonumber \\
& v_{4'} = \begin{pmatrix}
1 & -r_{2}(\frac{1}{k})e^{-\theta_{21}} & 0 \\
0 & 1 & 0 \\
0 & 0 & 1
\end{pmatrix}, \; v_{4''} = \begin{pmatrix}
1 & 0 & 0 \\
r_{2}(k)e^{\theta_{21}} & 1 & 0 \\
0 & 0 & 1
\end{pmatrix}, \; v_{5'} = \begin{pmatrix}
1 & 0 & 0 \\
0 & 1 & -r_{1}(\omega k)e^{-\theta_{32}} \\
0 & 0 & 1
\end{pmatrix}, \nonumber \\
& v_{5''} = \begin{pmatrix}
1 & 0 & 0 \\
0 & 1 & 0 \\
0 & r_{1}(\frac{1}{\omega k})e^{\theta_{32}} & 1
\end{pmatrix}, \;  v_{6'} = \begin{pmatrix}
1 & 0 & 0 \\
0 & 1 & 0 \\
-r_{2}(\frac{1}{\omega^{2} k})e^{\theta_{31}} & 0 & 1
\end{pmatrix}, \; v_{6''} = \begin{pmatrix}
1 & 0 & r_{2}(\omega^{2}k)e^{-\theta_{31}} \\
0 & 1 & 0 \\
0 & 0 & 1
\end{pmatrix}, \nonumber \\
& v_{7} = \begin{pmatrix}
1 & -r_{1}(k)e^{-\theta_{21}} & r_{2}(\omega^{2}k)e^{-\theta_{31}} \\
-r_{2}(k)e^{\theta_{21}} & 1+r_{1}(k)r_{2}(k) & \big(r_{2}(\frac{1}{\omega k})-r_{2}(k)r_{2}(\omega^{2}k)\big)e^{-\theta_{32}} \\
r_{1}(\omega^{2}k)e^{\theta_{31}} & \big(r_{1}(\frac{1}{\omega k})-r_{1}(k)r_{1}(\omega^{2}k)\big)e^{\theta_{32}} & f(\omega^{2}k)
\end{pmatrix}, \nonumber \\
& v_{8} = \begin{pmatrix}
f(k) & r_{1}(k)e^{-\theta_{21}} & \big(r_{1}(\frac{1}{\omega^{2} k})-r_{1}(k)r_{1}(\omega k)\big)e^{-\theta_{31}} \\
r_{2}(k)e^{\theta_{21}} & 1 & -r_{1}(\omega k) e^{-\theta_{32}} \\
\big( r_{2}(\frac{1}{\omega^{2}k})-r_{2}(\omega k)r_{2}(k) \big)e^{\theta_{31}} & -r_{2}(\omega k) e^{\theta_{32}} & 1+r_{1}(\omega k)r_{2}(\omega k)
\end{pmatrix}, \nonumber \\
& v_{9} = \begin{pmatrix}
1+r_{1}(\omega^{2}k)r_{2}(\omega^{2}k) & \big( r_{2}(\frac{1}{k})-r_{2}(\omega k)r_{2}(\omega^{2} k) \big)e^{-\theta_{21}} & -r_{2}(\omega^{2}k)e^{-\theta_{31}} \\
\big(r_{1}(\frac{1}{k})-r_{1}(\omega k) r_{1}(\omega^{2} k)\big)e^{\theta_{21}} & f(\omega k) & r_{1}(\omega k)e^{-\theta_{32}} \\
-r_{1}(\omega^{2}k)e^{\theta_{31}} & r_{2}(\omega k) e^{\theta_{32}} & 1
\end{pmatrix}, \label{vdef}
\end{align}
where $v_j, v_{j'}, v_{j''}$ denote the restrictions of $v$ to $\Gamma_{j}$, $\Gamma_{j'}$, and $\Gamma_{j''}$, respectively. Let $\Gamma_{\star} = \{i\kappa_j\}_{j=1}^6 \cup \{0\}$ denote the set of intersection points of $\Gamma$, and define
\begin{align}\label{def of Acal and Bcal}
\mathcal{A} := \begin{pmatrix}
0 & 0 & 1 \\
1 & 0 & 0 \\
0 & 1 & 0
\end{pmatrix} \qquad \mbox{ and } \qquad \mathcal{B} := \begin{pmatrix}
0 & 1 & 0 \\
1 & 0 & 0 \\
0 & 0 & 1
\end{pmatrix}.
\end{align}
The RH problem for $n$ is as follows.
\begin{RHproblem}[RH problem for $n$]\label{RHn}
Find a $1 \times 3$-row-vector valued function $n(x,t,k)$ with the following properties:
\begin{enumerate}[$(a)$]
\item\label{RHnitema} $n(x,t,\cdot) : \C \setminus \Gamma \to \mathbb{C}^{1 \times 3}$ is analytic.

\item\label{RHnitemb} The limits of $n(x,t,k)$ as $k$ approaches $\Gamma \setminus \Gamma_\star$ from the left and right exist, are continuous on $\Gamma \setminus \Gamma_\star$, and are denoted by $n_+$ and $n_-$, respectively. Furthermore, 
\begin{align}\label{njump}
  n_+(x,t,k) = n_-(x, t, k) v(x, t, k) \qquad \text{for} \quad k \in \Gamma \setminus \Gamma_\star.
\end{align}

\item\label{RHnitemc} $n(x,t,k) = O(1)$ as $k \to k_{\star} \in \Gamma_\star$.

\item\label{RHnitemd} For $k \in \C \setminus \Gamma$, $n$ obeys the symmetries
\begin{align}\label{nsymm}
n(x,t,k) = n(x,t,\omega k)\mathcal{A}^{-1} = n(x,t,k^{-1}) \mathcal{B}.
\end{align}

\item\label{RHniteme} $n(x,t,k) = (1,1,1) + O(k^{-1})$ as $k \to \infty$.
\end{enumerate}
\end{RHproblem}
Recall that $u_{0},v_{0}$ are such that Assumptions $(i)$, $(ii)$, and $(iii)$ hold, and that $u_{1}(x):=v_{0}'(x)$. Hence, it follows from \cite[Theorems 2.6 and 2.12]{CLmain} that RH problem \ref{RHn} has a unique solution $n(x,t,k)$ for each $(x,t) \in \R \times [0,\infty)$, that
$$n_{3}^{(1)}(x,t) := \lim_{k\to \infty} k (n_{3}(x,t,k) -1)$$ 
is well-defined and smooth for $(x,t) \in \R \times [0,\infty)$, and that $u(x,t)$ defined by
\begin{align}\label{recoveruvn}
u(x,t) = -i\sqrt{3}\frac{\partial}{\partial x}n_{3}^{(1)}(x,t)
\end{align}
is a Schwartz class solution of (\ref{boussinesq}) on $\R \times [0,\infty)$ with initial data $u_{0},u_{1}$.

Note that $v$ in \eqref{vdef} obeys the symmetries
\begin{align}\label{vsymm}
v(x,t,k) = \mathcal{A} v(x,t,\omega k)\mathcal{A}^{-1}
 = \mathcal{B} v(x,t, k^{-1})^{-1}\mathcal{B}, \qquad k \in \Gamma,
\end{align}
and that $v$ depends on $x,t$ only via the phase functions $\Phi_{21}$, $\Phi_{31}$, and $\Phi_{32}$. As mentioned in Section \ref{section:main results}, the saddle points of $\Phi_{21}$ are given by $\{k_{j}\}_{j=1}^{4}$. From the relations
\begin{align}\label{relations between the different Phi}
\Phi_{31}(\zeta,k) = - \Phi_{21}(\zeta,\omega^{2}k), \qquad \Phi_{32}(\zeta, k) = \Phi_{21}(\zeta, \omega k),
\end{align}
we then conclude that $\{\omega k_{j}\}_{j=1}^{4}$ are the saddle points of $\Phi_{31}$ and that $\{\omega^{2} k_{j}\}_{j=1}^{4}$ are the saddle points of $\Phi_{32}$. 
The signature tables for $\Phi_{21}$, $\Phi_{31}$, and $\Phi_{32}$ are shown in Figure \ref{IIbis fig: Re Phi 21 31 and 32 for zeta=0.7}. 
\begin{figure}[h]
\begin{center}
\begin{tikzpicture}[master]
\node at (0,0) {\includegraphics[width=4.5cm]{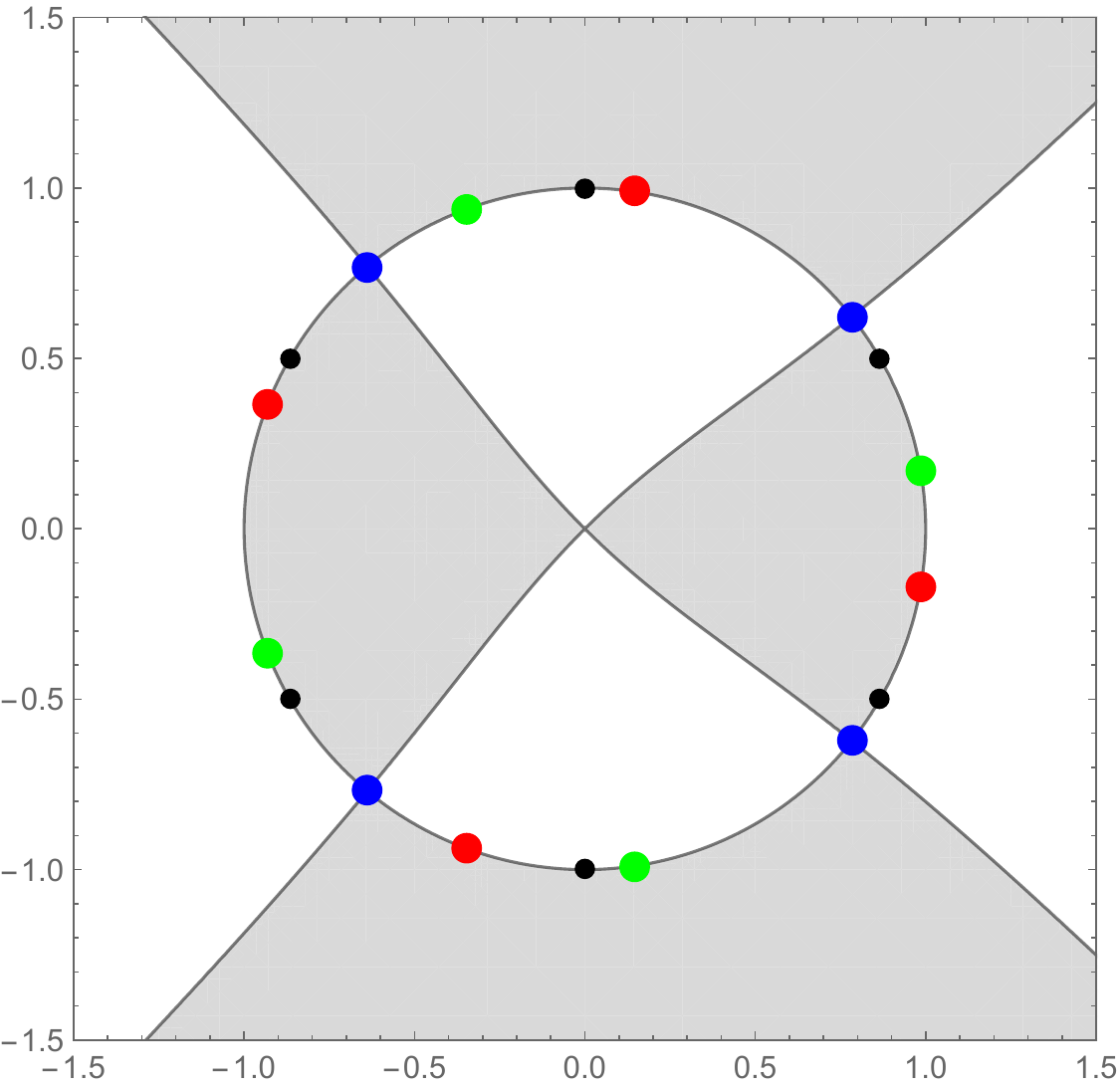}};

\node at (-1,1.18) {\tiny $k_1$};
\node at (-1.03,-.97) {\tiny $k_2$};
\node at (1.2,1.17) {\tiny $k_3$};
\node at (1.2,-1.03) {\tiny $k_4$};

\end{tikzpicture} \hspace{0.1cm} 
\begin{tikzpicture}[slave]
\node at (0,0) {\includegraphics[width=4.5cm]{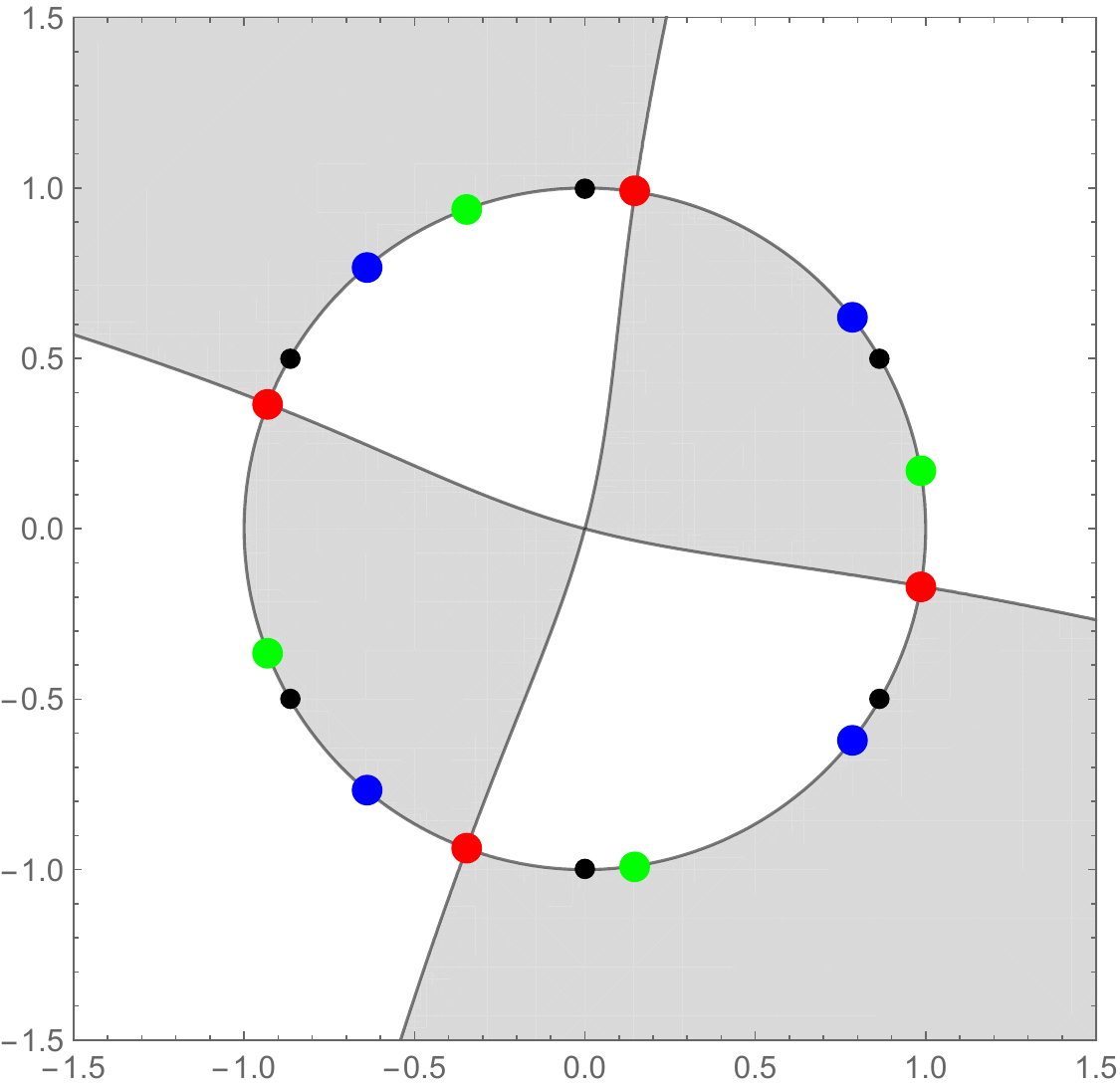}};

\node at (-.7,-1.3) {\tiny $\omega k_1$};
\node at (1.8,-0.02) {\tiny $\omega k_2$};
\node at (-1.5,0.5) {\tiny $\omega k_3$};
\node at (.65,1.53) {\tiny $\omega k_4$};

\end{tikzpicture} \hspace{0.1cm} 
\begin{tikzpicture}[slave]
\node at (0,0) {\includegraphics[width=4.5cm]{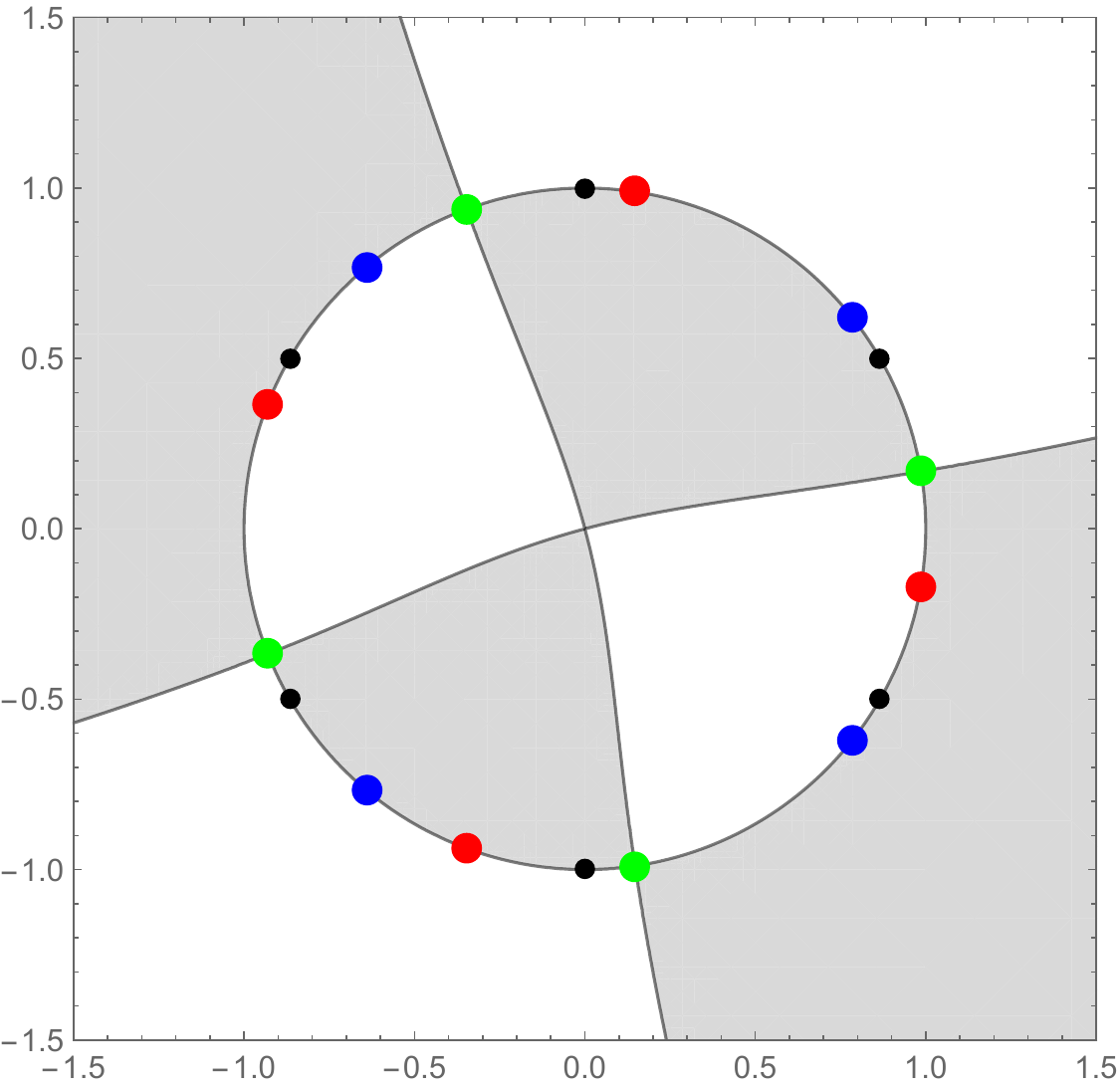}};

\node at (1.85,.16) {\tiny $\omega^2 k_1$};
\node at (-.8,1.5) {\tiny $\omega^2 k_2$};
\node at (0.7,-1.43) {\tiny $\omega^2 k_3$};
\node at (-1.6,-.33) {\tiny $\omega^2 k_4$};

\end{tikzpicture}
\end{center}
\begin{figuretext}
\label{IIbis fig: Re Phi 21 31 and 32 for zeta=0.7} From left to right: The signature tables for $\Phi_{21}$, $\Phi_{31}$, and $\Phi_{32}$ for $\zeta=\frac{1}{2\sqrt{3}}$. The grey regions correspond to $\{k \,|\, \re \Phi_{ij}>0\}$ and the white regions to $\{k \,|\, \re \Phi_{ij}<0\}$. The saddle points $k_{1},k_{2}, k_{3}, k_{4}$ of $\Phi_{21}$ are blue, the saddle points $\omega k_{1},\omega k_{2},\omega k_{3},\omega k_{4}$ of $\Phi_{31}$ are red, and the saddle points $\omega^{2} k_{1},\omega^{2} k_{2},\omega^{2} k_{3},\omega^{2} k_{4}$ of $\Phi_{32}$ are green. The black dots are the points $i \kappa_j$, $j=1,\ldots,6$.
\end{figuretext}
\end{figure}

Our proof employs the Deift--Zhou \cite{DZ1993} steepest descent method and involves a series of transformations $n \mapsto n^{(1)} \mapsto n^{(2)} \mapsto n^{(3)} \mapsto \hat{n}$. These transformations are invertible, meaning that the RH problems satisfied by $n^{(1)}, n^{(2)}, n^{(3)}, \hat{n}$ are equivalent to the original RH problem \ref{RHn}. In the sector $\zeta \in \mathcal{I}$ considered here\footnote{Recall that $\mathcal{I}$ denotes an arbitrary compact subset of $(0,\frac{1}{\sqrt{3}})$.}, it turns out that the main contribution to the long-time asymptotics of $u(x,t)$ come from a global parametrix $\Delta$ and from twelve local parametrices near the saddle points $\{k_{j},\omega k_{j},\omega^{2}k_{j}\}_{j=1}^{4}$. 

The jump contours and jump matrices for the RH problems for $n^{(1)}, n^{(2)}, n^{(3)}, \hat{n}$ will be denoted by $\Gamma^{(1)}, \Gamma^{(2)}, \Gamma^{(3)}, \hat{\Gamma}$ and $v^{(1)},v^{(2)},v^{(3)},\hat{v}$, respectively. The symmetries (\ref{nsymm}) and (\ref{vsymm}) will be maintained after each transformation, so that, for $j = 1, 2,3$,
\begin{align}\label{vjsymm}
& v^{(j)}(x,t,k) = \mathcal{A} v^{(j)}(x,t,\omega k)\mathcal{A}^{-1}
 = \mathcal{B} v^{(j)}(x,t,k^{-1})^{-1}\mathcal{B}, & & k \in \Gamma^{(j)},
	\\ \label{mjsymm}
& n^{(j)}(x,t, k) = n^{(j)}(x,t,\omega k)\mathcal{A}^{-1}
 = n^{(j)}(x,t, k^{-1}) \mathcal{B}, & & k \in \C \setminus \Gamma^{(j)},
\end{align}
and similarly for $\hat{n}$ and $\hat{v}$. These symmetries will save us some efforts, because (i) we will only need to construct explicitly two local parametrices near $\omega k_{4}$ and $\omega^{2}k_{2}$, and (ii) we will only have to define the transformations $n^{(j)} \mapsto n^{(j+1)}$ and $n^{(3)}\mapsto \hat{n}$ in the sector $\mathsf{S}:=\{k \in \mathbb{C}|  \arg k \in [\frac{\pi}{3},\frac{2\pi}{3}]\}$. In the rest of the complex plane, the local parametrices and the transformations will be defined using the $\mathcal{A}$- and $\mathcal{B}$-symmetries.

The analysis of the asymptotic sector $\zeta \in \mathcal{I}$ considered in this paper involves two main difficulties (which are not present for $\zeta>1$): (i) the opening of the lenses can only be implemented by introducing poles in the jump matrices, which complicates considerably the constructions of the global and local parametrices, and (ii) the model RH problem used for the local parametrix near $\omega^{2}k_{2}$, albeit using parabolic cylinder functions as in earlier works such as \cite{I1981}, involves jump matrices of a much more complicated nature than in the previous literature (see Lemma \ref{II Xlemma 3 green}).

\section{The $n \to n^{(1)}$ transformation}

The following properties of $r_{1}$ and $r_{2}$, which were proved in \cite[Theorem 2.3]{CLmain}, will be used throughout the steepest descent analysis: $r_1 \in C^\infty(\hat{\Gamma}_{1})$, $r_2 \in C^\infty(\hat{\Gamma}_{4}\setminus \{\omega^{2}, -\omega^{2}\})$, $r_{1}(\kappa_{j})\neq 0$ for $j=1,\ldots,6$, $r_{2}(k)$ has simple poles at $k=\pm \omega^2$, and simple zeros at $k=\pm\omega$, and $r_{1},r_{2}$ are rapidly decreasing as $|k|\to \infty$. Furthermore,
\begin{align}
& r_{1}(\tfrac{1}{\omega k}) + r_{2}(\omega k) + r_{1}(\omega^{2} k) r_{2}(\tfrac{1}{k}) = 0, & & k \in \partial \mathbb{D}\setminus \{\omega, -\omega\}, \label{r1r2 relation on the unit circle} \\
& r_{2}(k) = \tilde{r}(k)\overline{r_{1}(\bar{k}^{-1})}, \qquad \tilde{r}(k) :=\frac{\omega^{2}-k^{2}}{1-\omega^{2}k^{2}}, & & k \in \hat{\Gamma}_{4}\setminus \{0, \omega^{2}, -\omega^{2}\}, \label{r1r2 relation with kbar symmetry} \\
& r_{1}(1) = r_{1}(-1) = 1, \qquad r_{2}(1) = r_{2}(-1) = -1. \label{r1r2at0}
\end{align}
The first transformation $n \to n^{(1)}$ consists in opening lenses around $\partial \mathbb{D}\setminus \mathcal{Q}$. For $j=1,2$, define
$$\hat{r}_{j}(k) := \frac{r_{j}(k)}{1+r_{1}(k)r_{2}(k)}$$
and let
\begin{align*}
& r_{j,1}(k) := r_{j}(k), \quad r_{j,2}(k) := \hat{r}_{j}(k), \quad r_{j,3}(k) := \frac{r_{j}(k)-r_{j}(\frac{1}{\omega k})r_{j}(\frac{1}{\omega^{2} k})}{1+r_{1}(\frac{1}{\omega k})r_{2}(\frac{1}{\omega k})}, \\
& r_{j,4}(k) := \frac{r_{j}(k)}{f(k)}, \quad r_{j,5}(k) = \frac{r_{j}(k)-r_{j}(\frac{1}{\omega k})r_{j}(\frac{1}{\omega^{2}k})}{f(k)}.
\end{align*}
Since the above spectral functions are (in general) not defined in a neighborhood of $\partial \mathbb{D}$, we need to decompose them into  an analytic part and a small remainder. Let $M>1$ and let $\{U_{1}^{\ell}=U_{1}^{\ell}(\zeta,M)\}_{\ell=1}^{5}$ be the open sets defined by (see Figure \ref{IIbis fig: U1 and U2})
\begin{align*}
U_{1}^{1} & = \{k| \re \Phi_{21}>0 \} \cap \big(\{k| \arg k \in (\tfrac{5\pi}{6},\pi] \cup [-\pi,\arg(\omega^{2}k_{4})), \; M^{-1}<|k|<1 \} \\
& \quad \; \cup \{k| \arg k \in(-\tfrac{\pi}{2},\arg k_{4}), \; 1 < |k| < M\} \big), \\
U_{1}^{2} & = \{k| \re \Phi_{21}>0 \} \cap \big( \{k|\arg k \in(\arg k_{4},-\tfrac{\pi}{6}), \; M^{-1}<|k|<1\} \\
& \quad \; \cup \{k|\arg k \in (\tfrac{\pi}{2},\arg(\omega^{2}k_{2})), \; 1 < |k| < M\} \big), \\
U_{1}^{3} & = \{k| \re \Phi_{21}>0 \} \cap \{k| \arg k \in (\arg k_{1}, \tfrac{5\pi}{6}), \; M^{-1}<|k|<1\}, \\
U_{1}^{4} & = \{k| \re \Phi_{21}>0 \} \cap \{k| \arg k \in (\arg(\omega k_{2}),\tfrac{\pi}{6}), \; M^{-1}<|k|<1\}, \\
U_{1}^{5} & = \{k| \re \Phi_{21}>0 \} \cap \{k| \arg k \in (\arg(\omega^{2}k_{2}),\arg k_{1}), \; 1<|k|<M)\}.
\end{align*}
It is known from \cite[Lemma 2.13]{CLmain} that $1+r_{1}(k)r_{2}(k)>0$ for all $k \in \partial \mathbb{D}$ with $\arg k \in (\pi/3,\pi)\cup (-2\pi/3,0)$, and that $f(k)=0$ if and only if $k \in \{\pm 1, \pm \omega\}$. Hence, for each $\ell \in \{1,2,3\}$, $r_{1,\ell}$ is well-defined for $k \in \partial U_{1}^{\ell} \cap \partial \mathbb{D}$; $r_{1,4}$ is well-defined for $k \in (\partial U_{1}^{4} \cap \partial \mathbb{D})\setminus\{1\}$; and $r_{1,5}$ is well-defined for $k \in (\partial U_{1}^{5} \cap \partial \mathbb{D})\setminus\{\omega\}$. Let $n_{1}\geq 0$ be the smallest integer such that $(k-1)^{n_{1}}r_{1,4}(k)$ remains bounded as $k \to 1$, $k \in \partial \mathbb{D}$, and let $n_{\omega} \geq 0$ be the smallest integer such that $(k-\omega)^{n_{\omega}}r_{1,5}(k)$ remains bounded as $k \to \omega$, $k \in \partial \mathbb{D}$.

The next lemma establishes decompositions of $\{r_{1,\ell}\}_{\ell=1}^{5}$; we will then obtain decompositions of $\{r_{2,\ell}\}_{\ell=1}^{5}$ using the symmetry \eqref{r1r2 relation with kbar symmetry}. Let $N \geq 1$ be an integer.

\begin{figure}
\begin{center}
\begin{tikzpicture}[master]
\node at (0,0) {\includegraphics[width=4.5cm]{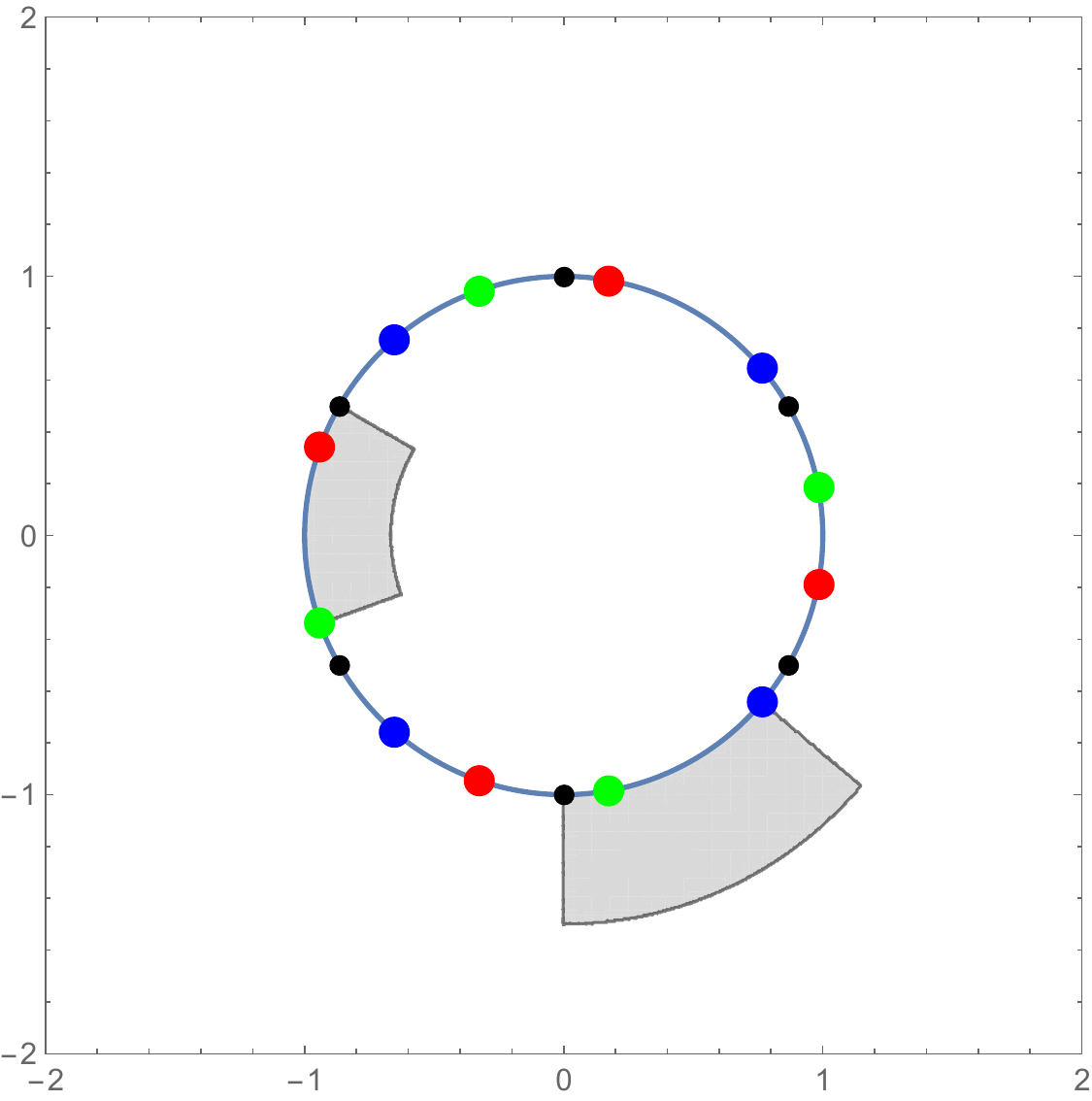}};
\node at (-0.8,0.15) {\tiny $U_{1}^{1}$};
\node at (0.6,-1.15) {\tiny $U_{1}^{1}$};
\end{tikzpicture} \hspace{0.5cm} \begin{tikzpicture}[slave]
\node at (0,0) {\includegraphics[width=4.5cm]{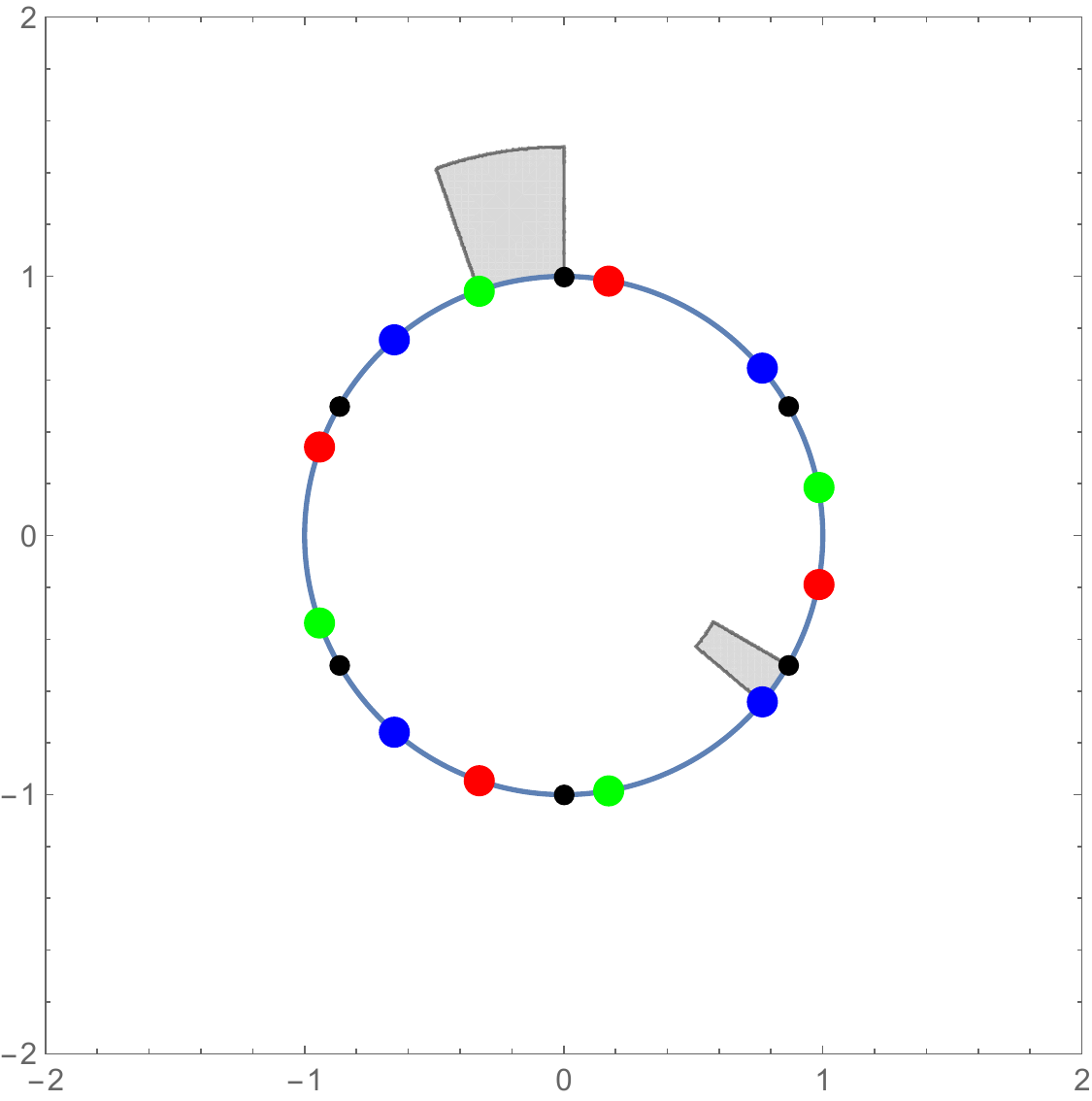}};
\node at (-0.15,1.35) {\tiny $U_{1}^{2}$};
\node at (0.8,-1.5) {\tiny $U_{1}^{2}$};
\draw[black,line width=0.15 mm,->-=1] (0.75,-1.35)--(0.75,-0.45);
\end{tikzpicture} \\
\begin{tikzpicture}[slave]
\node at (0,0) {\includegraphics[width=4.5cm]{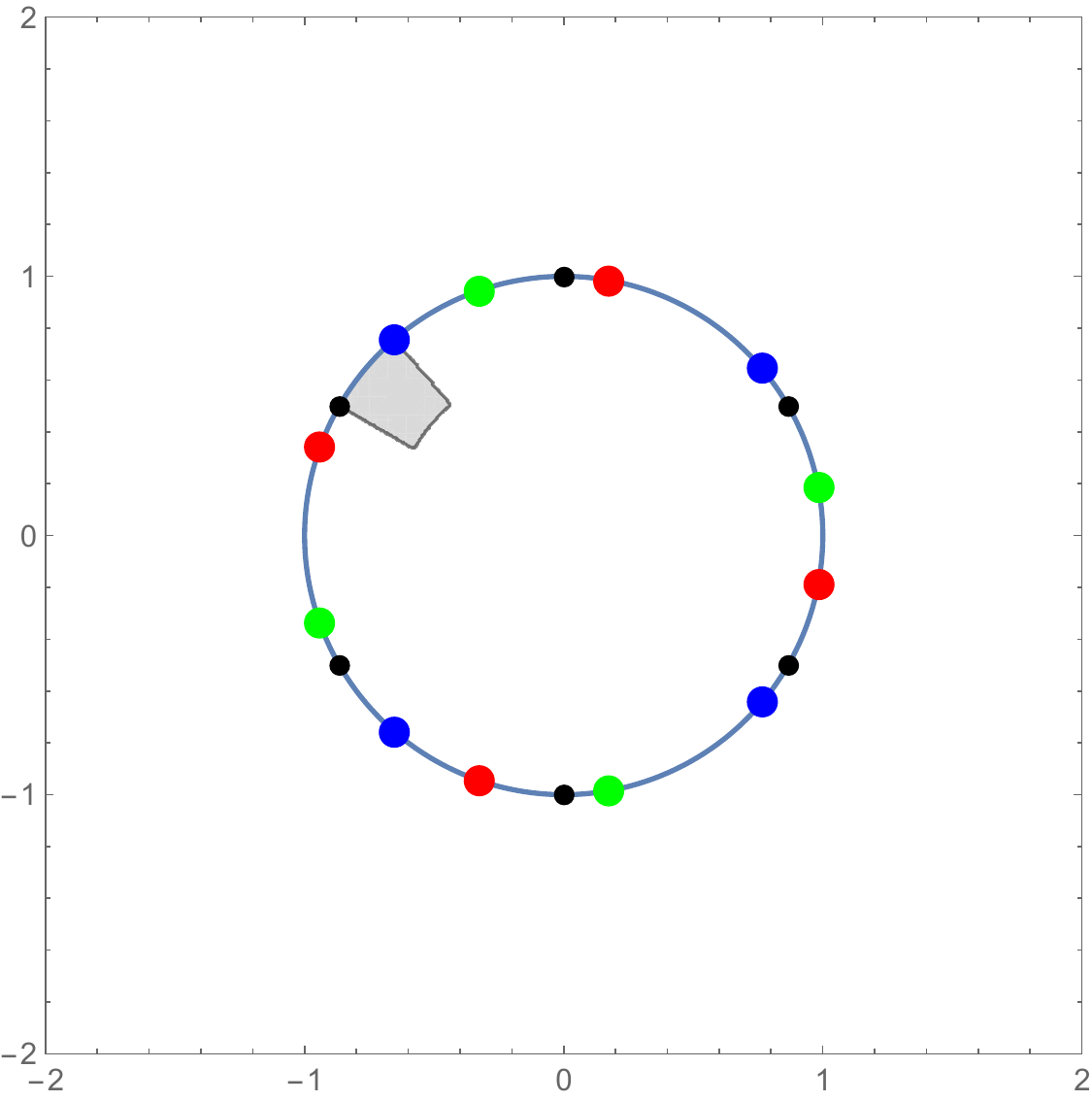}};
\node at (-0.61,0.61) {\tiny $U_{1}^{3}$};
\end{tikzpicture} \hspace{-0.05cm}
\begin{tikzpicture}[slave]
\node at (0,0) {\includegraphics[width=4.5cm]{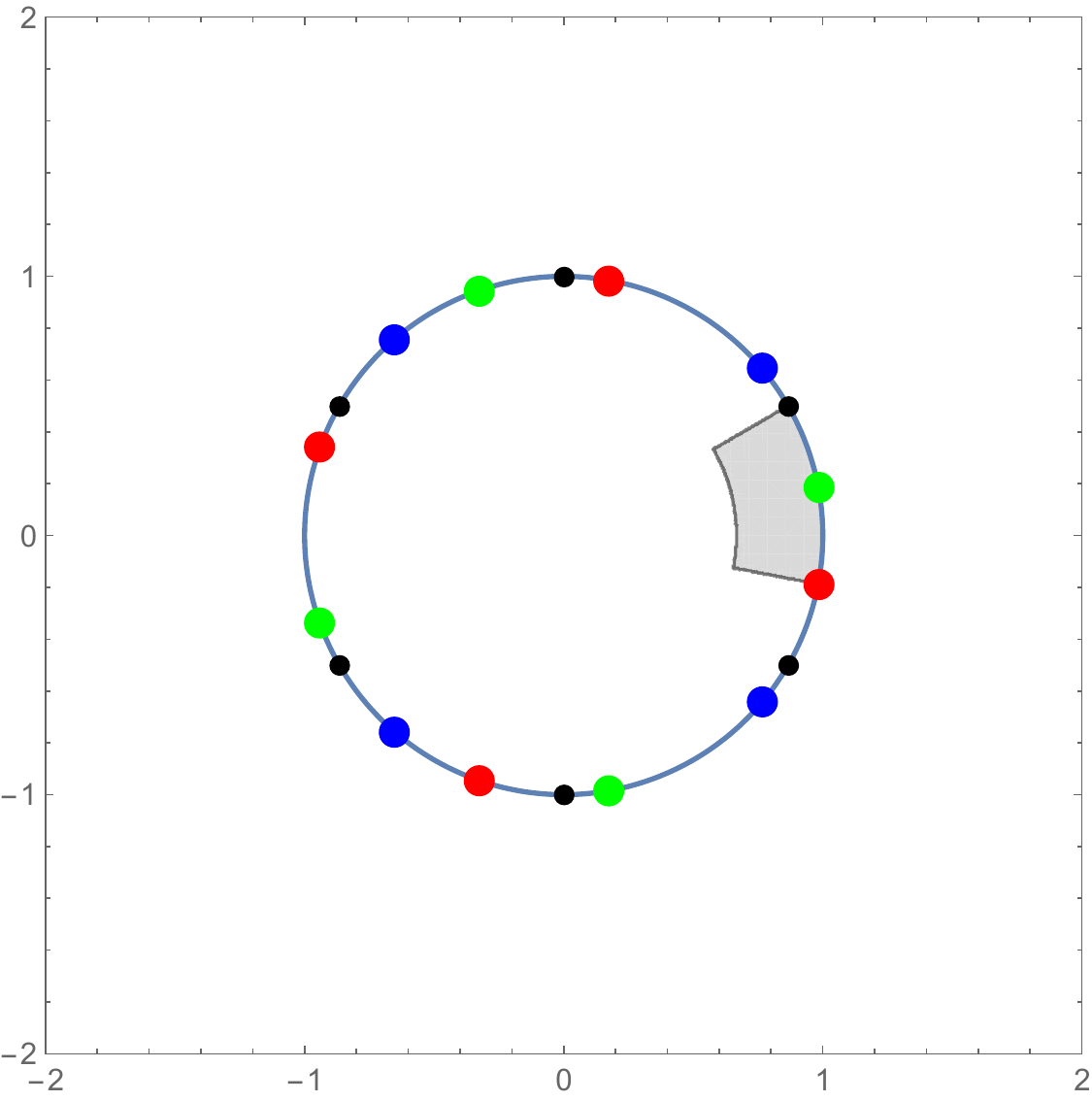}};
\node at (0.95,0.1) {\tiny $U_{1}^{4}$};
\end{tikzpicture} \hspace{-0.05cm}
\begin{tikzpicture}[slave]
\node at (0,0) {\includegraphics[width=4.5cm]{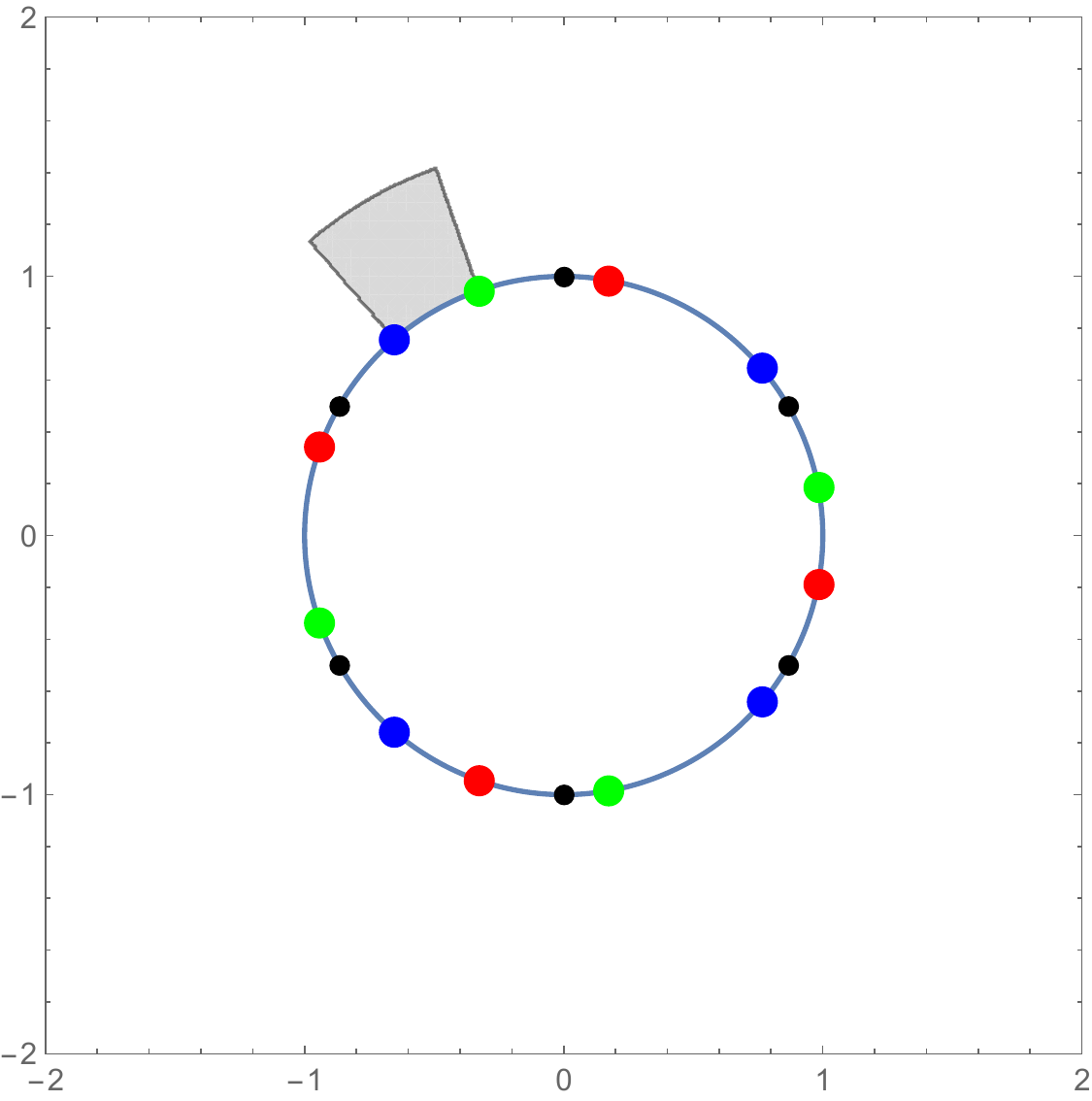}};
\node at (-0.6,1.2) {\tiny $U_{1}^{5}$};
\end{tikzpicture}
\end{center}
\begin{figuretext}
\label{IIbis fig: U1 and U2} The open subsets $\{U_{1}^{\ell}\}_{\ell=1}^{5}$ of the complex $k$-plane.
\end{figuretext}
\end{figure}

\begin{lemma}[Decomposition lemma]\label{IIbis decompositionlemma}
There exist $M>1$ and decompositions
\begin{align}
& r_{1,\ell}(k) = r_{1,\ell,a}(x, t, k) + r_{1,\ell,r}(x, t, k), & & k \in \partial U_{1}^{\ell} \cap \partial \mathbb{D}, \; \ell =1,\ldots,5, \label{IIbis decomposition lemma analytic + remainder}
\end{align}
such that the functions $\{r_{1,\ell,a},r_{1,\ell,r}\}_{\ell=1}^{5}$ have the following properties:
\begin{enumerate}[$(a)$]
\item 
For each $\zeta \in \mathcal{I}$, $t \geq 1$, and $\ell \in \{1,\dots,5\}$, $r_{1,\ell,a}(x, t, k)$ is defined and continuous for $k \in \bar{U}_1^{\ell}$ and analytic for $k \in U_1^{\ell}$.

\item For each $\zeta \in \mathcal{I}$, $t \geq 1$, and $\ell \in \{1,\dots,5\}$, the function $r_{1,\ell,a}$ satisfies
\begin{align*}
& \Big| r_{1,\ell,a}(x, t, k)-\sum_{j=0}^{N}\frac{r_{1,\ell}^{(j)}(k_{\star})}{j!}(k-k_{\star})^{j}  \Big| \leq C |k-k_{\star}|^{N+1}e^{\frac{t}{4}|\re \Phi_{21}(\zeta,k)|}, \; k \in \bar{U}_{1}^{\ell}, \; k_{\star} \in \mathcal{R}_{\ell},   
\end{align*}
where $\mathcal{R}_{1}=\{e^{5\pi i/6},\omega k_{3},-1,\omega^{2}k_{4},-i, \omega^{2}k_{3}, -\omega,k_{4}\}$, $\mathcal{R}_{2}=\{i,\omega^{2}k_{2},k_{4},e^{-\pi i/6}\}$, $\mathcal{R}_{3}=\{k_{1},e^{5\pi i/6}\}$, $\mathcal{R}_{4}=\{\omega k_{2},\omega^{2}k_{1},e^{\pi i/6}\}$, $\mathcal{R}_{5}=\{\omega^{2} k_{2}, k_{1}\}$, and the constant $C$ is independent of $\zeta, t, k$. Furthermore, for $\zeta \in \mathcal{I}$ and $t\geq 1$, we have
\begin{align*}
& \Big| r_{1,4,a}(x, t, k)-\sum_{j=0}^{N+n_{1}}\frac{[(\cdot-1)^{n_{1}}r_{1,4}(\cdot)]^{(j)}(1)}{j!}(k-1)^{j-n_{1}}  \Big| \leq C |k-1|^{N+1}e^{\frac{t}{4}|\re \Phi_{21}(\zeta,k)|}, \\
& \Big| r_{1,5,a}(x, t, k)-\sum_{j=0}^{N+n_{\omega}}\frac{[(\cdot-\omega)^{n_{\omega}}r_{1,5}(\cdot)]^{(j)}(\omega)}{j!}(k-\omega)^{j-n_{\omega}}  \Big| \leq C |k-\omega|^{N+1}e^{\frac{t}{4}|\re \Phi_{21}(\zeta,k)|},
\end{align*}
and these inequalities hold for $k \in \bar{U}_{1}^{4}$ and $k \in \bar{U}_{1}^{5}$, respectively.
\item For each $1 \leq p \leq \infty$ and $\ell \in \{1,\dots,5\}$, the $L^p$-norm of $r_{1,\ell,r}(x,t,\cdot)$ on $\partial \bar{U}_{1}^{\ell} \cap \partial \mathbb{D}$ is $O(t^{-N})$ uniformly for $\zeta \in \mathcal{I}$ as $t \to \infty$.
\end{enumerate}
\end{lemma}
\begin{proof}
For each $\ell \in \{1,\ldots,5\}$, $\theta \mapsto \im \Phi_{21}(\zeta,e^{i\theta})=(\zeta-\cos \theta)\sin \theta$ is monotone on each connected component of $\partial U_{1}^{\ell}\cap \partial \mathbb{D}$. Hence the claim can be proved using the techniques of \cite{DZ1993}. Since these techniques are rather standard by now, we omit the details. 
\end{proof}

Note that $\tilde{r}(k) \in \mathbb{R}$ for $k \in \partial \mathbb{D}\setminus \{-\omega^{2},\omega^{2}\}$ and that $\tilde{r}(k)= \tilde{r}(\frac{1}{\omega k})\tilde{r}(\frac{1}{\omega^{2} k})$. By \eqref{r1r2 relation with kbar symmetry}, we thus have $r_{2,\ell}(k) = \tilde{r}(k)\overline{r_{1,\ell}(\bar{k}^{-1})}$, $k \in \partial \mathbb{D}\cap U_{1}^{\ell}$, $\ell=1,\ldots,5$.
 Using this symmetry, for each $\ell\in \{1,\ldots,5\}$ we let $U_{2}^{\ell}:=\{k|\bar{k}^{-1}\in U_{1}^{\ell}\}$ and define a decomposition $r_{2,\ell}=r_{2,\ell,a}+r_{2,\ell,r}$ by
\begin{align*}
& r_{2,\ell,a}(k) := \tilde{r}(k)\overline{r_{1,\ell,a}(\bar{k}^{-1})}, \quad k \in U_{2}^{\ell},
& & r_{2,\ell,r}(k) := \tilde{r}(k)\overline{r_{1,\ell,r}(\bar{k}^{-1})}, \quad k \in \partial U_{2}^{\ell} \cap \partial \mathbb{D}.
\end{align*}

\begin{figure}[h]
\begin{center}
\begin{tikzpicture}[master,scale=0.9]
\node at (0,0) {};
\draw[black,line width=0.65 mm] (0,0)--(30:7.5);
\draw[black,line width=0.65 mm,->-=0.45,->-=0.91] (0,0)--(90:7);
\draw[black,line width=0.65 mm] (0,0)--(150:7.5);
\draw[dashed,black,line width=0.15 mm] (0,0)--(60:7.5);
\draw[dashed,black,line width=0.15 mm] (0,0)--(120:7.5);

\draw[black,line width=0.65 mm] ([shift=(30:3*1.5cm)]0,0) arc (30:150:3*1.5cm);
\draw[black,arrows={-Triangle[length=0.27cm,width=0.18cm]}]
($(67:3*1.5)$) --  ++(-20:0.001);

\draw[black,arrows={-Triangle[length=0.27cm,width=0.18cm]}]
($(123-5:3*1.5)$) --  ++(120+90-5:0.001);

\node at (70:3.2*1.5) {\small $2$};

\node at (115:3.22*1.5) {\small $11$};

\node at (82.5:1.97*1.5) {\small $1''$};

\node at (86.5:6.15) {\small $4'$};

\node at (85:3.2*1.5) {\small $5$};
\node at (100:3.2*1.5) {\small $8$};

\draw[blue,fill] (129.688:3*1.5) circle (0.12cm);
\draw[green,fill] (110.3:3*1.5) circle (0.12cm);
\draw[red,fill] (80:3*1.5) circle (0.12cm);
\draw[blue,fill] (38.4686:3*1.5) circle (0.12cm);


\node at (40:4.2) {\small $k_3$};
\node at (80:4.1) {\small $\omega k_4$};
\node at (110:4.08) {\small $\omega^2 k_2$};
\node at (130:4.08) {\small $k_1$};

\draw[black,arrows={-Triangle[length=0.27cm,width=0.18cm]}]
($(88:3*1.5)$) --  ++(86+90:0.001);
\draw[black,arrows={-Triangle[length=0.27cm,width=0.18cm]}]
($(103:3*1.5)$) --  ++(103+90:0.001);
\end{tikzpicture}
\end{center}
\begin{figuretext}
\label{IIbis Gammap0p}The contour $\Gamma^{(0)}$ (solid), the boundary of $\mathsf{S}$ (dashed) and, from right to left, the saddle points $k_{3}$ (blue), $\omega k_{4}$ (red), $\omega^{2}k_{2}$ (green), and $k_{1}$ (blue).
\end{figuretext}
\end{figure}

We now proceed with the first transformation $n \mapsto n^{(1)}$. As explained in Section \ref{overviewsec}, we will first define this transformation in the sector $\mathsf{S}$, and then appeal to the symmetries \eqref{nsymm} to extend it to the whole complex plane. Since we will need to open lenses differently on four subsets of $\partial \mathbb{D}\cap \mathsf{S}$, we find it convenient to introduce a new contour $\Gamma^{(0)}$. This contour coincides as a set with $\Gamma\cap \mathsf{S}$, but is oriented and labeled differently, see Figure \ref{IIbis Gammap0p}. We let $\Gamma_{j}^{(0)}$ denote the subcontour of $\Gamma^{(0)}$ labeled by $j$ in Figure \ref{IIbis Gammap0p}. We emphasize that 
\begin{align*}
\Gamma_{11}^{(0)}:= \{e^{i\theta} \,|\, \theta \in (\arg(\omega^{2}k_{2}),\tfrac{2\pi}{3})\}, \qquad \Gamma_{2}^{(0)}:= \{e^{i\theta} \,|\, \theta \in (\tfrac{\pi}{3},\arg(\omega k_{4}))\}.
\end{align*}
On $\Gamma_{2}^{(0)}$, we will use the factorization
\begin{align*}
v_{9} = v_{3}^{(1)}v_{2}^{(1)}v_{1}^{(1)},
\end{align*}
where
\begin{align}
& v_{3}^{(1)} = \begin{pmatrix}
1 & 0 & -r_{2,a}(\omega^{2}k)e^{-\theta_{31}} \\
r_{1,a}(\frac{1}{k})e^{\theta_{21}} & 1 & r_{1,a}(\omega k)e^{-\theta_{32}} \\
0 & 0 & 1
\end{pmatrix}, \; v_{1}^{(1)} = \begin{pmatrix}
1 & r_{2,a}(\frac{1}{k})e^{-\theta_{21}} & 0 \\
0 & 1 & 0  \\
-r_{1,a}(\omega^{2}k)e^{\theta_{31}} & r_{2,a}(\omega k)e^{\theta_{32}} & 1
\end{pmatrix}, \nonumber \\
& v_{2}^{(1)} = I + v_{2,r}^{(1)}, \qquad v_{2,r}^{(1)}= \begin{pmatrix}
r_{1,r}(\omega^{2}k)r_{2,r}(\omega^{2}k) & g_{2}(\omega k)e^{-\theta_{21}} & -r_{2,r}(\omega^{2}k)e^{-\theta_{31}} \\
g_{1}(\omega k)e^{\theta_{21}} & g(\omega k) & h_{1}(\omega k)e^{-\theta_{32}} \\
-r_{1,r}(\omega^{2}k)e^{\theta_{31}} & h_{2}(\omega k)e^{\theta_{32}} & 0
\end{pmatrix}, \label{vp1p 123}
\end{align}
and
\begin{align*}
h_{1}(k) = & \; r_{1,r}(k) + r_{1,a}(\tfrac{1}{\omega^{2}k})r_{2,r}(\omega k), \qquad h_{2}(k) =  r_{2,r}(k) + r_{2,a}(\tfrac{1}{\omega^{2}k})r_{1,r}(\omega k), \\
g_{1}(k) = & \; r_{1,r}(\tfrac{1}{\omega^{2}k})-r_{1,r}(\omega k) \big( r_{1,r}(k)+r_{1,a}(\tfrac{1}{\omega^{2}k})r_{2,r}(\omega k) \big), \\
g_{2}(k) = & \; r_{2,r}(\tfrac{1}{\omega^{2}k})-r_{2,r}(\omega k) \big( r_{2,r}(k)+r_{2,a}(\tfrac{1}{\omega^{2}k})r_{1,r}(\omega k) \big), \\
g(k) = & \; r_{1,r}(k)\big(r_{1,r}(\omega k)r_{2,a}(\tfrac{1}{\omega^{2}k})+r_{2,r}(k)\big) \\
& +r_{1,a}(\tfrac{1}{\omega^{2}k})r_{2,r}(\omega k)\big( r_{1,r}(\omega k)r_{2,a}(\tfrac{1}{\omega^{2}k})+r_{2,r}(k) \big) + r_{1,r}(\tfrac{1}{\omega^{2}k})r_{2,r}(\tfrac{1}{\omega^{2}k}).
\end{align*}
On $\Gamma_{5}^{(0)}$, we will use the factorization
\begin{align}
& v_{9}^{-1} = v_{4}^{(1)}v_{5}^{(1)}v_{6}^{(1)}, \qquad v_{5}^{(1)} = \begin{pmatrix}
\frac{1}{1+r_{1}(\omega^{2}k)r_{2}(\omega^{2}k)} & 0 & 0 \\
0 & 1 & 0 \\
0 & 0 & 1+r_{1}(\omega^{2}k)r_{2}(\omega^{2}k)
\end{pmatrix} + v_{5,r}^{(1)}, \nonumber \\ 
& v_{4}^{(1)} = \begin{pmatrix}
1 & c_{12,a}e^{-\theta_{21}} & c_{13,a}e^{-\theta_{31}} \\
0 & 1 & 0 \\
0 & c_{32,a}e^{\theta_{32}} & 1
\end{pmatrix}, \qquad v_{6}^{(1)} = \begin{pmatrix}
1 & 0 & 0 \\
c_{21,a}e^{\theta_{21}} & 1 & c_{23,a}e^{-\theta_{32}} \\
c_{31,a}e^{\theta_{31}} & 0 & 1
\end{pmatrix}, \label{vp1p 456}
\end{align}
where
\begin{align*}
& c_{12} = -r_{2}(\tfrac{1}{k}), & & c_{13} = \hat{r}_{2}(\omega^{2}k), & & c_{23} = r_{2}(\tfrac{1}{\omega k}), \\
& c_{21} = -r_{1}(\tfrac{1}{k}), & & c_{31} = \hat{r}_{1}(\omega^{2}k), & & c_{32} = r_{1}(\tfrac{1}{\omega k}),
\end{align*}
and $c_{ij,a}, c_{ij,r}$ denote the analytic continuation and the remainder of $c_{ij}$ from Lemma \ref{IIbis decompositionlemma}, i.e., $c_{12,a} = -r_{2,1,a}(\frac{1}{k})$, $c_{12,r} = -r_{2,1,r}(\frac{1}{k})$, $c_{13,a} = -r_{2,2,a}(\omega^2 k)$, etc.

On $\Gamma_{8}^{(0)}$, we will use the factorization
\begin{align}
& v_{7} = v_{7}^{(1)}v_{8}^{(1)}v_{9}^{(1)}, \qquad v_{8}^{(1)} = \begin{pmatrix}
\frac{1}{f(k)} & 0 & 0 \\
0 & 1+r_{1}(k)r_{2}(k) & 0 \\
0 & 0 & \frac{f(k)}{1+r_{1}(k)r_{2}(k)}
\end{pmatrix} + v_{8,r}^{(1)}, \nonumber \\
& v_{7}^{(1)} = \begin{pmatrix}
1 & a_{12,a}e^{-\theta_{21}} & a_{13,a}e^{-\theta_{31}} \\
0 & 1 & 0 \\
0 & a_{32,a}e^{\theta_{32}} & 1
\end{pmatrix}, \qquad v_{9}^{(1)} = \begin{pmatrix}
1 & 0 & 0 \\
a_{21,a}e^{\theta_{21}} & 1 & a_{23,a}e^{-\theta_{32}} \\
a_{31,a}e^{\theta_{31}} & 0 & 1
\end{pmatrix}, \label{vp1p 789}
\end{align}
where
\begin{align*}
& a_{12} = -\hat{r}_{1}(k), & & a_{13} = -\frac{r_{1}(\frac{1}{\omega^{2}k})}{f(k)}, & & a_{23} = \frac{r_{2}(\frac{1}{\omega k})-r_{2}(k)r_{2}(\omega^{2}k)}{1+r_{1}(k)r_{2}(k)}, \\
& a_{21} = -\hat{r}_{2}(k), & & a_{31} = - \frac{r_{2}(\frac{1}{\omega^{2}k})}{f(k)}, & & a_{32} = \frac{r_{1}(\frac{1}{\omega k})-r_{1}(k)r_{1}(\omega^{2}k)}{1+r_{1}(k)r_{2}(k)}.
\end{align*}
and $a_{ij,a}, a_{ij,r}$ denote the analytic continuation and the remainder of $a_{ij}$ from Lemma \ref{IIbis decompositionlemma}, e.g., $a_{13,a} = -r_{1,4,a}(\frac{1}{\omega^2 k})$ and $a_{23,a} = r_{2,3,a}(\frac{1}{\omega k})$.

Finally, on $\Gamma_{11}^{(0)}$, we will use the factorization
\begin{align}
& v_{7} = v_{10}^{(1)}v_{11}^{(1)}v_{12}^{(1)}, \qquad v_{11}^{(1)} = \begin{pmatrix}
\frac{1}{f(k)} & 0 & 0 \\
0 & \frac{f(k)}{f(\omega^{2}k)} & 0 \\
0 & 0 & f(\omega^{2}k)
\end{pmatrix} + v_{11,r}^{(1)}, \nonumber \\
& v_{10}^{(1)} = \begin{pmatrix}
1 & b_{12,a}e^{-\theta_{21}} & b_{13,a}e^{-\theta_{31}} \\
0 & 1 & b_{23,a}e^{-\theta_{32}} \\
0 & 0 & 1
\end{pmatrix}, \qquad v_{12}^{(1)} = \begin{pmatrix}
1 & 0 & 0 \\
b_{21,a}e^{\theta_{21}} & 1 & 0 \\
b_{31,a}e^{\theta_{31}} & b_{32,a}e^{\theta_{32}} & 1
\end{pmatrix}, \label{vp1p 101112}
\end{align}
where
\begin{align*}
& b_{12} = - \frac{r_{1}(k)- r_{1}(\frac{1}{\omega k})r_{1}(\frac{1}{\omega^{2}k})}{f(k)}, & & b_{13} = \frac{r_{2}(\omega^{2}k)}{f(\omega^{2}k)}, & & b_{23} = \frac{r_{2}(\frac{1}{\omega k})-r_{2}(k)r_{2}(\omega^{2}k)}{f(\omega^{2}k)}, \\
& b_{21} = -\frac{r_{2}(k) - r_{2}(\frac{1}{\omega k})r_{2}(\frac{1}{\omega^{2}k})}{f(k)}, & & b_{31} = \frac{r_{1}(\omega^{2}k)}{f(\omega^{2}k)}, & & b_{32} = \frac{r_{1}(\frac{1}{\omega k})-r_{1}(k)r_{1}(\omega^{2}k)}{f(\omega^{2}k)},
\end{align*}
and $b_{ij,a}, b_{ij,r}$ denote the analytic continuation and the remainder of $b_{ij}$ from Lemma \ref{IIbis decompositionlemma}, e.g., $b_{12,a} = -r_{1,5,a}(k)$ and $b_{23,a} = r_{2,5,a}(\frac{1}{\omega k})$.
The long expressions for $v_{5,r}^{(1)}$, $v_{8,r}^{(1)}$, and $v_{11,r}^{(1)}$ are similar to $v_{2,r}^{(1)}$ and are omitted, but we mention that Lemma \ref{IIbis decompositionlemma} ensures that 
\begin{align*}
\| v_{j,r}^{(1)} \|_{(L^{1}\cap L^{\infty})(\Gamma_{j}^{(0)})} = O(t^{-1}), \qquad \mbox{as } t \to \infty, \; j= 2,5,8,11.
\end{align*}

\begin{figure}
\begin{center}
\begin{tikzpicture}[master]
\node at (0,0) {};
\draw[black,line width=0.65 mm] (0,0)--(30:7.5);
\draw[black,line width=0.65 mm,->-=0.25,->-=0.57,->-=0.71,->-=0.91] (0,0)--(90:7.5);
\draw[black,line width=0.65 mm] (0,0)--(150:7.5);
\draw[dashed,black,line width=0.15 mm] (0,0)--(60:7.5);
\draw[dashed,black,line width=0.15 mm] (0,0)--(120:7.5);

\draw[black,line width=0.65 mm] ([shift=(30:3*1.5cm)]0,0) arc (30:150:3*1.5cm);
\draw[black,arrows={-Triangle[length=0.27cm,width=0.18cm]}]
($(66:3*1.5)$) --  ++(-21:0.001);
\draw[black,arrows={-Triangle[length=0.27cm,width=0.18cm]}]
($(119:3*1.5)$) --  ++(116+90:0.001);

\draw[black,line width=0.65 mm] ([shift=(30:3*1.5cm)]0,0) arc (30:150:3*1.5cm);

\node at (68:3.15*1.5) {\scriptsize $2$};
\node at (70:3.6*1.5) {\scriptsize $1$};
\node at (70:2.47*1.5) {\scriptsize $3$};

\node at (116.5:3.155*1.5) {\scriptsize $11$};
\node at (115.5:3.65*1.5) {\scriptsize $10$};
\node at (114:2.55*1.5) {\scriptsize $12$};

\node at (99.5:1.15*1.5) {\scriptsize $1''$};
\node at (93.5:2.75*1.5) {\scriptsize $1_{r}''$};

\node at (92:6.6) {\scriptsize $4'$};
\node at (93.1:3.4*1.5) {\scriptsize $4_{r}'$};

\node at (86.5:3.17*1.5) {\scriptsize $5$};
\node at (84.5:3.55*1.5) {\scriptsize $4$};
\node at (83:2.65*1.5) {\scriptsize $6$};

\node at (100:3.17*1.5) {\scriptsize $8$};
\node at (100:3.62*1.5) {\scriptsize $7$};
\node at (100:2.45*1.5) {\scriptsize $9$};

\draw[black,arrows={-Triangle[length=0.27cm,width=0.18cm]}]
($(88:3*1.5)$) --  ++(86+90:0.001);
\draw[black,arrows={-Triangle[length=0.27cm,width=0.18cm]}]
($(103:3*1.5)$) --  ++(103+90:0.001);

\draw[black,line width=0.65 mm] (150:3.65)--($(129.688:3*1.5)+(129.688+135:0.5)$)--(129.688:3*1.5)--($(129.688:3*1.5)+(129.688+45:0.5)$)--(150:5.8);
\draw[black,line width=0.65 mm] (30:3.65)--($(38.4686:3*1.5)+(38.4686-135:0.5)$)--(38.4686:3*1.5)--($(38.4686:3*1.5)+(38.4686-45:0.5)$)--(30:5.8);
\draw[black,line width=0.65 mm,-<-=0.12,->-=0.78] (90:3.65)--($(80:3*1.5)+(80+135:0.5)$)--(80:3*1.5)--($(80:3*1.5)+(80+45:0.5)$)--(90:5.8);
\draw[black,line width=0.65 mm,->-=0.25,-<-=0.65] (90:3.65)--($(110.3:3*1.5)+(110.3-135:0.5)$)--(110.3:3*1.5)--($(110.3:3*1.5)+(110.3-45:0.5)$)--(90:5.8);
\draw[black,line width=0.65 mm,-<-=0.12,->-=0.78] (120:3.65)--($(110.3:3*1.5)+(110.3+135:0.5)$)--(110.3:3*1.5)--($(110.3:3*1.5)+(110.3+45:0.5)$)--(120:5.8);
\draw[black,line width=0.65 mm] (120:3.65)--($(129.688:3*1.5)+(129.688-135:0.5)$)--(129.688:3*1.5)--($(129.688:3*1.5)+(129.688-45:0.5)$)--(120:5.8);
\draw[black,line width=0.65 mm,-<-=0.12,->-=0.78] (60:3.65)--($(80:3*1.5)+(80-135:0.5)$)--(80:3*1.5)--($(80:3*1.5)+(80-45:0.5)$)--(60:5.8);
\draw[black,line width=0.65 mm] (60:3.65)--($(38.4686:3*1.5)+(38.4686+135:0.5)$)--(38.4686:3*1.5)--($(38.4686:3*1.5)+(38.4686+45:0.5)$)--(60:5.8);

\draw[blue,fill] (129.688:3*1.5) circle (0.12cm);
\draw[green,fill] (110.3:3*1.5) circle (0.12cm);
\draw[red,fill] (80:3*1.5) circle (0.12cm);
\draw[blue,fill] (38.4686:3*1.5) circle (0.12cm);

\draw[black,line width=0.65 mm,->-=0.6] (60:4.5)--(60:5.8);
\draw[black,line width=0.65 mm,->-=0.75] (60:3.65)--(60:4.5);
\draw[black,line width=0.65 mm,->-=0.6] (120:4.5)--(120:5.8);
\draw[black,line width=0.65 mm,->-=0.75] (120:3.65)--(120:4.5);

\node at (122:5) {\scriptsize $1_{s}$};
\node at (122.5:4.2) {\scriptsize $2_{s}$};
\node at (57:5.1) {\scriptsize $3_{s}$};
\node at (56.5:4.15) {\scriptsize $4_{s}$};
\end{tikzpicture}
\end{center}
\begin{figuretext}
\label{IIbis Gammap1p}The contour $\Gamma^{(1)}$ (solid), the boundary of $\mathsf{S}$ (dashed) and, from right to left, the saddle points $k_{3}$ (blue), $\omega k_{4}$ (red), $\omega^{2}k_{2}$ (green), and $k_{1}$ (blue).
\end{figuretext}
\end{figure}

Let $\Gamma^{(1)}$ be the contour shown in Figure \ref{IIbis Gammap1p}.
Define the piecewise analytic function $n^{(1)}$ by
\begin{align}\label{Sector IIbis first transfo}
n^{(1)}(x,t,k) = n(x,t,k)G^{(1)}(x,t,k),\qquad k \in \C \setminus \Gamma^{(1)},
\end{align}
where $G^{(1)}$ is defined for $k \in \mathsf{S}$ by
\begin{align}\label{IIbis Gp1pdef}
\hspace{-0.1cm} G^{(1)} \hspace{-0.1cm} = \hspace{-0.1cm} \begin{cases} 
v_{3}^{(1)}, & \hspace{-0.3cm} k \mbox{ on the $-$ side of }\Gamma_{2}^{(1)}, \\
(v_{1}^{(1)})^{-1}, & \hspace{-0.3cm} k \mbox{ on the $+$ side of }\Gamma_{2}^{(1)},  \\
v_{4}^{(1)}, & \hspace{-0.3cm} k \mbox{ on the $-$ side of }\Gamma_{5}^{(1)}, \\
(v_{6}^{(1)})^{-1}, & \hspace{-0.3cm} k \mbox{ on the $+$ side of }\Gamma_{5}^{(1)},
\end{cases} \; G^{(1)} \hspace{-0.1cm} = \hspace{-0.1cm} \begin{cases} 
v_{7}^{(1)}, & \hspace{-0.3cm} k \mbox{ on the $-$ side of }\Gamma_{8}^{(1)}, \\
(v_{9}^{(1)})^{-1}, & \hspace{-0.3cm} k \mbox{ on the $+$ side of }\Gamma_{8}^{(1)},  \\
v_{10}^{(1)}, & \hspace{-0.3cm} k \mbox{ on the $-$ side of }\Gamma_{11}^{(1)}, \\
(v_{12}^{(1)})^{-1}, & \hspace{-0.3cm} k \mbox{ on the $+$ side of }\Gamma_{11}^{(1)},  \\
I, & \hspace{-0.3cm} \mbox{otherwise},
\end{cases}
\end{align}
and $G^{(1)}$ is extended to $\mathbb{C}\setminus \Gamma^{(1)}$ by
\begin{align}\label{symmetry of G1}
G^{(1)}(x,t, k) = \mathcal{A} G^{(1)}(x,t,\omega k)\mathcal{A}^{-1}
 = \mathcal{B} G^{(1)}(x,t, k^{-1}) \mathcal{B}.
\end{align}
The functions $G^{(1)}$ and $n^{(1)}$ are analytic on $\C \setminus \Gamma^{(1)}$. Indeed, consider for example the small region inside the lens on the $-$ side of $\Gamma_{11}^{(1)}$; in this region $G^{(1)} = v_{10}^{(1)}$ involves the functions $r_{1,5,a}(k)$, $r_{2,5,a}(\frac{1}{\omega k})$, and $r_{2,4,a}(\omega^2 k)$, and these functions are all analytic in this region as a consequence of Lemma \ref{IIbis decompositionlemma} and the definitions of $U_1^5$, $U_2^5$, and $U_2^4$.

The next lemma easily follows from Figure \ref{IIbis fig: Re Phi 21 31 and 32 for zeta=0.7}. Note that $v_{10}^{(1)},v_{11}^{(1)},v_{12}^{(1)}$ (and hence also $n^{(1)}(x,t,k)$) are singular at $k = \omega$, because $f(\omega) = f(1) = 0$. This is why the disks $D_\epsilon(\omega^j)$ have been excluded in Lemma \ref{lemma:G1p1p}. 

\begin{lemma}\label{lemma:G1p1p}
For any $\epsilon>0$, $G^{(1)}(x,t,k)$ and $G^{(1)}(x,t,k)^{-1}$ are uniformly bounded for $k \in \mathbb{C}\setminus (\Gamma^{(1)} \cup \cup_{j=0}^{2} D_\epsilon(\omega^j))$, $t\geq 1$, and $\zeta \in \mathcal{I}$. Furthermore, $G^{(1)}(x,t,k)=I$ for all large enough $|k|$.
\end{lemma}

By construction, $n^{(1)}_{+}=n^{(1)}_{-}v_{j}^{(1)}$ on $\Gamma^{(1)}_{j}$, where $v_{j}^{(1)}$ is given by \eqref{vp1p 123} for $j=1,2,3$, by \eqref{vp1p 456} for $j=4,5,6$, by \eqref{vp1p 789} for $j=7,8,9$, by \eqref{vp1p 101112} for $j=10,11,12$, and
\begin{align}
v_{1''}^{(1)} = & \; v_{1''}, \quad v_{4'}^{(1)} = v_{4'}, \quad v_{1_{r}''}^{(1)} = v_{6}^{(1)} v_{1''}(v_{9}^{(1)})^{-1}, \quad v_{4_{r}'}^{(1)} = (v_{4}^{(1)})^{-1}v_{4'}v_{7}^{(1)}, \nonumber \\
 v_{1_s}^{(1)} = &\; G_-^{(1)}(k)^{-1}G_+^{(1)}(k)
= v_{10}^{(1)}(k)^{-1} \mathcal{A} \mathcal{B} v_{12}^{(1)}(\tfrac{1}{\omega k})^{-1} \mathcal{B} \mathcal{A}^{-1}  \nonumber
	\\ \label{v1sexpression}
=&\; \begin{pmatrix}
 1 & -(b_{12,a}(k)+b_{32,a}(\frac{1}{\omega k}))e^{-\theta_{21}} & (v_{1_s}^{(1)})_{13}
 \\
 0 & 1 & -(b_{21,a}(\frac{1}{\omega k})+b_{23,a}(k))e^{-\theta_{32}} \\
 0 & 0 & 1 \\
\end{pmatrix},
	\\ \nonumber
 (v_{1_s}^{(1)})_{13} := &\;  \big(b_{12,a}(k) (b_{21,a}(\tfrac{1}{\omega k})+b_{23,a}(k)) +b_{21,a}(\tfrac{1}{\omega k}) b_{32,a}(\tfrac{1}{\omega k}) -b_{13,a}(k) -b_{31,a}(\tfrac{1}{\omega k})\big) e^{-\theta_{31}}.
\end{align}
The expressions for $v_{2_{s}}, v_{3_{s}}, v_{4_{s}}$ are similar to \eqref{v1sexpression} and we do not write them down. 

\section{The $n^{(1)} \to n^{(2)}$ transformation}
We now proceed with a second opening of lenses. 
We will use the following factorizations (the matrices with subscripts $u$ and $d$ will be used to deform contours up and down, respectively, see \eqref{IIbis Gp2pdef} below):
\begin{align}\nonumber
& v_{1}^{(1)}= v_{1}^{(2)}v_{1u}^{(2)}, \quad v_{1}^{(2)} = \begin{pmatrix}
1 & 0 & 0 \\
0 & 1 & 0 \\
-r_{1,a}(\omega^{2}k)e^{\theta_{31}} & 0 & 1
\end{pmatrix}, 
	\\ 
&\hspace{2.8cm} v_{1u}^{(2)} = \begin{pmatrix}
1 & r_{2,a}(\frac{1}{k})e^{-\theta_{21}} & 0 \\
0 & 1 & 0 \\
0 & \big(r_{1,a}(\omega^2 k)r_{2,a}(\frac{1}{k}) + r_{2,a}(\omega k)\big)e^{\theta_{32}} & 1
\end{pmatrix}, \nonumber 
	\\ \nonumber
& v_{3}^{(1)}= v_{3d}^{(2)}v_{3}^{(2)}, \; 
v_{3d}^{(2)} = \begin{pmatrix}
1 & 0 & 0 \\
r_{1,a}(\frac{1}{k})e^{\theta_{21}} & 1 & \big(r_{1,a}(\omega k) + r_{1,a}(\frac{1}{k}) r_{2,a}(\omega^2 k)\big)e^{-\theta_{32}} \\
0 & 0 & 1
\end{pmatrix}, 
	\\ 
&\hspace{2.7cm} v_{3}^{(2)} = \begin{pmatrix}
1 & 0 & -r_{2,a}(\omega^{2}k)e^{-\theta_{31}} \\
0 & 1 & 0 \\
0 & 0 & 1
\end{pmatrix}, \nonumber \\
& v_{4}^{(1)}=v_{4u}^{(2)}v_{4}^{(2)}, \quad v_{4u}^{(2)} = \begin{pmatrix}
1 & c_{12,a}e^{-\theta_{21}} & 0 \\
0 & 1 & 0 \\
0 & c_{32,a}e^{\theta_{32}} & 1
\end{pmatrix}, \quad v_{4}^{(2)} = \begin{pmatrix}
1 & 0 & c_{13,a}e^{-\theta_{31}} \\
0 & 1 & 0 \\
0 & 0 & 1
\end{pmatrix}, \nonumber \\
& v_{6}^{(1)}=v_{6}^{(2)}v_{6d}^{(2)}, \quad v_{6}^{(2)} = \begin{pmatrix}
1 & 0 & 0 \\
0 & 1 & 0 \\
c_{31,a}e^{\theta_{31}} & 0 & 1
\end{pmatrix}, \quad v_{6d}^{(2)} = \begin{pmatrix}
1 & 0 & 0 \\
c_{21,a}e^{\theta_{21}} & 1 & c_{23,a}e^{-\theta_{32}} \\
0 & 0 & 1
\end{pmatrix}, \nonumber \\
& v_{7}^{(1)} = v_{7u}^{(2)}v_{7}^{(2)}, \; v_{7}^{(2)} = \begin{pmatrix}
1 & \hspace{-0.13cm} 0 & \hspace{-0.13cm} 0 \\
0 & \hspace{-0.13cm} 1 & \hspace{-0.13cm} 0 \\
0 & \hspace{-0.13cm} a_{32,a}e^{\theta_{32}} & \hspace{-0.13cm} 1
\end{pmatrix}, \; v_{7u}^{(2)} = \begin{pmatrix}
1 & \hspace{-0.13cm} (a_{12,a}-a_{13,a}a_{32,a})e^{-\theta_{21}} & \hspace{-0.18cm} a_{13,a}e^{-\theta_{31}} \\
0 & \hspace{-0.13cm} 1 & \hspace{-0.13cm} 0 \\
0 & \hspace{-0.13cm} 0 & \hspace{-0.13cm} 1
\end{pmatrix}\hspace{-0.1cm}, \nonumber \\
& v_{9}^{(1)} = v_{9}^{(2)}v_{9d}^{(2)}, \quad v_{9d}^{(2)} = \begin{pmatrix}
1 & 0 & 0 \\
(a_{21,a}-a_{23,a}a_{31,a})e^{\theta_{21}} & 1 & 0 \\
a_{31,a}e^{\theta_{31}} & 0 & 1
\end{pmatrix}, \quad v_{9}^{(2)} = \begin{pmatrix}
1 & 0 & 0 \\
0 & 1 & a_{23,a}e^{-\theta_{32}} \\
0 & 0 & 1
\end{pmatrix}, \nonumber \\
& v_{10}^{(1)} = v_{10u}^{(2)} v_{10}^{(2)}, \; v_{10}^{(2)} = \begin{pmatrix}
1 & \hspace{-0.13cm} 0 & \hspace{-0.13cm} 0 \\
0 & \hspace{-0.13cm} 1 & \hspace{-0.13cm} b_{23,a}e^{-\theta_{32}} \\
0 & \hspace{-0.13cm} 0 & \hspace{-0.13cm} 1
\end{pmatrix}\hspace{-0.1cm}, \; v_{10u}^{(2)} = \begin{pmatrix}
1 & \hspace{-0.18cm} b_{12,a}e^{-\theta_{21}} & \hspace{-0.2cm} (b_{13,a}-b_{12,a}b_{23,a})e^{-\theta_{31}} \\
0 & \hspace{-0.13cm} 1 & \hspace{-0.13cm} 0 \\
0 & \hspace{-0.13cm} 0 & \hspace{-0.13cm} 1
\end{pmatrix}\hspace{-0.1cm}, \nonumber \\
& v_{12}^{(1)} = v_{12}^{(2)} v_{12d}^{(2)}, \; v_{12d}^{(2)} = \begin{pmatrix}
1 & 0 & 0 \\
b_{21,a}e^{\theta_{21}} & 1 & 0 \\
(b_{31,a}-b_{21,a}b_{32,a})e^{\theta_{31}} & 0 & 1
\end{pmatrix}, \; v_{12}^{(2)} = \begin{pmatrix}
1 & 0 & 0 \\
0 & 1 & 0 \\
0 & b_{32,a}e^{\theta_{32}} & 1
\end{pmatrix}. \label{Vv2def}
\end{align}
\begin{figure}
\begin{center}
\begin{tikzpicture}[master]
\node at (0,0) {};
\draw[black,line width=0.65 mm] (0,0)--(30:7.5);
\draw[black,line width=0.65 mm,->-=0.25,->-=0.57,->-=0.71,->-=0.91] (0,0)--(90:7.5);
\draw[black,line width=0.65 mm] (0,0)--(150:7.5);
\draw[dashed,black,line width=0.15 mm] (0,0)--(60:7.5);
\draw[dashed,black,line width=0.15 mm] (0,0)--(120:7.5);

\draw[black,line width=0.65 mm] ([shift=(30:3*1.5cm)]0,0) arc (30:150:3*1.5cm);
\draw[black,arrows={-Triangle[length=0.27cm,width=0.18cm]}]
($(66:3*1.5)$) --  ++(-21:0.001);
\draw[black,arrows={-Triangle[length=0.27cm,width=0.18cm]}]
($(119:3*1.5)$) --  ++(116+90:0.001);

\draw[black,line width=0.65 mm] ([shift=(30:3*1.5cm)]0,0) arc (30:150:3*1.5cm);

\node at (68:3.14*1.5) {\scriptsize $2$};
\node at (70:3.6*1.5) {\scriptsize $1$};
\node at (68:2.73*1.5) {\scriptsize $3$};

\node at (116.5:3.155*1.5) {\scriptsize $11$};
\node at (114.5:3.52*1.5) {\scriptsize $10$};
\node at (114:2.6*1.5) {\scriptsize $12$};

\node at (93.7:2.72*1.5) {\scriptsize $1_{r}''$};
%
\node at (93:3.4*1.5) {\scriptsize $4_{r}'$};

\node at (86.5:3.12*1.5) {\scriptsize $5$};
\node at (84.5:3.5*1.5) {\scriptsize $4$};
\node at (84.4:2.62*1.5) {\scriptsize $6$};

\node at (100:3.15*1.5) {\scriptsize $8$};
\node at (100:3.58*1.5) {\scriptsize $7$};
\node at (102.8:2.57*1.5) {\scriptsize $9$};

\draw[black,arrows={-Triangle[length=0.27cm,width=0.18cm]}]
($(88:3*1.5)$) --  ++(86+90:0.001);
\draw[black,arrows={-Triangle[length=0.27cm,width=0.18cm]}]
($(103:3*1.5)$) --  ++(103+90:0.001);

\draw[black,line width=0.65 mm] (150:3.65)--($(129.688:3*1.5)+(129.688+135:0.5)$)--(129.688:3*1.5)--($(129.688:3*1.5)+(129.688+45:0.5)$)--(150:5.8);
\draw[black,line width=0.65 mm] (30:3.65)--($(38.4686:3*1.5)+(38.4686-135:0.5)$)--(38.4686:3*1.5)--($(38.4686:3*1.5)+(38.4686-45:0.5)$)--(30:5.8);
\draw[black,line width=0.65 mm,-<-=0.12,->-=0.78] (90:3.65)--($(80:3*1.5)+(80+135:0.5)$)--(80:3*1.5)--($(80:3*1.5)+(80+45:0.5)$)--(90:5.8);
\draw[black,line width=0.65 mm,->-=0.25,-<-=0.65] (90:3.65)--($(110.3:3*1.5)+(110.3-135:0.5)$)--(110.3:3*1.5)--($(110.3:3*1.5)+(110.3-45:0.5)$)--(90:5.8);
\draw[black,line width=0.65 mm,-<-=0.12,->-=0.78] (120:3.65)--($(110.3:3*1.5)+(110.3+135:0.5)$)--(110.3:3*1.5)--($(110.3:3*1.5)+(110.3+45:0.5)$)--(120:5.8);
\draw[black,line width=0.65 mm] (120:3.65)--($(129.688:3*1.5)+(129.688-135:0.5)$)--(129.688:3*1.5)--($(129.688:3*1.5)+(129.688-45:0.5)$)--(120:5.8);
\draw[black,line width=0.65 mm,-<-=0.12,->-=0.78] (60:3.65)--($(80:3*1.5)+(80-135:0.5)$)--(80:3*1.5)--($(80:3*1.5)+(80-45:0.5)$)--(60:5.8);
\draw[black,line width=0.65 mm] (60:3.65)--($(38.4686:3*1.5)+(38.4686+135:0.5)$)--(38.4686:3*1.5)--($(38.4686:3*1.5)+(38.4686+45:0.5)$)--(60:5.8);

\draw[black,line width=0.65 mm,->-=0.6] (60:4.5)--(60:5.8);
\draw[black,line width=0.65 mm,->-=0.75] (60:3.65)--(60:4.5);
\draw[black,line width=0.65 mm,->-=0.6] (120:4.5)--(120:5.8);
\draw[black,line width=0.65 mm,->-=0.75] (120:3.65)--(120:4.5);

\node at (122:5) {\scriptsize $1_{s}$};
\node at (122.5:4.2) {\scriptsize $2_{s}$};
\node at (57:5.1) {\scriptsize $3_{s}$};
\node at (56.5:4.15) {\scriptsize $4_{s}$};

\node at (77:5.2) {\scriptsize $5_{s}$};
\node at (76.5:4) {\scriptsize $6_{s}$};
\node at (107.2:5.13) {\scriptsize $7_{s}$};
\node at (106.8:4) {\scriptsize $8_{s}$};

\draw[black,line width=0.65 mm] (129.688:3.65)--(129.688:5.8);
\draw[black,line width=0.65 mm,-<-=0.12,->-=0.78] (110.3:3.65)--(110.3:5.8);
\draw[black,line width=0.65 mm,-<-=0.12,->-=0.78] (80:3.65)--(80:5.8);
\draw[black,line width=0.65 mm] (38.4686:3.65)--(38.4686:5.8);

\draw[black,line width=0.65 mm] ([shift=(30:3.65cm)]0,0) arc (30:150:3.65cm);
\draw[black,line width=0.65 mm] ([shift=(30:5.8cm)]0,0) arc (30:150:5.8cm);

\draw[blue,fill] (129.688:3*1.5) circle (0.12cm);
\draw[green,fill] (110.3:3*1.5) circle (0.12cm);
\draw[red,fill] (80:3*1.5) circle (0.12cm);
\draw[blue,fill] (38.4686:3*1.5) circle (0.12cm);

\draw[black,line width=0.65 mm,->-=0.6] (60:4.5)--(60:5.8);
\draw[black,line width=0.65 mm,->-=0.75] (60:3.65)--(60:4.5);
\draw[black,line width=0.65 mm,->-=0.6] (120:4.5)--(120:5.8);
\draw[black,line width=0.65 mm,->-=0.75] (120:3.65)--(120:4.5);
\end{tikzpicture}
\end{center}
\begin{figuretext}
\label{IIbis Gammap2p}The contour $\Gamma^{(2)}$ (solid), the boundary of $\mathsf{S}$ (dashed) and, from right to left, the saddle points $k_{3}$ (blue), $\omega k_{4}$ (red), $\omega^{2}k_{2}$ (green), and $k_{1}$ (blue).
\end{figuretext}
\end{figure}
Let $\Gamma^{(2)}$ be the contour displayed in Figure \ref{IIbis Gammap2p}, and let $\Gamma_{j}^{(2)}$ be the subcontour of $\Gamma^{(2)}$ labeled by $j$ in Figure \ref{IIbis Gammap2p}. Define the piecewise analytic function $n^{(2)}$ by
\begin{align}\label{Sector IIbis second transfo}
n^{(2)}(x,t,k) = n^{(1)}(x,t,k)G^{(2)}(x,t,k), \qquad k \in \C \setminus \Gamma^{(2)},
\end{align}
where $G^{(2)}$ is defined for $k \in \mathsf{S}$ by
\begin{align}\label{IIbis Gp2pdef}
G^{(2)}(x,t,k) = \begin{cases} 
(v_{1u}^{(2)})^{-1}, & \hspace{-0.15cm} k \mbox{ above }\Gamma_{1}^{(2)}, \\
v_{3d}^{(2)}, & \hspace{-0.15cm}k \mbox{ below }\Gamma_{3}^{(2)}, \\
v_{4u}^{(2)}, & \hspace{-0.15cm}k \mbox{ above }\Gamma_{4}^{(2)}, \\
(v_{6d}^{(2)})^{-1}, & \hspace{-0.15cm}k \mbox{ below }\Gamma_{6}^{(2)},
\end{cases} \;\; G^{(2)}(x,t,k) = \begin{cases} 
v_{7u}^{(2)}, & \hspace{-0.15cm}k \mbox{ above }\Gamma_{7}^{(2)}, \\
(v_{9d}^{(2)})^{-1}, & \hspace{-0.15cm}k \mbox{ below }\Gamma_{9}^{(2)}, \\
v_{10u}^{(2)}, & \hspace{-0.15cm}k \mbox{ above }\Gamma_{10}^{(2)}, \\
(v_{12d}^{(2)})^{-1}, & \hspace{-0.15cm}k \mbox{ below }\Gamma_{12}^{(2)}, \\
I, & \hspace{-0.15cm}\mbox{otherwise},
\end{cases}
\end{align}
and $G^{(2)}$ is extended to $\mathbb{C}\setminus \Gamma^{(2)}$ using the $\mathcal{A}$- and $\mathcal{B}$-symmetries (as in \eqref{symmetry of G1}). The following lemma follows easily from Lemma \ref{IIbis decompositionlemma} and Figure \ref{IIbis fig: Re Phi 21 31 and 32 for zeta=0.7}.
\begin{lemma}
For any $\epsilon >0$, $G^{(2)}(x,t,k)$ and $G^{(2)}(x,t,k)^{-1}$ are uniformly bounded for $k \in \mathbb{C}\setminus (\Gamma^{(2)}\cup \cup_{j=0}^{2}D_{\epsilon}(\omega^{j}))$, $t\geq 1$, and $\zeta \in \mathcal{I}$. Furthermore, $G^{(2)}(x,t,k)=I$ for large enough $|k|$.
\end{lemma}

The jumps $v_j^{(2)}$ of $n^{(2)}$ are given for $j = 1, \dots, 12$ by (\ref{Vv2def}) and by
\begin{align}\nonumber
& v_2^{(2)} = v_2^{(1)}, \quad v_5^{(2)} = v_5^{(1)}, \quad v_{8}^{(2)} = v_8^{(1)}, \quad v_{11}^{(2)} = v_{11}^{(1)}, \quad v_{j}^{(2)} = v_{j}^{(1)} \text{ for } j = 1_s, \dots, 4_s,
	\\ \label{II jumps vj diagonal p2p}
&  
v_{5_s}^{(2)} =  v_{1u}^{(2)} v_{4u}^{(2)}, \quad
v_{6_s}^{(2)} =  v_{6d}^{(2)} v_{3d}^{(2)},\quad
v_{7_s}^{(2)} =  (v_{7u}^{(2)})^{-1} v_{10u}^{(2)},\quad
v_{8_s}^{(2)} =  v_{12d}^{(2)} (v_{9d}^{(2)})^{-1}.
\end{align}


The jumps on the parts of $\Gamma^{(2)}\cap \mathsf{S}$ that are unnumbered in Figure \ref{IIbis Gammap2p} are small as $t \to \infty$, so we do not write them down. The matrix $v^{(2)}$ on $\Gamma^{(2)}\setminus \mathsf{S}$ can be obtained using \eqref{vjsymm}. The next lemma is proved in the same way as \cite[Lemma 8.5]{CLmain}.

\begin{lemma}\label{vsmallnearilemmaV}
The $L^\infty$-norm of $v^{(2)} - I$ on $\Gamma_{4_r'}^{(2)} \cup \Gamma_{1_r''}^{(2)}$ is $O(t^{-N})$ as $t \to \infty$ uniformly for $\zeta \in \mathcal{I}$.
\end{lemma}

The next lemma shows that the jumps $\smash{v_{j}^{(2)}}$, $j = 1_s, \dots, 8_s$, also are small for large $t$.

\begin{lemma}\label{symmetryjumpslemma}
It is possible to choose the analytic approximations so that the $L^\infty$-norm of $v_j^{(2)} - I$ on $\Gamma_j^{(2)}$, $j = 1_s, \dots, 8_s$, is $O(t^{-N})$ as $t \to \infty$ uniformly for $\zeta \in \mathcal{I}$.
\end{lemma}
\begin{proof}
In the case of compactly supported data $u_{0},v_{0}$, we can choose $r_{1,a} = r_1$, $r_{2,a} = r_2$ and $a_{ij,a}=a_{ij}$, $b_{ij,a}=b_{ij}$, $c_{ij,a}=c_{ij}$ for all $1\leq i,j \leq 3$, and then long but straightforward calculations using (\ref{r1r2 relation on the unit circle}) show that $\smash{v_{j_{s}}^{(2)}}=I$ for $j = 1, \dots, 8$. In the general case of non-compactly supported data, the matrices $\smash{v_{j_s}^{(2)}}$ are non-trivial, but it is still possible to choose decompositions such that they remain close to $I$ as $t \to \infty$. The proof is similar to \cite[Proof of Lemma 9.2]{CLmain} so we omit it.
\end{proof}

Thus $v^{(2)}-I$ is $O(t^{-N})$ as $t \to \infty$, uniformly for $k \in \Gamma^{(2)} \setminus \big(\cup_{j=0}^{2}\cup_{\ell=1}^{4} D_{\epsilon}(\omega^{j} k_{\ell}) \cup \partial \D \big)$. 
The next transformation $n^{(2)}\to n^{(3)}$ makes $v^{(3)}-I$ uniformly small as $t \to \infty$ for $k \in \partial \mathbb{D}$. To prepare ourselves for this transformation, we first construct a global parametrix.

\section{Global parametrix}
For each $\zeta \in \mathcal{I}$, let
\begin{align*}
\delta_{1}(\zeta, \cdot): \mathbb{C}\setminus \Gamma_5^{(2)} \to \mathbb{C}, \quad \delta_{2}(\zeta, \cdot), \delta_{3}(\zeta, \cdot): \mathbb{C}\setminus \Gamma_8^{(2)}\to \mathbb{C}, \quad \delta_{4}(\zeta, \cdot), \delta_{5}(\zeta, \cdot): \mathbb{C}\setminus \Gamma_{11}^{(2)} \to \mathbb{C}
\end{align*}
be analytic with continuous boundary values $\{\delta_{j,\pm}\}_{j=1}^{5}$ satisfying
\begin{align*}
& \delta_{1,+}(\zeta, k) = \delta_{1,-}(\zeta, k)(1 + r_{1}(\omega^{2}k)r_{2}(\omega^{2}k)), & &  k \in \Gamma_{5}^{(2)}, \\
& \delta_{2,+}(\zeta, k) = \delta_{2,-}(\zeta, k)(1 + r_{1}(k)r_{2}(k)), & &  k \in \Gamma_{8}^{(2)}, \\
& \delta_{3,+}(\zeta, k) = \delta_{3,-}(\zeta, k)f(k), & &  k \in \Gamma_{8}^{(2)}, \\
& \delta_{4,+}(\zeta, k) = \delta_{4,-}(\zeta, k)f(k), & &  k \in \Gamma_{11}^{(2)}, \\
& \delta_{5,+}(\zeta, k) = \delta_{5,-}(\zeta, k)f(\omega^{2}k), & &  k \in \Gamma_{11}^{(2)},
\end{align*}
and such that
\begin{align}\label{IIbis delta asymp at inf}
\delta_{j}(\zeta, k) = 1 + O(k^{-1}), \qquad k \to \infty, \; j=1,\ldots,5.
\end{align}
By \cite[Lemma 2.13]{CLmain}, 
\begin{align*}
& 1 + r_{1}(\omega^{2}k)r_{2}(\omega^{2}k)>0 \mbox{ for all } k \in \overline{\Gamma_{5}^{(2)}}, \quad 1 + r_{1}(k)r_{2}(k)>0 \mbox{ for all } k \in \overline{\Gamma_{8}^{(2)}}, \\
& f(k)>0 \mbox{ for all } k \in \overline{\Gamma_{8}^{(2)}}, \; f(k)>0 \mbox{ for all } k \in \overline{\Gamma_{11}^{(2)}}\setminus\{\omega\}, \; f(\omega^{2}k)>0 \mbox{ for all } k \in \overline{\Gamma_{11}^{(2)}}\setminus\{\omega \}.
\end{align*}
Hence the solutions to the RH problems for $\delta_{1},\delta_{2},\delta_{3}$ are unique and given by \eqref{def of deltaj}. It also follows from  \cite[Lemma 2.13]{CLmain} that $f(\omega)=f(1)=0$, and therefore the solutions to the above RH problems for $\delta_{4},\delta_{5}$ are not unique. It turns out that the solutions relevant for us are those defined in \eqref{def of deltaj}.

The next lemma establishes several propreties of $\{\delta_{j}\}_{j=1}^{5}$. We omit the proof which consists of an analysis of the definitions \eqref{def of deltaj} and uses the assumption that $r_1 = 0$ on $[0,i]$.

\begin{lemma}\label{IIbis deltalemma}
The functions $\delta_{j}(\zeta, k)$, $j=1,\ldots,5$, have the following properties:
\begin{enumerate}[$(a)$]
\item The functions $\{\delta_{j}(\zeta,k)\}_{j=1}^{5}$ can be written as
\begin{align*}
& \delta_{1}(\zeta,k) = \exp \Big( -i \nu_{1} \ln_{\omega k_{4}}(k-\omega k_{4}) + i \nu_{3}  \ln_{i}(k-i) -\chi_{1}(\zeta,k) \Big), \\ 
& \delta_{2}(\zeta,k) = \exp \Big( i \nu_{2} \ln_{\omega^{2} k_{2}}(k-\omega^{2} k_{2}) -\chi_{2}(\zeta,k) \Big), \\
& \delta_{3}(\zeta,k) = \exp \Big( - i \nu_{3} \ln_{i}(k-i) + i \nu_{4} \ln_{\omega^{2} k_{2}}(k-\omega^{2} k_{2}) -\chi_{3}(\zeta,k) \Big), \\
& \delta_{4}(\zeta,k) = \exp \Big( -i \nu_{4} \ln_{\omega^{2} k_{2}}(k-\omega^{2} k_{2}) - \chi_{4}(\zeta,k) \Big), \\
& \delta_{5}(\zeta,k) = \exp \Big( -i \nu_{5} \ln_{\omega^{2} k_{2}}(k-\omega^{2} k_{2}) - \chi_{5}(\zeta,k) \Big),
\end{align*}
where the branches for the logarithms are described below \eqref{def of chij}, the functions $\{\chi_{j}\}_{j=1}^{3}$ are defined in \eqref{def of chij}, the functions $\{\chi_{j}\}_{j=4}^{5}$ are defined in the same way as $\{\tilde{\chi}_{j}\}_{j=4}^{5}$ in \eqref{def of chij} except that $\tilde{\ln}_{s}$ is replaced by $\ln_{s}$ and $\tilde{\ln}_{\omega}$ is replaced by $\ln_{\omega}$, and
\begin{align*}
& \nu_{1} := \nu_{1}(\omega^{2}k_{4}) = -\frac{1}{2\pi}\ln(1+r_{1}(k_{4})r_{2}(k_{4})), \; \nu_{2} := \nu_{2}(k_{2}) =  -\frac{1}{2\pi}\ln(1+r_{1}(\omega^{2} k_{2})r_{2}(\omega^{2} k_{2})), \\
& \nu_{3} := \nu_{3}(\omega^{2}i) = - \frac{1}{2\pi}\ln(f(i))
= - \frac{1}{2\pi}\ln(1+r_{1}(\omega^{2} i)r_{2}(\omega^{2} i)) , \\
& \nu_{4} := \nu_{4}(k_{2}) = - \frac{1}{2\pi}\ln(f(\omega^{2} k_{2})),  \quad \nu_{5} := \nu_{3}(k_{2})= - \frac{1}{2\pi} \ln(f(\omega k_{2})).
\end{align*}
\item The functions $\{\delta_{j}(\zeta,k)\}_{j=2}^{5}$ can be written as
\begin{align*}
& \delta_{2}(\zeta,k) = \exp \Big( i \nu_{2} \tilde{\ln}_{\omega^{2} k_{2}}(k-\omega^{2} k_{2}) -\tilde{\chi}_{2}(\zeta,k) \Big), \\
& \delta_{3}(\zeta,k) = \exp \Big( - i \nu_{3} \tilde{\ln}_{i}(k-i) + i \nu_{4} \tilde{\ln}_{\omega^{2} k_{2}}(k-\omega^{2} k_{2}) -\tilde{\chi}_{3}(\zeta,k) \Big), \\
& \delta_{4}(\zeta,k) = \exp \Big( -i \nu_{4} \tilde{\ln}_{\omega^{2} k_{2}}(k-\omega^{2} k_{2}) - \tilde{\chi}_{4}(\zeta,k) \Big), \\
& \delta_{5}(\zeta,k) = \exp \Big( -i \nu_{5} \tilde{\ln}_{\omega^{2} k_{2}}(k-\omega^{2} k_{2}) - \tilde{\chi}_{5}(\zeta,k) \Big),
\end{align*}
where the branches for the logarithms are described below \eqref{def of chij}, and $\tilde{\chi}_{j}(\zeta,k)$ is defined in the same way as $\chi_{j}(\zeta,k)$ except that $\ln_{s}$ is replaced by $\tilde{\ln}_{s}$ and $\ln_{\omega}$ is replaced by $\tilde{\ln}_{\omega}$.
\item For each $\zeta \in \mathcal{I}$ and $j \in \{1,\ldots,5\}$, $\delta_{j}(\zeta, k)$ and $\delta_{j}(\zeta, k)^{-1}$ are analytic functions of $k$ in their respective domains of definition. Furthermore, for any $\epsilon > 0$,
\begin{align}
& \sup_{\zeta \in \mathcal{I}} \sup_{\substack{k \in \C \setminus \Gamma_{5}^{(2)}}} |\delta_{1}(\zeta,k)^{\pm 1}| < \infty, \quad \sup_{\zeta \in \mathcal{I}} \sup_{\substack{k \in \C \setminus \Gamma_{8}^{(2)}}} |\delta_{2}(\zeta,k)^{\pm 1}| < \infty, \quad \sup_{\zeta \in \mathcal{I}} \sup_{\substack{k \in \C \setminus \Gamma_{8}^{(2)}}} |\delta_{3}(\zeta,k)^{\pm 1}| < \infty, \nonumber \\
& \sup_{\zeta \in \mathcal{I}} \sup_{\substack{k \in \C \setminus \Gamma_{11}^{(2)} \\ |k-\omega|>\epsilon}} |\delta_{4}(\zeta,k)^{\pm 1}| < \infty, \quad \sup_{\zeta \in \mathcal{I}} \sup_{\substack{k \in \C \setminus \Gamma_{11}^{(2)} \\ |k-\omega|>\epsilon}} |\delta_{5}(\zeta,k)^{\pm 1}| < \infty. \label{IIbis delta1bound}
\end{align}
\item As $k \to \omega k_{4}$ along a path which is nontangential to $\partial \mathbb{D}$, we have
\begin{align}
& |\chi_{1}(\zeta,k)-\chi_{1}(\zeta,\omega k_{4})| \leq C |k-\omega k_{4}| (1+|\ln|k-\omega k_{4}||),  \label{II asymp chi at k1}
\end{align}
and as $k \to \omega^{2} k_{2}$ along a path which is nontangential to $\partial \mathbb{D}$, we have
\begin{align}
& |\chi_{j}(\zeta,k)-\chi_{j}(\zeta,\omega^{2} k_{2})| \leq C |k-\omega^{2} k_{2}| (1+|\ln|k-\omega^{2} k_{2}||), & & j=2,3,4,5, \label{II asymp chi at omega2k2}
\end{align}
where $C$ is independent of $\zeta \in \mathcal{I}$.  
\end{enumerate}
\end{lemma}

\begin{remark}
The identity $f(i)=1+r_{1}(\omega^{2}i)r_{2}(\omega^{2}i)$, which is used in the above definition of $\nu_{3}$, is a consequence of Assumption $(iii)$ and the symmetries \eqref{r1r2 relation on the unit circle}--\eqref{r1r2 relation with kbar symmetry}. Indeed, by Assumption $(iii)$ we have $r_{1}(i)=0$. This implies by \eqref{r1r2 relation with kbar symmetry} that $r_{2}(i)=0$. We then find using \eqref{r1r2 relation on the unit circle} that $r_{1}(\omega^{2}i)+r_{2}(\frac{1}{\omega^{2}i})=0$ and $r_{1}(\frac{1}{\omega^{2}i})+r_{2}(\omega^{2}i)=0$. This implies in particular that $r_{1}(\frac{1}{\omega^{2}i})r_{2}(\frac{1}{\omega^{2}i})=r_{1}(\omega^{2}i)r_{2}(\omega^{2}i)$, from which $f(i)=1+r_{1}(\omega^{2}i)r_{2}(\omega^{2}i)$ follows.
\end{remark}

The following inequalities will be important for us.
\begin{lemma}\label{lemma: nuhat lemma IIbis}
For all $\zeta \in (0,\frac{1}{\sqrt{3}})$, the following inequalities hold:
\begin{align*}
\nu_{1} \geq 0, \qquad \hat{\nu}_{2}:=\nu_{2}+\nu_{5}-\nu_{4} \geq 0.
\end{align*}
\end{lemma}
\begin{proof}
For $\zeta \in (0,\frac{1}{\sqrt{3}})$, we have $\arg k_{4} \in (-\frac{\pi}{4},-\frac{\pi}{6})$ and $\arg k_{2} \in (-\frac{3\pi}{4},-\frac{2\pi}{3})$. Hence, the inequality for $\nu_{1}$ follows from \eqref{r1r2 relation with kbar symmetry} and the fact that $\tilde{r}(k) \leq 0$ for $k \in \{e^{i\theta}: \theta \in (-\frac{\pi}{3},\frac{\pi}{3})\}$, and the inequality for $\hat{\nu}_{2}$ follows from \cite[Lemma 2.13 $(iv)$]{CLmain}.
\end{proof}
In what follows we drop the $\zeta$-dependence of $\{\delta_{j}(\zeta,k)\}_{j=1}^{5}$ for conciseness. For $k \in \mathbb{C}\setminus \partial \mathbb{D}$, we define
\begin{align*}
\Delta_{33}(\zeta,k) = \frac{\delta_{1}(\omega^{2}k)\delta_{1}(\frac{1}{\omega k})}{\delta_{1}(k)\delta_{1}(\frac{1}{k})}    \frac{\delta_{2}(k)\delta_{2}(\frac{1}{k})}{\delta_{2}(\omega k)\delta_{2}(\frac{1}{\omega^{2}k})}    \frac{\delta_{3}(\omega^{2}k)\delta_{3}(\frac{1}{\omega k})}{\delta_{3}(k)\delta_{3}(\frac{1}{k})}     \frac{\delta_{4}(\omega^{2} k)\delta_{4}(\frac{1}{\omega k})}{\delta_{4}(\omega k)\delta_{4}(\frac{1}{\omega^{2} k})}     \frac{\delta_{5}(\omega k)\delta_{5}(\frac{1}{\omega^{2} k})}{\delta_{5}(k)\delta_{5}(\frac{1}{k})}. 
\end{align*}
This function satisfies $\Delta_{33}(\zeta,\frac{1}{k}) = \Delta_{33}(\zeta,k)$, and
\begin{subequations}\label{IIbis jumps Delta33}
\begin{align}
& \Delta_{33,+}(\zeta,k) = \Delta_{33,-}(\zeta,k)\frac{1}{1+r_{1}(\omega^{2}k)r_{2}(\omega^{2}k)} , & & k \in \Gamma_{5}^{(2)}, \\
& \Delta_{33,+}(\zeta,k) = \Delta_{33,-}(\zeta,k)\frac{1+r_{1}(k)r_{2}(k)}{f(k)}, & & k \in \Gamma_{8}^{(2)}, \\
& \Delta_{33,+}(\zeta,k) = \Delta_{33,-}(\zeta,k)\frac{1}{f(\omega^{2}k)}, & & k \in \Gamma_{11}^{(2)}.
\end{align}
\end{subequations}
Define also $\Delta_{11}(\zeta,k)=\Delta_{33}(\zeta,\omega k)$ and $\Delta_{22}(\zeta,k)=\Delta_{33}(\zeta,\omega^{2} k)$. We now define the inverse of the global parametrix by
\begin{align}\label{IIbis def of Delta}
\Delta(\zeta,k) = \begin{pmatrix}
\Delta_{11}(\zeta,k) & 0 & 0 \\
0 & \Delta_{22}(\zeta,k) & 0 \\
0 & 0 & \Delta_{33}(\zeta,k)
\end{pmatrix}, \quad \zeta \in \mathcal{I}, \; k \in \mathbb{C}\setminus \partial \mathbb{D}.
\end{align}
The function $\Delta$ satisfies
\begin{align*}
\Delta(\zeta,k) = \mathcal{A}\Delta(\zeta,\omega k)\mathcal{A}^{-1} = \mathcal{B}\Delta(\zeta,\tfrac{1}{k})\mathcal{B},
\end{align*}
and possesses the jumps
\begin{subequations}\label{IIbis jumps Delta11}
\begin{align}
& \Delta_{11,+}(\zeta,k) = \Delta_{11,-}(\zeta,k) (1+r_{1}(\omega^{2}k)r_{2}(\omega^{2}k)) , & & k \in \Gamma_{5}^{(2)}, \\
& \Delta_{11,+}(\zeta,k) = \Delta_{11,-}(\zeta,k)f(k), & & k \in \Gamma_{8}^{(2)}, \\
& \Delta_{11,+}(\zeta,k) = \Delta_{11,-}(\zeta,k)f(k), & & k \in \Gamma_{11}^{(2)},
\end{align}
\end{subequations}
and
\begin{subequations}\label{IIbis jumps Delta22}
\begin{align}
& \Delta_{22,+}(\zeta,k) = \Delta_{22,-}(\zeta,k), & & k \in \Gamma_{5}^{(2)}, \\
& \Delta_{22,+}(\zeta,k) = \Delta_{22,-}(\zeta,k) \frac{1}{1+r_{1}(k)r_{2}(k)}, & & k \in \Gamma_{8}^{(2)}, \\
& \Delta_{22,+}(\zeta,k) = \Delta_{22,-}(\zeta,k)\frac{f(\omega^{2}k)}{f(k)}, & & k \in \Gamma_{11}^{(2)}.
\end{align}
\end{subequations}
The next lemma follows from \eqref{IIbis def of Delta} and Lemma \ref{IIbis deltalemma}.

\begin{lemma}\label{Deltalemma}
For any $\epsilon>0$, $\Delta(\zeta,k)$ and $\Delta(\zeta,k)^{-1}$ are uniformly bounded for $k \in \mathbb{C}\setminus$ $ (\cup_{j=0}^{2}D_{\epsilon}(\omega^{j}) \cup \partial \mathbb{D})$ and $\zeta \in \mathcal{I}$. Furthermore, $\Delta(\zeta,k)=I+O(k^{-1})$ as $k \to \infty$.
\end{lemma}

\section{The $n^{(2)}\to n^{(3)}$ transformation}
Let $n^{(3)}$ be the piecewise analytic function given by
\begin{align}\label{IIbis def of mp3p}
n^{(3)}(x,t,k) = n^{(2)}(x,t,k)\Delta(\zeta,k), \qquad  k \in \mathbb{C}\setminus \Gamma^{(3)},
\end{align}
where $\Gamma^{(3)}=\Gamma^{(2)}$. It satisfies $n_{+}^{(3)}=n_{-}^{(3)}v^{(3)}$ on $\Gamma^{(3)}$, where $v^{(3)}=\Delta_{-}^{-1}v^{(2)}\Delta_{+}$.

Let $\mathcal{X}_{1}^{\epsilon}$ and $\mathcal{X}_{2}^{\epsilon}$ be small crosses centered at $\omega k_4$ and $\omega^2 k_2$, respectively. More precisely, $\mathcal{X}_{1}^{\epsilon}:= \cup_{j=1}^{4}\mathcal{X}_{1,j}^{\epsilon}$, where
\begin{align*}
&\mathcal{X}_{1,1}^{\epsilon} = \Gamma_{1}^{(3)}\cap D_{\epsilon}(\omega k_{4}), && \mathcal{X}_{1,2}^{\epsilon} = \Gamma_{4}^{(3)}\cap D_{\epsilon}(\omega k_{4}), 
	\\
& \mathcal{X}_{1,3}^{\epsilon} = \Gamma_{6}^{(3)}\cap D_{\epsilon}(\omega k_{4}), && \mathcal{X}_{1,4}^{\epsilon} = \Gamma_{3}^{(3)}\cap D_{\epsilon}(\omega k_{4}),
\end{align*}
are oriented outwards from $\omega k_4$, and $\mathcal{X}_{2}^{\epsilon}:= \cup_{j=1}^{4}\mathcal{X}_{2,j}^{\epsilon}$, where
\begin{align*}
& \mathcal{X}_{2,1}^{\epsilon} = \Gamma_{7}^{(3)}\cap D_{\epsilon}(\omega^{2} k_{2}), && \mathcal{X}_{2,2}^{\epsilon} = \Gamma_{10}^{(3)}\cap D_{\epsilon}(\omega^{2} k_{2}), 
	\\
& \mathcal{X}_{2,3}^{\epsilon} = \Gamma_{12}^{(3)}\cap D_{\epsilon}(\omega^{2} k_{2}), && \mathcal{X}_{2,4}^{\epsilon} = \Gamma_{9}^{(3)}\cap D_{\epsilon}(\omega^{2} k_{2}),
\end{align*}
are oriented outwards from $\omega^{2}k_{2}$. Let $\hat{\mathcal{X}}^\epsilon = \bigcup_{j=1}^{2}\big(\mathcal{X}_{j}^\epsilon \cup \omega \mathcal{X}_{j}^\epsilon \cup \omega^2 \mathcal{X}_{j}^\epsilon \cup (\mathcal{X}_{j}^\epsilon)^{-1} \cup (\omega \mathcal{X}_{j}^\epsilon)^{-1} \cup (\omega^2 \mathcal{X}_{j}^\epsilon)^{-1}\big)$ be the union of all the small crosses centered at the twelve saddle points.

\begin{lemma}\label{II v3lemma}
$v^{(3)}$ converges to the identity matrix $I$ as $t \to \infty$ uniformly for $\zeta \in \mathcal{I}$ and $k \in \Gamma^{(3)}\setminus \hat{\mathcal{X}}^\epsilon$. More precisely, for $\zeta \in \mathcal{I}$,
\begin{align}\label{II v3estimatesa}
& \|v^{(3)} - I\|_{(L^1 \cap L^\infty)(\Gamma^{(3)}\setminus \hat{\mathcal{X}}^\epsilon)} \leq Ct^{-1}.
\end{align}
\end{lemma}
\begin{proof}
From the expressions for $v_j^{(1)}$, $j = 2,5,8, 11$, in \eqref{vp1p 123}, \eqref{vp1p 456}, \eqref{vp1p 789}, and \eqref{vp1p 101112} together with \eqref{II jumps vj diagonal p2p}, \eqref{IIbis jumps Delta33}, \eqref{IIbis jumps Delta11}, \eqref{IIbis jumps Delta22}, we deduce that
\begin{align}\label{lol3}
& v_j^{(3)} = \Delta_{-}^{-1}v_j^{(1)}\Delta_{+} = I + \Delta_{-}^{-1}v_{j,r}^{(1)}\Delta_{+}, \quad j = 2,5,8,11.
\end{align}
By Lemma \ref{IIbis decompositionlemma} and Lemma \ref{Deltalemma}, the matrices $\Delta_{-}^{-1}v_{j,r}^{(1)}\Delta_{+}$, $j = 2,5,8,11$, are small as $ t \to \infty$.
Using also Lemmas \ref{vsmallnearilemmaV} and \ref{symmetryjumpslemma}, we infer that the $L^1$ and $L^\infty$ norms of $v^{(3)}-I$ tend to $0$ as $t \to \infty$, uniformly for $k \in \mathsf{S}\cap \Gamma^{(3)}\setminus \hat{\mathcal{X}}^\epsilon$. Then \eqref{II v3estimatesa} follows from the symmetries \eqref{vjsymm}.
\end{proof}

To define the next transformation $n^{(3)}\to \hat{n}$, we need to construct local parametrices near $\{k_{j},\omega k_{j},\omega^{2}k_{j}\}_{j=1}^{4}$. Thanks to the $\mathcal{A}$- and $\mathcal{B}$-symmetries, we can focus on the construction of the parametrices near $\omega k_{4}$ and $\omega^{2}k_{2}$. The construction of these two parametrices is the content of Sections \ref{section:loc param 1} and \ref{section:loc param 2}. We will then show in Section \ref{section:small norm} that these parametrices are good approximations of $n^{(3)}$ in $D_{\epsilon}(\omega k_{4})$ and $D_{\epsilon}(\omega^{2}k_{2})$.

\section{Local parametrix near $\omega k_{4}$}\label{section:loc param 1}

As $k \to \omega k_{4}$, we have
\begin{align*}
& -(\Phi_{31}(\zeta,k)-\Phi_{31}(\zeta,\omega k_{4})) =  \phi_{\omega k_{4}}(k-\omega k_{4})^{2} + O((k-\omega k_{4})^{3}), \;\; \phi_{\omega k_{4}} :=\omega \frac{4-3 k_{4} \zeta - k_{4}^{3} \zeta}{4 k_{4}^{4}}.
\end{align*}
Let $z_{1}=z_{1}(\zeta,t,k)$ be given by
\begin{align*}
z_{1}=z_{1,\star}\sqrt{t} (k-\omega k_{4}) \hat{z}_{1}, \qquad \hat{z}_{1} = \sqrt{\frac{2i(\Phi_{31}(\zeta,\omega k_{4})-\Phi_{31}(\zeta,k))}{z_{1,\star}^{2}(k-\omega k_{4})^{2}}},
\end{align*}
where the principal branch is chosen for $\hat{z}_{1}=\hat{z}_{1}(\zeta,k)$, and
\begin{align*}
z_{1,\star} = \sqrt{2} e^{\frac{\pi i}{4}} \sqrt{\phi_{\omega k_{4}}}, \qquad -i\omega k_{4}z_{1,\star}>0.
\end{align*}
We have $\hat{z}_{1}(\zeta,\omega k_{4})=1$ and $t(\Phi_{31}(\zeta, k) -\Phi_{31}(\zeta,\omega k_{4})) = \frac{iz_{1}^{2}}{2}$. Let $\epsilon>0$ be fixed and small enough that the map $z_{1}$ is conformal from $D_\epsilon(\omega k_4)$ to a neighborhood of $0$. The asymptotic behavior of $z_{1}$ as $k \to \omega k_{4}$ is given by
\begin{align*}
z_{1} = z_{1,\star}\sqrt{t}(k-\omega k_{4})(1+O(k-\omega k_{4})).
\end{align*}
For all $k \in D_\epsilon(\omega k_4)$, the following identity holds:
\begin{align*}
\ln_{\omega k_{4}}(k-\omega k_{4}) & = \ln[z_{1,\star}(k-\omega k_{4})\hat{z}_{1}]- \ln \hat{z}_{1} -\ln z_{1,\star} + 2\pi i,
\end{align*}
where $\ln$ is the principal logarithm. Shrinking $\epsilon > 0$ if necessary, the function $k \mapsto \ln \hat{z}_{1}$ is analytic in $D_{\epsilon}(\omega k_{4})$, $\ln \hat{z}_{1} = O(k-\omega k_{4})$ as $k \to \omega k_{4}$, and
\begin{align*}
\ln z_{1,\star} = \ln |z_{1,\star}| + i \arg z_{1,\star} = \ln |z_{1,\star}| + i \big( \tfrac{\pi}{2}-\arg (\omega k_{4}) \big),
\end{align*}
where $\arg(\omega k_{4})\in (\tfrac{\pi}{3},\tfrac{\pi}{2})$.
Using Lemma \ref{IIbis deltalemma}, we obtain
\begin{align*}
& \delta_{1}(\zeta,k) = e^{-i\nu_{1} \ln z_{1}} t^{\frac{i \nu_{1}}{2}} e^{i \nu_{1} \ln \hat{z}_{1}} e^{i \nu_{1} \ln z_{1,\star}} e^{i\nu_{3}\ln_{i}(k-i)} e^{-\chi_{1}(\zeta,k)}e^{2\pi \nu_{1}}. 
\end{align*}
For conciseness, we from now on, for $a \in \C$, use the notation $z_{1}^{a} := e^{a \ln z_{1}}$, $z_{1,\star}^{a} := e^{a \ln z_{1,\star}}$, and $\hat{z}_{1}^{a} := e^{a \ln \hat{z}_{1}}$. Using \eqref{IIbis def of Delta} and the above expression for $\delta_1(\zeta,k)$, we get
\begin{align}
& \frac{\Delta_{33}(\zeta,k)}{\Delta_{11}(\zeta,k)} = \frac{\mathcal{D}_{1}(k)}{\delta_1(\zeta, k)^2}
= \tilde{d}_{1,0}(\zeta,t)d_{1,1}(\zeta,k)z_{1}^{2i \nu_{1}}, \label{IIbis Delta33 Delta11} \\
& \tilde{d}_{1,0}(\zeta,t) = e^{-4\pi\nu_{1}}e^{2\chi_{1}(\zeta,\omega k_{4})} e^{-2i \nu_{3} \ln_{i}(\omega k_{4}-i)} t^{-i \nu_{1}} z_{1,\star}^{-2i\nu_{1}} \mathcal{D}_{1}(\omega k_{4}), \nonumber \\
& d_{1,1}(\zeta,k) = e^{2\chi_{1}(\zeta,k)-2\chi_{1}(\zeta,\omega k_{4})} e^{-2i\nu_{3}\big[\ln_{i}(k - i)-\ln_{i}(\omega k_{4}-i)\big]} \hat{z}_{1}^{-2i\nu_{1}} \frac{\mathcal{D}_{1}(k)}{\mathcal{D}_{1}(\omega k_{4})}, \nonumber \\
& \mathcal{D}_{1}(k) = \frac{\delta_{1}(\omega^{2} k)\delta_{1}(\frac{1}{\omega k})^{2}\delta_{1}(\omega k)}{\delta_{1}(\frac{1}{k})\delta_{1}(\frac{1}{\omega^{2}k})} \frac{\delta_{2}(k)\delta_{2}(\omega^{2}k)\delta_{2}(\frac{1}{k})^{2}}{\delta_{2}(\omega k)^{2} \delta_{2}(\frac{1}{\omega^{2}k})\delta_{2}(\frac{1}{\omega k})} \frac{\delta_{3}(\omega k) \delta_{3}(\omega^{2} k) \delta_{3}(\frac{1}{\omega k})^{2}}{\delta_{3}(k)^{2}\delta_{3}(\frac{1}{k})\delta_{3}(\frac{1}{\omega^{2} k})} \nonumber \\
& \hspace{1.5cm} \times \frac{\delta_{4}(\omega^{2}k)^{2} \delta_{4}(\frac{1}{k})\delta_{4}(\frac{1}{\omega k})}{\delta_{4}(k)\delta_{4}(\omega k)\delta_{4}(\frac{1}{\omega^{2} k})^{2}} \frac{\delta_{5}(\omega k)^{2} \delta_{5}(\frac{1}{\omega k})\delta_{5}(\frac{1}{\omega^{2}k})}{\delta_{5}(k) \delta_{5}(\frac{1}{k})^{2}\delta_{5}(\omega^{2}k)}.
\end{align}
Define $Y_1(\zeta,t)$ by
\begin{align*}
Y_{1}(\zeta,t) = \tilde{d}_{1,0}(\zeta,t)^{\frac{\sigma}{2}}e^{-\frac{t}{2}\Phi_{31}(\zeta,\omega k_{4})\sigma}\lambda_{1}^{\sigma},
\end{align*}
where $\sigma = \diag (1,0,-1)$ and $\lambda_{1}$ is a free parameter that will be fixed later. Let
\begin{align*}
\tilde{n}(x,t,k) = n^{(3)}(x,t,k)Y_{1}(\zeta,t), \qquad k \in D_\epsilon(\omega k_4),
\end{align*}
let $\tilde{v}$ be the jump matrix of $\tilde{n}$, and let $\tilde{v}_{j}$ be the restriction of $\tilde{v}$ to $\Gamma_{j}^{(3)}\cap D_\epsilon(\omega k_4)$. By \eqref{Vv2def}, \eqref{IIbis def of mp3p}, and \eqref{IIbis Delta33 Delta11}, we can write
\begin{align*}
& \tilde{v}_{1} = \begin{pmatrix}
1 & 0 & 0 \\
0 & 1 & 0 \\
-d_{1,1}^{-1}\lambda_{1}^{2}r_{1,a}(\omega^{2} k)z_{1}^{-2i \nu_{1}} e^{\frac{iz_{1}^{2}}{2}} & 0 & 1
\end{pmatrix}, \; \tilde{v}_{3} = \begin{pmatrix}
1 & 0 & -d_{1,1}\lambda_{1}^{-2} r_{2,a}(\omega^{2}k) z_{1}^{2i \nu_{1}} e^{-\frac{iz_{1}^{2}}{2}} \\
0 & 1 & 0 \\
0 & 0 & 1
\end{pmatrix}, \nonumber \\
& \tilde{v}_{4} = \begin{pmatrix}
1 & 0 & d_{1,1}\lambda_{1}^{-2} \hat{r}_{2,a}(\omega^{2}k) z_{1}^{2i \nu_{1}} e^{-\frac{iz_{1}^{2}}{2}} \\
0 & 1 & 0 \\
0 & 0 & 1
\end{pmatrix}, \; \tilde{v}_{6} = \begin{pmatrix}
1 & 0 & 0 \\
0 & 1 & 0 \\
d_{1,1}^{-1}\lambda_{1}^{2} \hat{r}_{1,a}(\omega^{2}k) z_{1}^{-2i \nu_{1}} e^{\frac{iz_{1}^{2}}{2}} & 0 & 1 
\end{pmatrix}. 
\end{align*}
The matrices $\tilde{v}_{j}$, $j=1,3,4,6$, suggest to approximate $\tilde{n}$ by $(1,1,1)Y_{1}\tilde{m}^{\omega k_{4}}$, where $\tilde{m}^{\omega k_{4}}(x,t,k)$ is the solution of the $3\times 3$ RH problem with $\tilde{m}^{\omega k_{4}}(x,t,k) \to I$ as $t \to \infty$ uniformly for $k\in \partial D_{\epsilon}(\omega k_{4})$ and whose jumps on $\mathcal{X}_{1}^{\epsilon}$ are
\begin{align}\nonumber
& \tilde{v}_{\mathcal{X}_{1,1}^{\epsilon}}^{\omega k_{4}}=\begin{pmatrix}
1 & 0 & 0 \\
0 & 1 & 0 \\
-\lambda_{1}^{2}r_{1}(\omega^{2} k_{\star})z_{1}^{-2i \nu_{1}} e^{\frac{iz_{1}^{2}}{2}} & 0 & 1
\end{pmatrix}, & & \tilde{v}_{\mathcal{X}_{1,2}^{\epsilon}}^{\omega k_{4}} =\begin{pmatrix}
1 & 0 & \lambda_{1}^{-2} \hat{r}_{2}(\omega^{2}k_{\star}) z_{1}^{2i \nu_{1}} e^{-\frac{iz_{1}^{2}}{2}} \\
0 & 1 & 0 \\
0 & 0 & 1
\end{pmatrix}, 
	\\ \label{vtildeomegak4def}
& \tilde{v}_{\mathcal{X}_{1,3}^{\epsilon}}^{\omega k_{4}} = \begin{pmatrix}
1 & 0 & 0 \\
0 & 1 & 0 \\
\lambda_{1}^{2} \hat{r}_{1}(\omega^{2}k_{\star}) z_{1}^{-2i \nu_{1}} e^{\frac{iz_{1}^{2}}{2}} & 0 & 1 
\end{pmatrix}, & & \tilde{v}_{\mathcal{X}_{1,4}^{\epsilon}}^{\omega k_{4}} = \begin{pmatrix}
1 & 0 & -\lambda_{1}^{-2} r_{2}(\omega^{2}k_{\star}) z_{1}^{2i \nu_{1}} e^{-\frac{iz_{1}^{2}}{2}} \\
0 & 1 & 0 \\
0 & 0 & 1
\end{pmatrix},
\end{align}
where $k_{\star}=\omega k_{4}$. 
We fix the free parameter $\lambda_{1}$ as follows:
\begin{align}\label{def of lambda1}
\lambda_{1} = |\tilde{r}(k_{4})|^{\frac{1}{4}} = |\tilde{r}(\tfrac{1}{k_{4}})|^{-\frac{1}{4}}.
\end{align}

\begin{lemma}\label{IIbis lemma: bound on Y}
The function $Y_{1}(\zeta,t)$ is uniformly bounded:
\begin{align}\label{IIbis Ybound}
\sup_{\zeta \in \mathcal{I}} \sup_{t \geq 2} | Y_{1}(\zeta,t)^{\pm 1}| \leq C.
\end{align}
Moreover, the functions $\tilde{d}_{1,0}(\zeta, t)$ and $d_{1,1}(\zeta, k)$ satisfy
\begin{align}\label{IIbis d0estimate}
& |\tilde{d}_{1,0}(\zeta, t)| = 1, & & \zeta \in \mathcal{I}, \ t \geq 2,\\
& |d_{1,1}(\zeta, k) - 1| \leq C |k - \omega k_4| (1+ |\ln|k-\omega k_4||), & & \zeta \in \mathcal{I}, \ k \in \mathcal{X}_{1}^{\epsilon}. \label{IIbis d1estimate}
\end{align}
\end{lemma}
\begin{proof}
Standard estimates yield (\ref{IIbis Ybound}) and (\ref{IIbis d1estimate}). Let us show (\ref{IIbis d0estimate}). Observe that
\begin{align}\label{d10absolutevalue}
 |\tilde{d}_{1,0}(\zeta,t)| = &\; e^{-4\pi\nu_{1}}  e^{2\re \chi_{1}(\zeta,\omega k_{4})}   e^{2\nu_3\arg_i(\omega k_{4}- i)} 
e^{2 \nu_{1} \arg z_{1,\star}} |\mathcal{D}_{1}(\omega k_{4})|.
\end{align} 
We infer from Lemma \ref{IIbis deltalemma} that, for $k \in \partial \D \setminus \{\arg k \in (\arg(\omega k_4), \pi/2)\}$,
\begin{align}\label{absdelta1}
 |\delta_{1}(\zeta,k)| = &\; e^{\nu_{1} \frac{\arg_i(k) + \arg(\omega k_4) + \pi}{2}
 - \nu_{3} \frac{\arg_i(k) + \arg(i) + \pi}{2}
  - \re \chi_{1}(\zeta,k) }
\end{align}
where $\arg_i(k) \in (\pi/2, 5\pi/2)$. It follows from the definition of $\chi_1$ in (\ref{def of chij}) that the right-hand side of (\ref{absdelta1}) is independent of $k \in \partial \D \setminus \{\arg k \in (\arg(\omega k_4), \pi/2)\}$.
Similar calculations show that $|\delta_{j}(\zeta,k)|$, $j = 2, 3, 4,5$, are independent of $k$ for $k \in \partial \D \setminus \{\arg k \in (\pi/2, 2\pi/3)\}$.
Thus
\begin{align*}
|\mathcal{D}_{1}(\omega k_4)| = &\; |\delta_1(\zeta, \omega k_4)|^2
= e^{2\nu_{1} \frac{2\arg(\omega k_4) + 3\pi}{2}
 - 2\nu_{3} \frac{\arg(\omega k_4) + \arg(i) + 3\pi}{2}
  - 2\re \chi_{1}(\zeta,\omega k_4) }.
\end{align*}
Substituting this expression for $|\mathcal{D}_{1}(\omega k_4)|$ as well as the expressions $\arg_i(\omega k_{4}- i) = \frac{\arg(\omega k_{4}) + \arg(i) + 3\pi}{2} $ and
$\arg z_{1,\star} = \tfrac{\pi}{2}-\arg (\omega k_{4})$ into (\ref{d10absolutevalue}), we obtain $|\tilde{d}_{1,0}(\zeta,t)| = 1$.
\end{proof}

The local parametrix $m^{\omega k_{4}}$ is defined by
\begin{align}\label{IIbis m omegak4 def}
m^{\omega k_{4}}(x,t,k) = Y_{1}(\zeta,t)m^{X,(1)}(q,z_{1}(\zeta,t,k))Y_{1}(\zeta,t)^{-1}, \qquad k \in D_{\epsilon}(\omega k_{4}),
\end{align}
where $m^{X,(1)}$ is the solution to the model RH problem of Lemma \ref{IIbis Xlemma 3} with 
\begin{align*}
& q = |\tilde{r}(k_{4})|^{\frac{1}{2}}r_{1}(k_{4}), \qquad \bar{q} = -|\tilde{r}(k_{4})|^{-\frac{1}{2}}r_{2}(k_{4}).
\end{align*}
Note that, for $\zeta \in \mathcal{I}$, we have $|q| < 1$ and
$$q = \lambda_1^2 r_1(\omega^2 k_\star), \quad
\frac{-\bar{q}}{1 - |q|^2} = \lambda_1^{-2} \hat{r}_2(\omega^2 k_\star), \quad
\bar{q} = -\lambda_1^{-2} r_2(\omega^2 k_\star), \quad
\frac{q}{1 - |q|^2} = \lambda_1^2\hat{r}_1(\omega^2 k_\star),$$
so that the jump matrices in Lemma \ref{IIbis Xlemma 3} agree with those in (\ref{vtildeomegak4def}) if $z = z_1$ and $\nu = \nu_1$.

\begin{lemma}\label{IIbis k0lemma}
For each $t \geq 2$ and $\zeta \in \mathcal{I}$, $m^{\omega k_4}(x,t,k)$ defined in \eqref{IIbis m omegak4 def} is analytic for $k \in D_\epsilon(\omega k_4) \setminus \mathcal{X}_{1}^\epsilon$. Moreover, $m^{\omega k_4}(x,t,k)$ is uniformly bounded for $t \geq 2$, $\zeta \in \mathcal{I}$, and $k \in D_\epsilon(\omega k_4) \setminus \mathcal{X}_{1}^\epsilon$. On $\mathcal{X}_{1}^\epsilon$, $m^{\omega k_4}$ satisfies $m_+^{\omega k_4} =  m_-^{\omega k_4} v^{\omega k_4}$, where the jump matrix $v^{\omega k_4}$ obeys
\begin{align}\label{IIbis v3vk0estimate}
\begin{cases}
 \| v^{(3)} - v^{\omega k_4} \|_{L^1(\mathcal{X}_{1}^\epsilon)} \leq C t^{-1} \ln t,
	\\
\| v^{(3)} - v^{\omega k_4} \|_{L^\infty(\mathcal{X}_{1}^\epsilon)} \leq C t^{-1/2} \ln t,
\end{cases} \qquad \zeta \in \mathcal{I}, \ t \geq 2.
\end{align}
Furthermore, as $t \to \infty$,
\begin{align}\label{IIbis mmodmuestimate2}
& \| m^{\omega k_4}(x,t,\cdot) - I \|_{L^\infty(\partial D_\epsilon(\omega k_4))} = O(t^{-1/2}),
	\\ \label{IIbis mmodmuestimate1}
& m^{\omega k_4} - I = \frac{Y_{1}(\zeta, t)m_1^{X,(1)} Y_{1}(\zeta, t)^{-1}}{z_{1,\star}\sqrt{t} (k-\omega k_4) \hat{z}_{1}(\zeta,k)} + O(t^{-1})
\end{align}
uniformly for $\zeta \in \mathcal{I}$, and $m_1^{X,(1)}=m_1^{X,(1)}(q)$ is given by \eqref{IIbis mXasymptotics 3}.
\end{lemma}
\begin{proof}
The proof is almost identical to \cite[Lemma 9.10]{CLmain}, so we omit it.
\end{proof}

\section{Local parametrix near $\omega^{2}k_{2}$}\label{section:loc param 2}
The construction of this local parametrix is similar to the one in \cite[Section 9.6]{CLmain} (the difference is that the functions $\{\delta_{j}\}_{j=1}^{5}$, $d_{2,0}, d_{2,1}, \mathcal{D}_{2}$ appearing in \cite[Section 9.6]{CLmain} need to be defined differently here). For convenience, we provide details of the construction. As $k \to \omega^{2} k_{2}$, we have
\begin{align*}
& -\hspace{-0.1cm}(\Phi_{32}(\zeta,k)-\Phi_{32}(\zeta,\omega^{2} k_{2})) \hspace{-0.05cm}= \hspace{-0.05cm} \phi_{\omega^{2} k_{2}}(k\hspace{-0.05cm}-\hspace{-0.05cm}\omega^{2} k_{2})^{2} \hspace{-0.05cm}+\hspace{-0.05cm} O((k\hspace{-0.05cm}-\hspace{-0.05cm}\omega^{2} k_{2})^{3}), 
 	\\
&\phi_{\omega^{2} k_{2}} := -\omega^{2} \frac{4-3 k_{2} \zeta - k_{2}^{3} \zeta}{4 k_{2}^{4}}.
\end{align*}
Let $z_{2}=z_{2}(\zeta,t,k)$ be given by
\begin{align*}
z_{2}=z_{2,\star}\sqrt{t} (k-\omega^{2} k_{2}) \hat{z}_{2}, \qquad \hat{z}_{2} = \sqrt{\frac{2i(\Phi_{32}(\zeta,\omega^{2} k_{2})-\Phi_{32}(\zeta,k))}{z_{2,\star}^{2}(k-\omega^{2} k_{2})^{2}}},
\end{align*}
where the principal branch is chosen for $\hat{z}_{2}=\hat{z}_{2}(\zeta,k)$ and 
\begin{align*}
z_{2,\star} = \sqrt{2} e^{\frac{\pi i}{4}} \sqrt{\phi_{\omega^{2} k_{2}}}, \qquad -i\omega^{2} k_{2}z_{2,\star}>0.
\end{align*}
We have $\hat{z}_{2}(\zeta,\omega^{2} k_{2})=1$ and $t(\Phi_{32}(\zeta, k) - \Phi_{32}(\zeta,\omega^{2} k_{2})) = \frac{iz_{2}^{2}}{2}$. Let $\epsilon > 0$ be fixed and small enough that the map $z_{2}$ is conformal from $D_\epsilon(\omega^{2} k_2)$ to a neighborhood of $0$. As $k \to \omega^{2} k_{2}$,
\begin{align*}
z_{2} = z_{2,\star}\sqrt{t}(k-\omega^{2} k_{2})(1+O(k-\omega^{2} k_{2})).
\end{align*}
For all $k \in D_\epsilon(\omega^{2} k_2)$, the following identities hold:
\begin{align*}
\ln_{\omega^{2} k_2}(k-\omega^{2} k_2) & = \ln_{0}[z_{2,\star}(k-\omega^{2} k_2)\hat{z}_{2}]- \ln \hat{z}_{2} -\ln z_{2,\star}, \\
\tilde{\ln}_{\omega^{2} k_2}(k-\omega^{2} k_2) & = \ln[z_{2,\star}(k-\omega^{2} k_2)\hat{z}_{2}]- \ln \hat{z}_{2} -\ln z_{2,\star},
\end{align*}
where $\ln_{0}(k):= \ln|k|+i\arg_{0}k$, $\arg_{0}(k)\in (0,2\pi)$, and $\ln$ is the principal logarithm. Shrinking $\epsilon > 0$ if necessary, the function $k \mapsto \ln \hat{z}_{2}$ is analytic in $D_{\epsilon}(\omega^{2} k_2)$, $\ln \hat{z}_{2} = O(k-\omega^{2} k_2)$ as $k \to \omega^{2} k_2$, and
\begin{align*}
\ln z_{2,\star} = \ln |z_{2,\star}| + i \arg z_{2,\star} = \ln |z_{2,\star}| + i \big( \tfrac{\pi}{2}-\arg (\omega^{2} k_2) \big), 
\end{align*}
where $\arg (\omega^{2} k_2) \in (\tfrac{\pi}{2},\tfrac{2\pi}{3})$. Using Lemma \ref{IIbis deltalemma}, we obtain
\begin{align*}
& \delta_{2}(\zeta,k) = e^{-\chi_{2}(\zeta,k)}t^{-\frac{i\nu_{2}}{2}}z_{2,(0)}^{i\nu_{2}}\hat{z}_{2}^{-i\nu_{2}}z_{2,\star}^{-i\nu_{2}}, \\
& \delta_{3}(\zeta,k) = e^{-i\nu_{3}\ln_{i}(k-i)}e^{-\chi_{3}(\zeta,k)}t^{-\frac{i\nu_{4}}{2}}z_{2,(0)}^{i\nu_{4}}\hat{z}_{2}^{-i\nu_{4}}z_{2,\star}^{-i\nu_{4}}, \\
& \delta_{4}(\zeta,k) = t^{\frac{i\nu_{4}}{2}} z_{2}^{-i\nu_{4}} \hat{z}_{2}^{i\nu_{4}} z_{2,\star}^{i\nu_{4}}e^{-\tilde{\chi}_{4}(\zeta,k)}, \\
& \delta_{5}(\zeta,k) = t^{\frac{i\nu_{5}}{2}}z_{2}^{-i\nu_{5}}\hat{z}_{2}^{i\nu_{5}}z_{2,\star}^{i\nu_{5}}e^{-\tilde{\chi}_{5}(\zeta,k)},
\end{align*}
where we have used the notation $z_{2,(0)}^{a} := e^{a \ln_{0} z_{2}}$, $\hat{z}_{2}^{a} := e^{a \ln \hat{z}_{2}}$, $z_{2,\star}^{a} := e^{a \ln z_{2,\star}}$, $z_{2}^{a} := e^{a \ln z_{2}}$ for $a \in \C$. Using \eqref{IIbis def of Delta} and the above expressions for $\{\delta_{j}(\zeta,k)\}_{j=2}^{5}$, we get
\begin{align}\label{IIbis Delta33 Delta22}
& \frac{\Delta_{33}}{\Delta_{22}}
= \frac{\delta_2(k)^2 \delta_4(k) }{\delta_3(k) \delta_5(k)^2} \mathcal{D}_2(k) 
= \tilde{d}_{2,0}(\zeta,t)d_{2,1}(\zeta,k)z_{2}^{i (2\nu_{5}-\nu_{4})} z_{2,(0)}^{i (2\nu_{2}-\nu_{4})}, \\
& \tilde{d}_{2,0}(\zeta,t) = e^{-2\chi_{2}(\zeta,\omega^{2} k_{2}) + \chi_{3}(\zeta,\omega^{2} k_{2}) - \tilde{\chi}_{4}(\zeta,\omega^{2} k_{2}) + 2\tilde{\chi}_{5}(\zeta,\omega^{2} k_{2}) } \nonumber \\
& \hspace{1.6cm} \times e^{i\nu_{3}\ln_{i}(\omega^{2} k_{2}-i)} t^{i (\nu_{4}-\nu_{5}-\nu_{2})} z_{2,\star}^{2i(\nu_{4}-\nu_{5}-\nu_{2})} \mathcal{D}_{2}(\omega^{2} k_{2}), \nonumber \\
& d_{2,1}(\zeta,k) = e^{-2\chi_{2}(\zeta,k)+2\chi_{2}(\zeta,\omega^{2} k_{2}) + \chi_{3}(\zeta,k)-\chi_{3}(\zeta,\omega^{2} k_{2}) - \tilde{\chi}_{4}(\zeta,k)+\tilde{\chi}_{4}(\zeta,\omega^{2} k_{2}) + 2\tilde{\chi}_{5}(\zeta,k)-2\tilde{\chi}_{5}(\zeta,\omega^{2} k_{2}) } \nonumber \\
& \hspace{1.6cm} \times e^{i\nu_{3}[\ln_{i}(k-i)-\ln_{i}(\omega^{2} k_{2}-i)]} \hat{z}_{2}^{2i(\nu_{4}-\nu_{5}-\nu_{2})} \frac{\mathcal{D}_{2}(k)}{\mathcal{D}_{2}(\omega^{2} k_{2})}, \nonumber \\
& \mathcal{D}_{2}(k) = \frac{\delta_{1}(\omega^{2} k)^{2}  \delta_{1}(\frac{1}{\omega k})\delta_{1}(\frac{1}{\omega^{2} k})}{\delta_{1}(k)\delta_{1}(\frac{1}{k})^{2}\delta_{1}(\omega k)} \frac{\delta_{2}(\frac{1}{k})\delta_{2}(\frac{1}{\omega k})}{\delta_{2}(\omega k) \delta_{2}(\frac{1}{\omega^{2}k})^{2}\delta_{2}(\omega^{2} k)} \frac{\delta_{3}(\omega^{2} k)^{2}  \delta_{3}(\frac{1}{\omega k})\delta_{3}(\frac{1}{\omega^{2} k})}{\delta_{3}(\frac{1}{k})^{2}\delta_{3}(\omega k)} \nonumber \\
& \hspace{1.5cm} \times \frac{\delta_{4}(\omega^{2}k) \delta_{4}(\frac{1}{\omega k})^{2}}{\delta_{4}(\omega k)^{2}\delta_{4}(\frac{1}{\omega^{2} k})\delta_{4}(\frac{1}{k})} \frac{\delta_{5}(\omega k) \delta_{5}(\frac{1}{\omega^{2}k})^{2}\delta_{5}(\omega^{2} k)}{\delta_{5}(\frac{1}{k}) \delta_{5}(\frac{1}{\omega k})}.
\end{align}
Define
\begin{align*}
Y_{2}(\zeta,t) = \tilde{d}_{2,0}(\zeta,t)^{\frac{\tilde{\sigma}}{2}}e^{-\frac{t}{2}\Phi_{32}(\zeta,\omega^{2}k_{2})\tilde{\sigma}}\lambda_{2}^{\tilde{\sigma}},
\end{align*}
where $\tilde{\sigma} = \diag (0,1,-1)$ and $\lambda_{2}$ is a free parameter that will be fixed later. Let
\begin{align*}
\tilde{n}(x,t,k) = n^{(3)}(x,t,k)Y_{2}(\zeta,t), \qquad k \in D_\epsilon(\omega^{2} k_2),
\end{align*}
let $\tilde{v}$ be the jump matrix of $\tilde{n}$, and let $\tilde{v}_{j}$ be the restriction of $\tilde{v}$ to $\Gamma_{j}^{(3)}\cap D_{\epsilon}(\omega^{2}k_{2})$. By \eqref{Vv2def}, \eqref{IIbis def of mp3p}, and \eqref{IIbis Delta33 Delta22}, we can write
\begin{align*}
& \tilde{v}_{7}^{-1} = \begin{pmatrix}
1 & 0 & 0 \\
0 & 1 & 0 \\
0 & -d_{2,1}^{-1}\lambda_{2}^{2} \big( \frac{r_{1}(\frac{1}{\omega k})-r_{1}(k)r_{1}(\omega^{2}k)}{1+r_{1}(k)r_{2}(k)} \big)_{a} z_{2}^{-i (2\nu_{5}-\nu_{4})} z_{2,(0)}^{-i (2\nu_{2}-\nu_{4})} e^{\frac{iz_{2}^{2}}{2}} & 1
\end{pmatrix}, \\
& \tilde{v}_{10} = \begin{pmatrix}
1 & 0 & 0 \\
0 & 1 & d_{2,1}\lambda_{2}^{-2} \big( \frac{r_{2}(\frac{1}{\omega k})-r_{2}(k)r_{2}(\omega^{2}k)}{f(\omega^{2}k)} \big)_{a} z_{2}^{i (2\nu_{5}-\nu_{4})} z_{2,(0)}^{i (2\nu_{2}-\nu_{4})} e^{-\frac{iz_{2}^{2}}{2}} \\
0 & 0 & 1
\end{pmatrix}, \\
& \tilde{v}_{12} = \begin{pmatrix}
1 & 0 & 0 \\
0 & 1 & 0 \\
0 & d_{2,1}^{-1}\lambda_{2}^{2} \big( \frac{r_{1}(\frac{1}{\omega k})-r_{1}(k)r_{1}(\omega^{2}k)}{f(\omega^{2}k)} \big)_{a} z_{2}^{-i (2\nu_{5}-\nu_{4})} z_{2,(0)}^{-i (2\nu_{2}-\nu_{4})} e^{\frac{iz_{2}^{2}}{2}} & 1 
\end{pmatrix}, \\
& \tilde{v}_{9}^{-1} = \begin{pmatrix}
1 & 0 & 0 \\
0 & 1 & -d_{2,1}\lambda_{2}^{-2} \big( \frac{r_{2}(\frac{1}{\omega k})-r_{2}(k)r_{2}(\omega^{2}k)}{1+r_{1}(k)r_{2}(k)} \big)_{a} z_{2}^{i (2\nu_{5}-\nu_{4})} z_{2,(0)}^{i (2\nu_{2}-\nu_{4})} e^{-\frac{iz_{2}^{2}}{2}} \\
0 & 0 & 1
\end{pmatrix}.
\end{align*}
The above expressions for the matrices $\tilde{v}_{j}$, $j=7,9,10,12$, suggest that we approximate $\tilde{n}$ by $(1,1,1)Y_{2}\tilde{m}^{\omega^{2}k_{2}}$, where $\tilde{m}^{\omega^{2}k_{2}}(x,t,k)$ is the solution to the $3\times 3$ RH problem with $\tilde{m}^{\omega^{2}k_{2}}(x,t,k)\to I$ as $t \to \infty$ uniformly for $k \in \partial \mathbb{D}(\omega^{2}k_{2})$ and whose jumps on $\mathcal{X}_{2}^{\epsilon}$ are
\begin{align}\nonumber
& \tilde{v}_{\mathcal{X}_{2,1}^{\epsilon}}^{\omega^{2}k_{2}} = \begin{pmatrix}
1 & 0 & 0 \\
0 & 1 & 0 \\
0 & -\lambda_{2}^{2} \frac{r_{1}(\frac{1}{\omega k_{\star}})-r_{1}(k_{\star})r_{1}(\omega^{2}k_{\star})}{1+r_{1}(k_{\star})r_{2}(k_{\star})} z_{2}^{-i (2\nu_{5}-\nu_{4})} z_{2,(0)}^{-i (2\nu_{2}-\nu_{4})} e^{\frac{iz_{2}^{2}}{2}} & 1
\end{pmatrix}, 
	\\ \nonumber
& \tilde{v}_{\mathcal{X}_{2,2}^{\epsilon}}^{\omega^{2}k_{2}} = \begin{pmatrix}
1 & 0 & 0 \\
0 & 1 & \lambda_{2}^{-2} \frac{r_{2}(\frac{1}{\omega k_{\star}})-r_{2}(k_{\star})r_{2}(\omega^{2}k_{\star})}{f(\omega^{2}k_{\star})}  z_{2}^{i (2\nu_{5}-\nu_{4})} z_{2,(0)}^{i (2\nu_{2}-\nu_{4})} e^{-\frac{iz_{2}^{2}}{2}} \\
0 & 0 & 1
\end{pmatrix}, 
	\\ \nonumber
& \tilde{v}_{\mathcal{X}_{2,3}^{\epsilon}}^{\omega^{2}k_{2}} = \begin{pmatrix}
1 & 0 & 0 \\
0 & 1 & 0 \\
0 & \lambda_{2}^{2} \frac{r_{1}(\frac{1}{\omega k_{\star}})-r_{1}(k_{\star})r_{1}(\omega^{2}k_{\star})}{f(\omega^{2}k_{\star})} z_{2}^{-i (2\nu_{5}-\nu_{4})} z_{2,(0)}^{-i (2\nu_{2}-\nu_{4})} e^{\frac{iz_{2}^{2}}{2}} & 1 
\end{pmatrix}, 
	\\ \label{vtildeomega2k2def}
& \tilde{v}_{\mathcal{X}_{2,4}^{\epsilon}}^{\omega^{2}k_{2}} = \begin{pmatrix}
1 & 0 & 0 \\
0 & 1 & -\lambda_{2}^{-2} \frac{r_{2}(\frac{1}{\omega k_{\star}})-r_{2}(k_{\star})r_{2}(\omega^{2}k_{\star})}{1+r_{1}(k_{\star})r_{2}(k_{\star})}  z_{2}^{i (2\nu_{5}-\nu_{4})} z_{2,(0)}^{i (2\nu_{2}-\nu_{4})} e^{-\frac{iz_{2}^{2}}{2}} \\
0 & 0 & 1
\end{pmatrix},
\end{align}
where $k_{\star}=\omega^{2} k_{2}$. 
We fix the free parameter $\lambda_{2}$ as follows:
\begin{align}\label{def of lambda2}
\lambda_{2} = |\tilde{r}(\tfrac{1}{\omega k_{\star}})|^{\frac{1}{4}} = |\tilde{r}(\tfrac{1}{k_{2}})|^{\frac{1}{4}}.
\end{align}
The local parametrix $m^{\omega^{2} k_{2}}$ is defined by
\begin{align}\label{IIbis m omega2k2 def}
m^{\omega^{2} k_{2}}(x,t,k) = Y_{2}(\zeta,t)m^{X,(2)}(q_{2},q_{4},q_{5},q_{6},z_{2}(\zeta,t,k))Y_{2}(\zeta,t)^{-1}, \qquad k \in D_{\epsilon}(\omega^{2} k_{2}),
\end{align}
where $m^{X,(2)}$ is the solution to the model RH problem of Lemma \ref{II Xlemma 3 green} with 
\begin{align*}
& q_{2} = \tilde{r}(k_{\star})^{\frac{1}{2}}r_{1}(k_{\star}), & & \bar{q}_{2} = \tilde{r}(k_{\star})^{-\frac{1}{2}}r_{2}(k_{\star}) = \tilde{r}(k_{\star})^{\frac{1}{2}}\overline{r_{1}(k_{\star})}, \\
& q_{4} = |\tilde{r}(\tfrac{1}{\omega^{2}k_{\star}})|^{\frac{1}{2}}r_{1}(\tfrac{1}{\omega^{2}k_{\star}}), & & \bar{q}_{4} = -|\tilde{r}(\tfrac{1}{\omega^{2}k_{\star}})|^{-\frac{1}{2}}r_{2}(\tfrac{1}{\omega^{2}k_{\star}}) = |\tilde{r}(\tfrac{1}{\omega^{2}k_{\star}})|^{\frac{1}{2}}\overline{r_{1}(\tfrac{1}{\omega^{2}k_{\star}})}, \\
& q_{5} = |\tilde{r}(\omega^{2}k_{\star})|^{\frac{1}{2}}r_{1}(\omega^{2}k_{\star}), & & \bar{q}_{5} = -|\tilde{r}(\omega^{2}k_{\star})|^{-\frac{1}{2}}r_{2}(\omega^{2}k_{\star}) = |\tilde{r}(\omega^{2}k_{\star})|^{\frac{1}{2}}\overline{r_{1}(\omega^{2}k_{\star})}, \\
& q_{6} = |\tilde{r}(\tfrac{1}{\omega k_{\star}})|^{\frac{1}{2}}r_{1}(\tfrac{1}{\omega k_{\star}}), & & \bar{q}_{6} = -|\tilde{r}(\tfrac{1}{\omega k_{\star}})|^{-\frac{1}{2}}r_{2}(\tfrac{1}{\omega k_{\star}}) = |\tilde{r}(\tfrac{1}{\omega k_{\star}})|^{\frac{1}{2}}\overline{r_{1}(\tfrac{1}{\omega k_{\star}})}.
\end{align*}
Using \eqref{r1r2 relation on the unit circle} and the relation
\begin{align*}
\big| \tilde{r}(\tfrac{1}{\omega^{2}k}) \big|^{-\frac{1}{2}} = \big| \tilde{r}(\omega^{2}k) \big|^{\frac{1}{2}} = |\tilde{r}(k)|^{-\frac{1}{2}}\big| \tilde{r}(\tfrac{1}{\omega k}) \big|^{\frac{1}{2}}, \qquad k \in \mathbb{C}\setminus\{-1,1\},
\end{align*}
we see that $q_{4}-\bar{q}_{5}-q_{2}\bar{q}_{6}=0$, so that the condition \eqref{II condition on q2q4q5q6} holds as it must. 
Moreover, for $\zeta \in \mathcal{I}$, we have $1 + |q_{2}|^2 - |q_{4}|^{2} = f(k_\star) >0$, $1 - |q_{5}|^2 - |q_{6}|^{2} = f(\omega^{2}k_{\star}) >0$, and
\begin{align*}
\tfrac{q_{6}-q_{2}q_{5}}{1+|q_{2}|^{2}}
= \lambda_{2}^{2} \tfrac{r_{1}(\frac{1}{\omega k_{\star}})-r_{1}(k_{\star})r_{1}(\omega^{2}k_{\star})}{1+r_{1}(k_{\star})r_{2}(k_{\star})},
\qquad 
-\tfrac{\bar{q}_6-\bar{q}_2\bar{q}_5}{1-|q_{5}|^{2}-|q_{6}|^{2}} = \lambda_{2}^{-2} \tfrac{r_{2}(\frac{1}{\omega k_{\star}})-r_{2}(k_{\star})r_{2}(\omega^{2}k_{\star})}{f(\omega^{2}k_{\star})},
	\\
\tfrac{q_{6}-q_{2}q_{5}}{1-|q_{5}|^{2}-|q_{6}|^{2}} = \lambda_{2}^{2} \tfrac{r_{1}(\frac{1}{\omega k_{\star}})-r_{1}(k_{\star})r_{1}(\omega^{2}k_{\star})}{f(\omega^{2}k_{\star})},
\qquad
\tfrac{\bar{q}_6-\bar{q}_2\bar{q}_5}{1+|q_{2}|^{2}} = -\lambda_{2}^{-2} \tfrac{r_{2}(\frac{1}{\omega k_{\star}})-r_{2}(k_{\star})r_{2}(\omega^{2}k_{\star})}{1+r_{1}(k_{\star})r_{2}(k_{\star})},
\end{align*}
so that the jump matrices in Lemma \ref{II Xlemma 3 green} agree with those in (\ref{vtildeomega2k2def}) if $z = z_2$.

The next lemma is proved in the same way as \cite[Lemma 9.9]{CLmain}.
\begin{lemma}\label{IIbis lemma: bound on Y green}
$Y_{2}(\zeta,t)$ satisfies $\sup_{\zeta \in \mathcal{I}} \sup_{t \geq 2} | Y_{2}(\zeta,t)^{\pm 1}| \leq C$. Furthermore,
\begin{align}
& |\tilde{d}_{2,0}(\zeta, t)| = e^{\pi (2\nu_{2}-\nu_{4})}, & & \zeta \in \mathcal{I}, \ t \geq 2, \label{IIbis d0estimate green} \\
& |d_{2,1}(\zeta, k) - 1| \leq C |k - \omega^{2} k_2| (1+ |\ln|k-\omega^{2} k_2||), & & \zeta \in \mathcal{I}, \ k \in \mathcal{X}_{2}^{\epsilon}. \label{IIbis d1estimate green}
\end{align}
\end{lemma}
The following lemma can be proved in the same way as \cite[Lemmas 9.10 and 9.12]{CLmain}.
\begin{lemma}\label{IIbis k0lemma green}
For each $t \geq 2$ and $\zeta \in \mathcal{I}$, the function $m^{\omega^{2} k_2}(x,t,k)$ defined in \eqref{IIbis m omega2k2 def} is analytic for $k \in D_\epsilon(\omega^{2} k_2) \setminus \mathcal{X}_{2}^\epsilon$. Moreover, $m^{\omega^{2} k_2}(x,t,k)$ is uniformly bounded for $t \geq 2$, $\zeta \in \mathcal{I}$, and $k \in D_\epsilon(\omega^{2} k_2) \setminus \mathcal{X}_{2}^\epsilon$. On $\mathcal{X}_{2}^\epsilon$, $m^{\omega^{2} k_2}$ satisfies $m_+^{\omega^{2} k_2} =  m_-^{\omega^{2} k_2} v^{\omega^{2} k_2}$, where the jump matrix $v^{\omega^{2} k_2}$ obeys
\begin{align}\label{IIbis v3vk0estimate green}
\begin{cases}
 \| v^{(3)} - v^{\omega^{2} k_2} \|_{L^1(\mathcal{X}_{2}^\epsilon)} \leq C t^{-1} \ln t,
	\\
\| v^{(3)} - v^{\omega^{2} k_2} \|_{L^\infty(\mathcal{X}_{2}^\epsilon)} \leq C t^{-1/2} \ln t,
\end{cases} \qquad \zeta \in \mathcal{I}, \ t \geq 2.
\end{align}
Furthermore, as $t \to \infty$,
\begin{align}\label{IIbis mmodmuestimate2 green}
& \| m^{\omega^{2} k_2}(x,t,\cdot) - I \|_{L^\infty(\partial D_\epsilon(\omega^{2} k_2))} = O(t^{-1/2}),
	\\ \label{IIbis mmodmuestimate1 green}
& m^{\omega^{2} k_2} - I = \frac{Y_{2}(\zeta, t)m_1^{X,(2)} Y_{2}(\zeta, t)^{-1}}{z_{2,\star}\sqrt{t} (k-\omega^{2} k_2) \hat{z}_{2}(\zeta,k)} + O(t^{-1})
\end{align}
uniformly for $\zeta \in \mathcal{I}$, where $m_1^{X,(2)} = m_1^{X,(2)}(q_{2},q_{4},q_{5},q_{6})$ is given by \eqref{II mXasymptotics 3 green}.
\end{lemma}

\section{The $n^{(3)}\to \hat{n}$ transformation}\label{section:small norm}

Using the relations
\begin{align*}
& m^{\omega k_4}(x,t,k) = \mathcal{A} m^{\omega k_4}(x,t,\omega k)\mathcal{A}^{-1} = \mathcal{B} m^{\omega k_4}(x,t,k^{-1}) \mathcal{B}, \\
& m^{\omega^{2} k_2}(x,t,k) = \mathcal{A} m^{\omega^{2} k_2}(x,t,\omega k)\mathcal{A}^{-1} = \mathcal{B} m^{\omega^{2} k_2}(x,t,k^{-1}) \mathcal{B},
\end{align*}
we extend the domain of definition of $m^{\omega k_4}$ and $m^{\omega^{2} k_2}$ from $D_\epsilon(\omega k_4)$ and $D_\epsilon(\omega^{2} k_2)$ to $\mathcal{D}_{\omega k_4}$ and $\mathcal{D}_{\omega^{2} k_2}$ respectively, where
\begin{align*}
& \mathcal{D}_{\omega k_4} = D_\epsilon(\omega k_4) \cup \omega D_\epsilon(\omega k_4) \cup \omega^2 D_\epsilon(\omega k_4) \cup D_\epsilon(\tfrac{1}{\omega k_4}) \cup \omega D_\epsilon(\tfrac{1}{\omega k_4}) \cup \omega^2 D_\epsilon(\tfrac{1}{\omega k_4}), \\
& \mathcal{D}_{\omega^{2} k_2} = D_\epsilon(\omega^{2} k_2) \cup \omega D_\epsilon(\omega^{2} k_2) \cup \omega^2 D_\epsilon(\omega^{2} k_2) \cup D_\epsilon(\tfrac{1}{\omega^{2} k_2}) \cup \omega D_\epsilon(\tfrac{1}{\omega^{2} k_2}) \cup \omega^2 D_\epsilon(\tfrac{1}{\omega^{2} k_2}).
\end{align*}
We will prove that 
\begin{align}\label{Sector IIbis final transfo}
\hat{n} :=
\begin{cases}
n^{(3)} (m^{\omega k_4})^{-1}, & k \in \mathcal{D}_{\omega k_4}, \\
n^{(3)} (m^{\omega^{2} k_2})^{-1}, & k \in \mathcal{D}_{\omega^{2} k_2}, \\
n^{(3)}, & \text{elsewhere},
\end{cases}
\end{align}
satisfies a small-norm RH problem as $t \to \infty$, $\zeta \in \mathcal{I}$. Let $\hat{\Gamma} = \Gamma^{(3)} \cup \partial \mathcal{D}$ with $\mathcal{D} := \mathcal{D}_{\omega k_4}\cup \mathcal{D}_{\omega^{2} k_2}$ be the contour represented in Figure \ref{IIbis Gammahat.pdf}. We orient the circles that are part of $\partial \mathcal{D}$ in the clockwise direction, and define $\hat{v}$ by
\begin{align}\label{def of vhat II}
\hat{v}= \begin{cases}
v^{(3)}, & k \in \hat{\Gamma} \setminus \bar{\mathcal{D}}, 	\\
m^{\omega k_4}, & k \in \partial \mathcal{D}_{\omega k_4}, \\
m^{\omega^{2} k_2}, & k \in \partial \mathcal{D}_{\omega^{2} k_2}, \\
m_-^{\omega k_4} v^{(3)}(m_+^{\omega k_4})^{-1}, & k \in \hat{\Gamma} \cap \mathcal{D}_{\omega k_4}, \\
m_-^{\omega^{2} k_2} v^{(3)}(m_+^{\omega^{2} k_2})^{-1}, & k \in \hat{\Gamma} \cap \mathcal{D}_{\omega^{2} k_2}.
\end{cases}
\end{align}
The function $\hat{n}$ satisfies the following RH problem: (a) $\hat{n}:\mathbb{C}\setminus \hat{\Gamma}\to \mathbb{C}^{1\times 3}$ is analytic, (b) $\hat{n}_{+} = \hat{n}_{-}\hat{v}$ for $k \in \hat{\Gamma}\setminus \hat{\Gamma}_{\star}$, where $\hat{\Gamma}_{\star}$ is the set of self-intersection points of $\hat{\Gamma}$, (c) $\hat{n}(x,t,k) = (1,1,1)+O(k^{-1})$ as $k \to \infty$, and (d) $\hat{n}(x,t,k)$ remains bounded as $k\to k_{\star}\in \hat{\Gamma}_{\star}$.
\begin{figure}
\begin{center}
\begin{tikzpicture}[master]
\node at (0,0) {};
\draw[black,line width=0.5 mm] (0,0)--(30:7.5);
\draw[black,line width=0.5 mm,->-=0.25,->-=0.57,->-=0.71,->-=0.91] (0,0)--(90:7.5);
\draw[black,line width=0.5 mm] (0,0)--(150:7.5);
\draw[dashed,black,line width=0.15 mm] (0,0)--(60:7.5);
\draw[dashed,black,line width=0.15 mm] (0,0)--(120:7.5);

\draw[black,line width=0.5 mm] ([shift=(30:3*1.5cm)]0,0) arc (30:150:3*1.5cm);
\draw[black,arrows={-Triangle[length=0.27cm,width=0.18cm]}]
($(66:3*1.5)$) --  ++(-21:0.001);
\draw[black,arrows={-Triangle[length=0.21cm,width=0.14cm]}]
($(116.5:3*1.5)$) --  ++(116+90:0.001);

\draw[black,line width=0.5 mm] ([shift=(30:3*1.5cm)]0,0) arc (30:150:3*1.5cm);

\draw[black,arrows={-Triangle[length=0.21cm,width=0.14cm]}]
($(86:3*1.5)$) --  ++(86+90:0.001);
\draw[black,arrows={-Triangle[length=0.27cm,width=0.18cm]}]
($(100:3*1.5)$) --  ++(100+90:0.001);

\draw[black,line width=0.5 mm] (150:3.65)--($(129.688:3*1.5)+(129.688+135:0.5)$)--(129.688:3*1.5)--($(129.688:3*1.5)+(129.688+45:0.5)$)--(150:5.8);
\draw[black,line width=0.5 mm] (30:3.65)--($(38.4686:3*1.5)+(38.4686-135:0.5)$)--(38.4686:3*1.5)--($(38.4686:3*1.5)+(38.4686-45:0.5)$)--(30:5.8);
\draw[black,line width=0.5 mm,-<-=0.07,->-=0.78] (90:3.65)--($(80:3*1.5)+(80+135:0.5)$)--(80:3*1.5)--($(80:3*1.5)+(80+45:0.5)$)--(90:5.8);
\draw[black,line width=0.5 mm,->-=0.25,-<-=0.65] (90:3.65)--($(110.3:3*1.5)+(110.3-135:0.5)$)--(110.3:3*1.5)--($(110.3:3*1.5)+(110.3-45:0.5)$)--(90:5.8);
\draw[black,line width=0.5 mm,-<-=0.07,->-=0.78] (120:3.65)--($(110.3:3*1.5)+(110.3+135:0.5)$)--(110.3:3*1.5)--($(110.3:3*1.5)+(110.3+45:0.5)$)--(120:5.8);
\draw[black,line width=0.5 mm] (120:3.65)--($(129.688:3*1.5)+(129.688-135:0.5)$)--(129.688:3*1.5)--($(129.688:3*1.5)+(129.688-45:0.5)$)--(120:5.8);
\draw[black,line width=0.5 mm,-<-=0.12,->-=0.78] (60:3.65)--($(80:3*1.5)+(80-135:0.5)$)--(80:3*1.5)--($(80:3*1.5)+(80-45:0.5)$)--(60:5.8);
\draw[black,line width=0.5 mm] (60:3.65)--($(38.4686:3*1.5)+(38.4686+135:0.5)$)--(38.4686:3*1.5)--($(38.4686:3*1.5)+(38.4686+45:0.5)$)--(60:5.8);

\draw[black,line width=0.5 mm,->-=0.6] (60:4.5)--(60:5.8);
\draw[black,line width=0.5 mm,->-=0.75] (60:3.65)--(60:4.5);
\draw[black,line width=0.5 mm,->-=0.6] (120:4.5)--(120:5.8);
\draw[black,line width=0.5 mm,->-=0.75] (120:3.65)--(120:4.5);

\draw[black,line width=0.5 mm] (129.688:3.65)--(129.688:5.8);
\draw[black,line width=0.5 mm,-<-=0.17,->-=0.84] (110.3:3.65)--(110.3:5.8);
\draw[black,line width=0.5 mm,-<-=0.17,->-=0.84] (80:3.65)--(80:5.8);
\draw[black,line width=0.5 mm] (38.4686:3.65)--(38.4686:5.8);

\draw[black,line width=0.5 mm] ([shift=(30:3.65cm)]0,0) arc (30:150:3.65cm);
\draw[black,line width=0.5 mm] ([shift=(30:5.8cm)]0,0) arc (30:150:5.8cm);

\draw[blue,fill] (129.688:3*1.5) circle (0.12cm);
\draw[green,fill] (110.3:3*1.5) circle (0.12cm);
\draw[red,fill] (80:3*1.5) circle (0.12cm);
\draw[blue,fill] (38.4686:3*1.5) circle (0.12cm);

\draw[black,line width=0.5 mm,->-=0.6] (60:4.5)--(60:5.8);
\draw[black,line width=0.5 mm,->-=0.75] (60:3.65)--(60:4.5);
\draw[black,line width=0.5 mm,->-=0.6] (120:4.5)--(120:5.8);
\draw[black,line width=0.5 mm,->-=0.75] (120:3.65)--(120:4.5);

\draw[black,line width=0.5 mm] (129.688:3*1.5) circle (0.5cm);
\draw[black,line width=0.5 mm] (110.3:3*1.5) circle (0.5cm);
\draw[black,line width=0.5 mm] (80:3*1.5) circle (0.5cm);
\draw[black,line width=0.5 mm] (38.4686:3*1.5) circle (0.5cm);
\end{tikzpicture}
\end{center}
\begin{figuretext}
\label{IIbis Gammahat.pdf}The contour $\hat{\Gamma} = \Gamma^{(3)} \cup \partial \mathcal{D}$ for $\arg k \in [\frac{\pi}{6},\frac{5\pi}{6}]$ (solid), the boundary of $\mathsf{S}$ (dashed) and, from right to left, the saddle points $k_{3}$ (blue), $\omega k_{4}$ (red), $\omega^{2}k_{2}$ (green), and $k_{1}$ (blue).
\end{figuretext}
\end{figure} 

\begin{lemma}\label{IIbis whatlemma}
Let $\hat{w} = \hat{v}-I$. Uniformly for $t \geq 2$ and $\zeta \in \mathcal{I}$, we have
\begin{subequations}\label{IIbis hatwestimate}
\begin{align}\label{IIbis hatwestimate1}
& \| \hat{w}\|_{(L^1\cap L^\infty)(\hat{\Gamma} \setminus (\partial \mathcal{D} \cup \hat{\mathcal{X}}^\epsilon))} \leq C t^{-1},
	\\\label{IIbis hatwestimate3}
& \| \hat{w}\|_{(L^1\cap L^{\infty})(\partial \mathcal{D})} \leq C t^{-1/2},	\\\label{IIbis hatwestimate4}
& \| \hat{w}\|_{L^1(\hat{\mathcal{X}}^\epsilon)} \leq C t^{-1}\ln t,
	\\\label{IIbis hatwestimate5}
& \| \hat{w}\|_{L^\infty(\hat{\mathcal{X}}^\epsilon)} \leq C t^{-1/2}\ln t,
\end{align}
\end{subequations}
\end{lemma}
\begin{proof}
Since $\hat{\Gamma} \setminus (\partial \mathcal{D} \cup \hat{\mathcal{X}}^\epsilon) = \Gamma^{(3)}\setminus \hat{\mathcal{X}}^\epsilon$, \eqref{IIbis hatwestimate1} is a consequence of \eqref{def of vhat II}, Lemma \ref{II v3lemma}, and (for the part inside $\mathcal{D}$) the uniform boundedness of $m^{\omega k_4}$ and $m^{\omega^2 k_2}$. 
The estimate \eqref{IIbis hatwestimate3} follows from \eqref{def of vhat II}, \eqref{IIbis mmodmuestimate2}, and \eqref{IIbis mmodmuestimate2 green}. Finally, \eqref{IIbis hatwestimate4} and \eqref{IIbis hatwestimate5} follow from \eqref{def of vhat II}, \eqref{IIbis v3vk0estimate}, \eqref{IIbis v3vk0estimate green}, and the uniform boundedness of $m^{\omega k_4}$ and $m^{\omega^2 k_2}$.
\end{proof} 
For a function $h$ defined on $\hat{\Gamma}$, we define $\hat{\mathcal{C}}h$ by
\begin{align*}
(\hat{\mathcal{C}}h)(k) = \frac{1}{2\pi i} \int_{\hat{\Gamma}} \frac{h(k')dk'}{k' - k}, \qquad k \in \C \setminus \hat{\Gamma}.
\end{align*}

Lemma \ref{IIbis whatlemma} implies that
\begin{align}\label{IIbis hatwLinfty}
\begin{cases}
\|\hat{w}\|_{L^1(\hat{\Gamma})}\leq C t^{-1/2},
	\\
\|\hat{w}\|_{L^\infty(\hat{\Gamma})}\leq C t^{-1/2}\ln t,
\end{cases}	 \qquad t \geq 2, \ \zeta \in \mathcal{I}.
\end{align}
Therefore, using the estimate $\| f \|_{L^p} \leq \| f \|_{L^1}^{1/p}\|f \|_{L^{\infty}}^{(p-1)/p}$, we find
\begin{align}\label{IIbis Lp norm of what}
& \|\hat{w}\|_{L^p(\hat{\Gamma})} 
\leq C t^{-\frac{1}{2}} (\ln t)^{\frac{p-1}{p}},  \qquad t \geq 2, \ \zeta \in \mathcal{I},
\end{align}
for each $1 \leq p \leq \infty$. It follows from \eqref{IIbis Lp norm of what} that
$\hat{w} \in L^{2}(\hat{\Gamma}) \cap L^{\infty}(\hat{\Gamma})$, and therefore $\hat{\mathcal{C}}_{\hat{w}}h := \hat{\mathcal{C}}_{-}(h \hat{w})$ is well-defined as an operator $\hat{\mathcal{C}}_{\hat{w}}=\hat{\mathcal{C}}_{\hat{w}(x,t,\cdot)}: L^{2}(\hat{\Gamma})+L^{\infty}(\hat{\Gamma}) \to L^{2}(\hat{\Gamma})$. It also follows from \eqref{IIbis Lp norm of what} that there exists a $T > 0$ such that $I - \hat{\mathcal{C}}_{\hat{w}(x, t, \cdot)} \in \mathcal{B}(L^{2}(\hat{\Gamma}))$ is invertible whenever $t \geq T$ and $\zeta \in \mathcal{I}$, where $\mathcal{B}(L^{2}(\hat{\Gamma}))$ denotes the space of bounded linear operators on $L^{2}(\hat{\Gamma})$. By standard theory for small-norm RH problems, we have
\begin{align}\label{IIbis hatmrepresentation}
\hat{n}(x, t, k) = (1,1,1) + \hat{\mathcal{C}}(\hat{\mu}\hat{w}) = (1,1,1) + \frac{1}{2\pi i}\int_{\hat{\Gamma}} \hat{\mu}(x, t, s) \hat{w}(x, t, s) \frac{ds}{s - k}
\end{align}
for $t \geq T$ and $\zeta \in \mathcal{I}$, where 
\begin{align}\label{IIbis hatmudef}
\hat{\mu} = (1,1,1) + (I - \hat{\mathcal{C}}_{\hat{w}})^{-1}\hat{\mathcal{C}}_{\hat{w}}(1,1,1) \in (1,1,1) + L^{2}(\hat{\Gamma}).
\end{align}
Moreover, since $\|\hat{\mathcal{C}}_{-}\|_{\mathcal{B}(L^{p}(\hat{\Gamma}))}< \infty$ for $p \in (1,\infty)$, we infer from \eqref{IIbis Lp norm of what} and \eqref{IIbis hatmudef} that for any $p \in (1,\infty)$ there exists $C_{p}>0$ such that
\begin{align}\label{estimate on mu}
& \|\hat{\mu} - (1,1,1)\|_{L^p(\hat{\Gamma})} \leq  C_{p} t^{-\frac{1}{2}}(\ln t)^{\frac{p-1}{p}}
\end{align}
holds for all sufficiently large $t$ and all $\zeta \in \mathcal{I}$. Now, we turn to the problem of finding asymptotics for $\hat{n}$. We introduce the following nontangential limit:
\begin{align*}
& \hat{n}^{(1)}(x,t):=\ntlim_{k\to \infty} k(\hat{n}(x,t,k) - (1,1,1))
= - \frac{1}{2\pi i}\int_{\hat{\Gamma}} \hat{\mu}(x,t,k) \hat{w}(x,t,k) dk.
\end{align*}
Since
\begin{align*}
\hat{n}^{(1)}(x,t) = & -\frac{(1,1,1)}{2\pi i}\int_{\partial \mathcal{D}} \hat{w}(x,t,k) dk  -\frac{(1,1,1)}{2\pi i}\int_{\hat{\Gamma}\setminus\partial \mathcal{D}} \hat{w}(x,t,k) dk
	\\
& -\frac{1}{2\pi i}\int_{\hat{\Gamma}} (\hat{\mu}(x,t,k)-(1,1,1))\hat{w}(x,t,k) dk,
\end{align*}
we conclude from \eqref{estimate on mu} and Lemma \ref{IIbis whatlemma} that
\begin{align}\label{IIbis limlhatm}
& \hat{n}^{(1)}(x,t) = -\frac{(1,1,1)}{2\pi i}\int_{\partial \mathcal{D}} \hat{w}(x,t,k) dk + O(t^{-1}\ln t) \qquad \mbox{as } t \to \infty,
\end{align}
uniformly for $\zeta \in \mathcal{I}$. Let us define $\{F_{1}^{(l)},F_{2}^{(l)}\}_{l \in \mathbb{Z}}$ by
\begin{align*}
& F_{1}^{(l)}(\zeta,t) = - \frac{1}{2\pi i} \int_{\partial D_\epsilon(\omega k_4)} k^{l-1}\hat{w}(x,t,k) dk	
= - \frac{1}{2\pi i}\int_{\partial D_\epsilon(\omega k_4)}k^{l-1}(m^{\omega k_4} - I) dk, \\
& F_{2}^{(l)}(\zeta,t) = - \frac{1}{2\pi i} \int_{\partial D_\epsilon(\omega^{2} k_2)} k^{l-1}\hat{w}(x,t,k) dk	
= - \frac{1}{2\pi i}\int_{\partial D_\epsilon(\omega^{2} k_2)}k^{l-1}(m^{\omega^{2} k_2} - I) dk.
\end{align*}
Using \eqref{IIbis mmodmuestimate1}, \eqref{IIbis mmodmuestimate1 green}, and $\hat{z}_{1}(\zeta,\omega k_{4})=1=\hat{z}_{2}(\zeta,\omega^{2} k_{2})$, and recalling that $\partial D_\epsilon(\omega k_{4})$ and $\partial D_\epsilon(\omega^{2} k_{2})$ are oriented clockwise, we find
\begin{align}
F_{1}^{(l)}(\zeta, t) &  =  -i(\omega k_{4})^{l}Z_{1}(\zeta,t) + O(t^{-1}) & & \mbox{as } t \to \infty, \label{IIbis asymptotics for F1l} \\
F_{2}^{(l)}(\zeta, t) & =  -i(\omega^{2} k_{2})^{l}Z_{2}(\zeta,t) + O(t^{-1}) & & \mbox{as } t \to \infty, \label{IIbis asymptotics for F4l}
\end{align}
uniformly for $\zeta \in \mathcal{I}$, where
\begin{subequations}\label{def of Z1Z2}
\begin{align}
& Z_{1}(\zeta,t) \hspace{-0.05cm} = \hspace{-0.05cm} \frac{Y_{1}(\zeta,t) m_1^{X,(1)} Y_{1}(\zeta,t)^{-1}}{-i\omega k_{4}z_{1,\star}\sqrt{t}} \hspace{-0.05cm} = \hspace{-0.05cm} 
\frac{t^{-\frac{1}{2}}}{-i\omega k_{4}z_{1,\star}} \hspace{-0.1cm} \begin{pmatrix}
\hspace{-0.05cm} 0 & \hspace{-0.1cm} 0 & \hspace{-0.1cm}  \frac{\beta_{12}^{(1)}\tilde{d}_{1,0}\lambda_{1}^{2}}{e^{t\Phi_{31}(\zeta,\omega k_{4})}} \\
\hspace{-0.05cm} 0 & \hspace{-0.1cm} 0 & \hspace{-0.1cm} 0 \\
\hspace{-0.05cm} \frac{\beta_{21}^{(1)}e^{t\Phi_{31}(\zeta,\omega k_{4})}}{\tilde{d}_{1,0}\lambda_{1}^{2}} & \hspace{-0.1cm} 0 & \hspace{-0.1cm} 0
\end{pmatrix}\hspace{-0.1cm} , \\
& Z_{2}(\zeta,t) = \frac{Y_{2}(\zeta,t) m_1^{X,(2)} Y_{2}(\zeta,t)^{-1}}{-i\omega^{2} k_{2}z_{2,\star}\sqrt{t}} = 
\frac{t^{-\frac{1}{2}}}{-i\omega^{2}k_{2}z_{2,\star}} \begin{pmatrix}
0 & 0 & 0 \\
0 & 0 & \frac{\beta_{12}^{(2)}\tilde{d}_{2,0}\lambda_{2}^{2}}{e^{t\Phi_{32}(\zeta,\omega^{2}k_{2})}} \\
0 & \frac{\beta_{21}^{(2)}e^{t\Phi_{32}(\zeta,\omega^{2}k_{2})}}{\tilde{d}_{2,0}\lambda_{2}^{2}} & 0
\end{pmatrix},
\end{align}
\end{subequations}
and, letting $\tilde{q}_1 := q = |\tilde{r}(k_{4})|^{\frac{1}{2}}r_{1}(k_{4})$,
\begin{subequations}\label{betap1p and betap2p Sector IIbis}
\begin{align}
& \beta_{12}^{(1)} = \frac{e^{\frac{3\pi i}{4}}e^{\frac{\pi \nu_{1}}{2}} \sqrt{2\pi}\bar{\tilde{q}}_1}{(e^{\pi \nu_{1}}-e^{-\pi \nu_{1}})\Gamma(-i\nu_{1})}, & & \beta_{21}^{(1)} = \frac{e^{-\frac{3\pi i}{4}}e^{\frac{\pi \nu_{1}}{2}}\sqrt{2\pi} \tilde{q}_1}{(e^{\pi \nu_{1}}-e^{-\pi \nu_{1}})\Gamma(i\nu_{1})}, \\
&  \beta_{12}^{(2)} = \frac{e^{\frac{3\pi i}{4}}e^{\frac{\pi\hat{\nu}_{2}}{2}}e^{2\pi (\nu_{4}-\nu_{2})}\sqrt{2\pi}(\bar{q}_{6}-\bar{q}_{2}\bar{q}_{5})}{(e^{\pi \hat{\nu}_{2}}-e^{-\pi \hat{\nu}_{2}})\Gamma(-i\hat{\nu}_{2})}, & & \beta_{21}^{(2)} = \frac{e^{-\frac{3\pi i}{4}}e^{\frac{\pi \hat{\nu}_{2}}{2}}e^{2\pi\nu_{2}}\sqrt{2\pi}(q_{6}-q_{2}q_{5})}{(e^{\pi \hat{\nu}_{2}}-e^{-\pi \hat{\nu}_{2}})\Gamma(i\hat{\nu}_{2})},
\end{align}
\end{subequations}
and $\nu_{1}$, $\nu_{2}$, $\nu_{4}$, $\hat{\nu}_{2}$ have been defined in Lemmas \ref{IIbis deltalemma} and \ref{lemma: nuhat lemma IIbis}.
\begin{lemma}\label{IIbis lemma: some integrals by symmetry}
For $l \in \mathbb{Z}$ and $j=0,1,2$, we have
\begin{align}
& -\frac{1}{2\pi i}\int_{\omega^{j} \partial D_\epsilon(\omega k_4)} k^{l}\hat{w}(x,t,k)dk = \omega^{j(l+1)} \mathcal{A}^{-j}F_{1}^{(l+1)}(\zeta,t)\mathcal{A}^{j}, \label{IIbis int1 F1} \\
& -\frac{1}{2\pi i}\int_{\omega^{j} \partial D_\epsilon((\omega k_4)^{-1})} k^{l}\hat{w}(x,t,k)dk = -\omega^{j(l+1)}\mathcal{A}^{-j}\mathcal{B} F_{1}^{(-l-1)}(\zeta,t) \mathcal{B}\mathcal{A}^{j}, \label{IIbis int1 tilde F1} \\
& -\frac{1}{2\pi i}\int_{\omega^{j} \partial D_\epsilon(\omega^{2} k_2)} k^{l}\hat{w}(x,t,k)dk = \omega^{j(l+1)} \mathcal{A}^{-j}F_{2}^{(l+1)}(\zeta,t)\mathcal{A}^{j}, \label{IIbis int1 F2} \\
& -\frac{1}{2\pi i}\int_{\omega^{j} \partial D_\epsilon((\omega^{2} k_2)^{-1})} k^{l}\hat{w}(x,t,k)dk = -\omega^{j(l+1)}\mathcal{A}^{-j}\mathcal{B} F_{2}^{(-l-1)}(\zeta,t) \mathcal{B}\mathcal{A}^{j}. \label{IIbis int1 tilde F2}
\end{align}
\end{lemma}
\begin{proof}
This follows from the symmetries $\hat{w}(x, t, k) = \mathcal{A} \hat{w}(x, t, \omega k) \mathcal{A}^{-1} = \mathcal{B} \hat{w}(x, t, k^{-1}) \mathcal{B}$ valid for $k \in \partial \mathcal{D}$. The proof is identical to \cite[Proof of Lemma 9.10]{CLmain}, so we omit further details here.
\end{proof}
Using Lemma \ref{IIbis lemma: some integrals by symmetry}, we obtain
\begin{align*}
& \frac{- 1}{2\pi i}\int_{\partial \mathcal{D}} \hat{w}(x,t,k) dk =  \sum_{j=0}^{2}  \frac{-1}{2\pi i}\int_{\omega^{j} \partial D_\epsilon(\omega k_4)} \hat{w}(x,t,k) dk - \sum_{j=0}^{2}  \frac{1}{2\pi i}\int_{\omega^{j} \partial D_\epsilon((\omega k_4)^{-1})}  \hat{w}(x,t,k) dk \\
& \hspace{3.2cm} - \sum_{j=0}^{2}  \frac{1}{2\pi i}\int_{\omega^{j} \partial D_\epsilon(\omega^{2} k_2)} \hat{w}(x,t,k) dk - \sum_{j=0}^{2}  \frac{1}{2\pi i}\int_{\omega^{j} \partial D_\epsilon((\omega^{2} k_2)^{-1})} \hat{w}(x,t,k) dk \\
& \hspace{1.5cm} = \sum_{j=0}^{2} \omega^{j} \mathcal{A}^{-j}\big( F_{1}^{(1)}(\zeta,t) + F_{2}^{(1)}(\zeta,t) \big)\mathcal{A}^{j} - \sum_{j=0}^{2} \omega^{j}\mathcal{A}^{-j}\mathcal{B} \big( F_{1}^{(-1)}(\zeta,t) + F_{2}^{(-1)}(\zeta,t) \big) \mathcal{B}\mathcal{A}^{j}.
\end{align*}
Hence, by \eqref{IIbis limlhatm}, \eqref{IIbis asymptotics for F1l}, and \eqref{IIbis asymptotics for F4l}, we obtain $\hat{n}^{(1)} = (1,1,1)\hat{m}^{(1)}$, where
\begin{align}
& \hat{m}^{(1)}(x,t) =  \; \sum_{j=0}^{2} \omega^{j} \mathcal{A}^{-j}F_{1}^{(1)}(\zeta,t)\mathcal{A}^{j} - \sum_{j=0}^{2} \omega^{j}\mathcal{A}^{-j}\mathcal{B} F_{1}^{(-1)}(\zeta,t) \mathcal{B}\mathcal{A}^{j} \nonumber \\
& + \sum_{j=0}^{2} \omega^{j} \mathcal{A}^{-j}F_{2}^{(1)}(\zeta,t)\mathcal{A}^{j} - \sum_{j=0}^{2} \omega^{j}\mathcal{A}^{-j}\mathcal{B} F_{2}^{(-1)}(\zeta,t) \mathcal{B}\mathcal{A}^{j} + O(t^{-1}\ln t) \nonumber \\
& = -i(\omega k_{4})^{1}\sum_{j=0}^{2} \omega^{j} \mathcal{A}^{-j}Z_{1}(\zeta,t)\mathcal{A}^{j} + i(\omega k_{4})^{-1} \sum_{j=0}^{2} \omega^{j}\mathcal{A}^{-j}\mathcal{B} Z_{1}(\zeta,t) \mathcal{B}\mathcal{A}^{j} \nonumber \\
& -i(\omega^{2} k_{2})^{1}\sum_{j=0}^{2} \omega^{j} \mathcal{A}^{-j}Z_{2}(\zeta,t)\mathcal{A}^{j} + i(\omega^{2} k_{2})^{-1} \sum_{j=0}^{2} \omega^{j}\mathcal{A}^{-j}\mathcal{B} Z_{2}(\zeta,t) \mathcal{B}\mathcal{A}^{j} + O(t^{-1}\ln t) \label{IIbis mhatplp asymptotics}
\end{align}
as $t \to \infty$ uniformly for $\zeta \in \mathcal{I}$. 

\section{Asymptotics of $u$}

Recall from \eqref{recoveruvn} that $u(x,t) = -i\sqrt{3}\frac{\partial}{\partial x}n_{3}^{(1)}(x,t)$, where $n_{3}(x,t,k) = 1+n_{3}^{(1)}(x,t)k^{-1}+O(k^{-2})$ as $k \to \infty$. Using \eqref{Sector IIbis first transfo}, \eqref{Sector IIbis second transfo}, \eqref{IIbis def of mp3p}, and \eqref{Sector IIbis final transfo} to invert the transformations $n \mapsto n^{(1)}\mapsto n^{(2)}\mapsto n^{(3)}\mapsto \hat{n}$, we obtain
\begin{align*}
n = \hat{n}\Delta^{-1}(G^{(2)})^{-1}(G^{(1)})^{-1}, \qquad k \in \mathbb{C}\setminus (\hat{\Gamma}\cup\bar{\mathcal{D}}),
\end{align*}
where $G^{(1)}$, $G^{(2)}$, $\Delta$ are defined in \eqref{IIbis Gp1pdef}, \eqref{IIbis Gp2pdef}, and \eqref{IIbis def of Delta}, respectively. Using that $\partial_x = t^{-1}\partial_\zeta$, it follows that
\begin{align}\nonumber
u(x,t) & = -i\sqrt{3}\frac{\partial}{\partial x}\bigg( \hat{n}_{3}^{(1)}(x,t) + \ntlim_{k\to \infty} k (\Delta_{33}(\zeta,k)^{-1}-1) \bigg) 
	\\
&= -i\sqrt{3}\frac{\partial}{\partial x} \hat{n}_{3}^{(1)}(x,t) + O(t^{-1}), \qquad t \to \infty, \label{IIbis recoverun}
\end{align}
where $\hat{n}_{3}(x,t,k) = 1+\hat{n}_{3}^{(1)}(x,t)k^{-1}+O(k^{-2})$ as $k \to \infty$. Utilizing \eqref{def of Z1Z2} and \eqref{IIbis mhatplp asymptotics}, we obtain, as $t \to \infty$,
\begin{align*}
 \hat{n}_{3}^{(1)} = &\; Z_{1,13}(i (\omega k_{4})^{-1}-i (\omega k_{4})^{1})+Z_{1,31}(i \omega^{2}(\omega k_{4})^{-1}- i \omega^{1}(\omega k_{4})^{1})
 	\\
& + Z_{2,23}(i (\omega^{2} k_{2})^{-1}-i (\omega^{2} k_{2})^{1})+Z_{2,32}(i \omega^{1}(\omega^{2} k_{2})^{-1}- i \omega^{2}(\omega^{2} k_{2})^{1}) + O(t^{-1}\ln t).
\end{align*}
By Lemma \ref{lemma: nuhat lemma IIbis}, $\nu_{1}\geq 0$ and $\hat{\nu}_{2} \geq 0$. We also have $\nu_2, \nu_4 \in \R$.
Hence (\ref{betap1p and betap2p Sector IIbis}) shows that $\overline{\beta_{12}^{(1)}} = \beta_{21}^{(1)}$ and $\overline{\beta_{12}^{(2)}} = e^{-2\pi(2\nu_2 - \nu_4)}\beta_{21}^{(2)}$. Since $\lambda_{1} = |\tilde{r}(k_{4})|^{\frac{1}{4}} > 0$ and $\lambda_{2} = |\tilde{r}(\tfrac{1}{k_{2}})|^{\frac{1}{4}} > 0$ by \eqref{def of lambda1} and \eqref{def of lambda2}, we conclude with the help of \eqref{IIbis d0estimate} and \eqref{IIbis d0estimate green} that $\overline{Z_{1,13}} = \lambda_1^4 Z_{1,31}$ and $\overline{Z_{2,23}} = \lambda_2^4 Z_{2,32}$.
Using also that $\lambda_1^4 = -\tilde{r}(k_4)$ and $\lambda_2^4 = -\tilde{r}(\frac{1}{k_2})$, it follows that
\begin{align*}
& \hat{n}_{3}^{(1)} = 2i \, \im \Big( Z_{1,13}(i (\omega k_{4})^{-1}-i (\omega k_{4})^{1}) + Z_{2,23}(i (\omega^{2} k_{2})^{-1}-i (\omega^{2} k_{2})^{1}) \Big) + O(t^{-1}\ln t) \\
& = 2 i \, \im \Bigg[ \frac{\beta_{12}^{(1)}\tilde{d}_{1,0}(i (\omega k_{4})^{-1}-i (\omega k_{4})^{1})}{-i \omega k_{4} z_{1,\star} \sqrt{t} e^{t \Phi_{31}(\zeta,\omega k_{4})}\big|\tilde{r}(k_{4})\big|^{-\frac{1}{2}}} + \frac{\beta_{12}^{(2)}\tilde{d}_{2,0}(i (\omega^{2} k_{2})^{-1}-i (\omega^{2} k_{2})^{1})}{-i \omega^{2} k_{2} z_{2,\star} \sqrt{t} e^{t \Phi_{32}(\zeta,\omega^{2} k_{2})}\big|\tilde{r}(\frac{1}{k_{2}})\big|^{-\frac{1}{2}}} \Bigg] + O\Big(\frac{\ln t}{t}\Big),
\end{align*}
as $t \to \infty$. 
Employing \eqref{betap1p and betap2p Sector IIbis} and the relations
\begin{align*}
& \tilde{d}_{1,0}(\zeta,t) = e^{i\arg \tilde{d}_{1,0}(\zeta,t)}, \qquad \tilde{d}_{2,0}(\zeta,t) = e^{\pi (2\nu_{2}-\nu_{4})}e^{i\arg \tilde{d}_{2,0}(\zeta,t)}, \\
& i (\omega k_{4})^{-1}-i (\omega k_{4})^{1} = 2 \sin(\arg(\omega k_{4}(\zeta))), \qquad i (\omega^{2} k_{2})^{-1}-i (\omega^{2} k_{2})^{1} = 2 \sin(\arg(\omega^{2} k_{2}(\zeta))), 
	\\
& |\beta_{12}^{(1)}| = \sqrt{\nu_1}, 
\qquad |\beta_{12}^{(2)}|  e^{\pi (2\nu_{2}-\nu_{4})} = \sqrt{\hat{\nu}_2}, \qquad |\tilde{r}(k_{4})|^{-\frac{1}{2}} = |\tilde{r}(\tfrac{1}{k_{4}})|^{\frac{1}{2}},
\end{align*}
we get \small
\begin{align*}
& \hat{n}_{3}^{(1)} = \frac{4i\sqrt{\nu_{1}}}{-i\omega k_{4} z_{1,\star}\sqrt{t}|\tilde{r}(\frac{1}{k_{4}})|^{\frac{1}{2}}}\sin(\arg(\omega k_{4}))\sin(\tfrac{3\pi}{4}-\arg \tilde{q}_1 +\arg\Gamma(i \nu_{1})+\arg \tilde{d}_{1,0}-t\im \Phi_{31}(\zeta,\omega k_{4})) \\
& +\frac{4i\sqrt{\hat{\nu}_{2}}|\tilde{r}(\frac{1}{k_{2}})|^{\frac{1}{2}}}{-i\omega^{2} k_{2} z_{2,\star}\sqrt{t}}\sin(\arg(\omega^{2} k_{2}))\sin(\tfrac{3\pi}{4}-\arg (q_{6}-q_{2}q_{5})+\arg\Gamma(i\hat{\nu}_{2})+\arg \tilde{d}_{2,0}-t\im \Phi_{32}(\zeta,\omega^{2} k_{2})) \\
& + O(t^{-1}\ln t), 
\end{align*}
\normalsize as $t \to \infty$ uniformly for $\zeta \in \mathcal{I}$. Formula \eqref{asymp for u sector V} now directly follows from \eqref{IIbis recoverun} (as in \cite{CLWasymptotics}, we can prove that the above error $O(t^{-1}\ln t)$ can be differentiated with respect to $x$ without getting worse) and the fact that $\frac{d}{d\zeta}[\im \Phi_{31}(\zeta,\omega k_{4}(\zeta))]= -\im k_{4}$ and $\frac{d}{d\zeta}[\im \Phi_{32}(\zeta,\omega^{2} k_{2}(\zeta))] = \im k_{2}$. This finishes the proof of Theorem \ref{asymptoticsth}.

\appendix 

\section{Model RH problems}\label{appendix A}
Let $X = X_1 \cup \cdots \cup X_4 \subset \C$ be the cross defined by
\begin{align} \nonumber
&X_1 = \bigl\{se^{\frac{i\pi}{4}}\, \big| \, 0 \leq s < \infty\bigr\}, && 
X_2 = \bigl\{se^{\frac{3i\pi}{4}}\, \big| \, 0 \leq s < \infty\bigr\},  
	\\ \label{Xdef}
&X_3 = \bigl\{se^{-\frac{3i\pi}{4}}\, \big| \, 0 \leq s < \infty\bigr\}, && 
X_4 = \bigl\{se^{-\frac{i\pi}{4}}\, \big| \, 0 \leq s < \infty\bigr\},
\end{align}
and oriented away from the origin, see Figure \ref{X.pdf}.
\begin{figure}
\begin{center}
 \begin{overpic}[width=.3\textwidth]{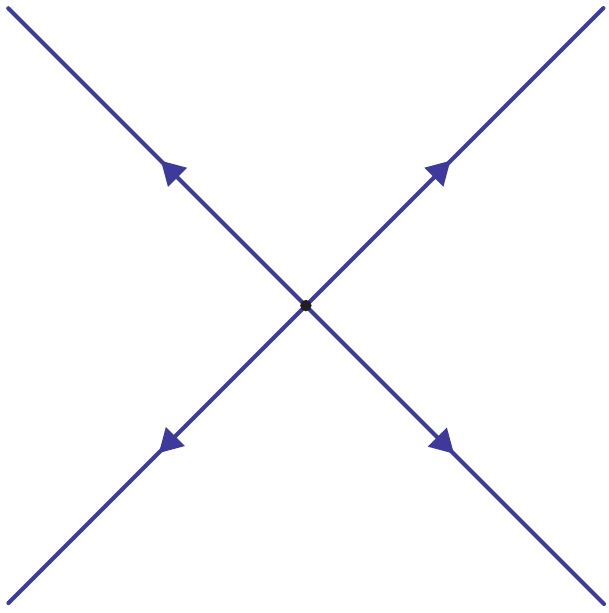}
      \put(74,68){\small $X_1$}
      \put(17,68){\small $X_2$}
      \put(15,27){\small $X_3$}
      \put(75,27){\small $X_4$}
      \put(48,42){$0$}
    \end{overpic}
     \begin{figuretext}\label{X.pdf}
        The contour $X = X_1 \cup X_2 \cup X_3 \cup X_4$.
     \end{figuretext}
     \end{center}
\end{figure}

\begin{lemma}[Model RH problem needed near $k=\omega k_{4}$]\label{IIbis Xlemma 3}
Let $q \in \mathbb{D}$, and define
\begin{align*}
\nu = -\tfrac{1}{2\pi} \ln(1 - |q|^2).
\end{align*}
Define the jump matrix $v^{X,(1)}(z)$ for $z \in X$ by
\begin{align}\nonumber 
& \begin{pmatrix} 1 & 0 & 0		\\
0 & 1 & 0 \\
 -q z^{-2i \nu} e^{\frac{iz^{2}}{2}} & 0 & 1 \end{pmatrix} \mbox{ if } z \in X_{1}, & & 
 \begin{pmatrix} 1 & 0 & \frac{-\bar{q}}{1 - |q|^2}z^{2i \nu} e^{-\frac{iz^{2}}{2}} 	\\
0 & 1 & 0 \\
0 & 0 & 1  \end{pmatrix} \mbox{ if } z \in X_{2}, 
	\\ \label{vX1def}
& \begin{pmatrix} 1 & 0 & 0 \\
0 & 1 & 0 \\
\frac{q}{1 - |q|^2} z^{-2i \nu} e^{\frac{iz^{2}}{2}}	& 0 & 1 \end{pmatrix} \mbox{ if } z \in X_{3}, & & \begin{pmatrix} 1 & 0 & \bar{q} z^{2i \nu}e^{-\frac{iz^{2}}{2}}	\\
0 & 1 & 0 \\
0 & 0 & 1 \end{pmatrix} \mbox{ if } z \in X_{4},
\end{align}
where $z^{i\nu}$ has a branch cut along $(-\infty,0]$, such that $z^{i\nu} = |z|^{i\nu}e^{-\nu  \arg(z)}$, $\arg(z) \in (-\pi,\pi)$. Then the RH problem 
\begin{enumerate}[$(a)$]
\item $m^{X,(1)}(\cdot) = m^{X,(1)}(q, \cdot) : \C \setminus X \to \mathbb{C}^{3 \times 3}$ is analytic;

\item on $X \setminus \{0\}$, the boundary values of $m^{X,(1)}$ exist, are continuous, and satisfy $m_+^{X,(1)} =  m_-^{X,(1)} v^{X,(1)}$;

\item $m^{X,(1)}(z) = I + O(z^{-1})$ as $z \to \infty$, and $m^{X,(1)}(z) = O(1)$ as $z \to 0$;
\end{enumerate}
has a unique solution $m^{X,(1)}(q,z)$. This solution satisfies
\begin{align}\label{IIbis mXasymptotics 3}
m^{X,(1)}(z) = I + \frac{m_{1}^{X,(1)}}{z} + O\biggl(\frac{1}{z^2}\biggr), \quad z \to \infty,  \quad m_{1}^{X,(1)}:=\begin{pmatrix} 
0 & 0 & \beta_{12}^{(1)} \\
0 & 0 & 0 \\
\beta_{21}^{(1)} & 0 & 0 \end{pmatrix},
\end{align}  
where the error term is uniform with respect to $\arg z \in [-\pi, \pi]$ and $q$ in compact subsets of $\mathbb{D}$, and
\begin{align}\label{IIbis betaXdef 3}
& \beta_{12}^{(1)} := \frac{e^{\frac{3\pi i}{4}}e^{\frac{\pi\nu}{2}}\sqrt{2\pi}\bar{q}}{(e^{\pi \nu}-e^{-\pi \nu})\Gamma(-i\nu)}, \qquad \beta_{21}^{(1)} := \frac{e^{-\frac{3\pi i}{4}}e^{\frac{\pi \nu}{2}}\sqrt{2\pi}q}{(e^{\pi \nu}-e^{-\pi \nu})\Gamma(i\nu)}.
\end{align}
(Note that $\beta_{12}^{(1)}\beta_{21}^{(1)} = \nu$ because $|\Gamma(i\nu)| = \frac{\sqrt{2\pi}}{\sqrt{\nu(e^{\pi\nu}-e^{-\pi\nu})}}$.)
\end{lemma}
Lemma \ref{II Xlemma 3 green} below was also used in \cite{CLmain} for the local analysis near $k=\omega^{2} k_{2}$ when $\frac{x}{t}\in (\frac{1}{\sqrt{3}},1)$. We provide here a proof of this lemma.
\begin{lemma}[Model RH problem needed near $k=\omega^{2} k_{2}$]\label{II Xlemma 3 green}
Let $q_{2}, q_{4},q_{5},q_{6} \in \mathbb{C}$ be such that $1 + |q_{2}|^2 - |q_{4}|^{2}>0$, $1 - |q_{5}|^2 - |q_{6}|^{2}>0$, and 
\begin{align}\label{II condition on q2q4q5q6}
q_{4}-\bar{q}_5-q_{2}\bar{q}_6=0.
\end{align}
Define $\nu_2, \nu_4, \nu_5 \in \R$ by
\begin{align*}
\nu_{2} = -\tfrac{1}{2\pi} \ln(1 + |q_{2}|^2), \quad \nu_{4} = -\tfrac{1}{2\pi} \ln(1 + |q_{2}|^2 - |q_{4}|^{2}), \quad \nu_{5} = -\tfrac{1}{2\pi} \ln(1 - |q_{5}|^2 - |q_{6}|^{2}).
\end{align*}
Define the jump matrix $v^{X,(2)}(z)$ for $z \in X$ by
\begin{align}\label{II vXdef 3 green} 
v^{X,(2)}(z) = \begin{cases}
\begin{pmatrix} 
1 & 0 & 0 \\
0 & 1 & 0		\\
0 & -\frac{q_{6}-q_{2}q_{5}}{1+|q_{2}|^{2}} z^{-i (2\nu_{5}-\nu_{4})} z_{(0)}^{-i (2\nu_{2}-\nu_{4})} e^{\frac{iz^{2}}{2}} & 1 \end{pmatrix}, &   z \in X_1, 
  	\\
\begin{pmatrix}
1 & 0 & 0 \\
0 & 1 & -\frac{\bar{q}_6-\bar{q}_2\bar{q}_5}{1-|q_{5}|^{2}-|q_{6}|^{2}} z^{i (2\nu_{5}-\nu_{4})} z_{(0)}^{i (2\nu_{2}-\nu_{4})} e^{-\frac{iz^{2}}{2}} 	\\
0 & 0 & 1  \end{pmatrix}, &  z \in X_2, 
	\\
\begin{pmatrix} 
1 & 0 & 0 \\
0 & 1 & 0 \\
0 & \frac{q_{6}-q_{2}q_{5}}{1-|q_{5}|^{2}-|q_{6}|^{2}} z^{-i (2\nu_{5}-\nu_{4})} z_{(0)}^{-i (2\nu_{2}-\nu_{4})} e^{\frac{iz^{2}}{2}}	& 1 \end{pmatrix}, &  z \in X_3,
	\\
 \begin{pmatrix} 
1 & 0 & 0 \\ 
0 & 1	& \frac{\bar{q}_6-\bar{q}_2\bar{q}_5}{1+|q_{2}|^{2}} z^{i (2\nu_{5}-\nu_{4})} z_{(0)}^{i (2\nu_{2}-\nu_{4})} e^{-\frac{iz^{2}}{2}}	\\
0 & 0	& 1 \end{pmatrix}, &  z \in X_4,
\end{cases}
\end{align}
where $z^{i\nu}$ has a branch cut along $(-\infty,0]$ and $z_{(0)}^{i\nu}$ has a branch cut along $[0,+\infty)$ such that $z^{i\nu} = |z|^{i\nu}e^{-\nu  \arg(z)}$, $\arg(z) \in (-\pi,\pi)$, and $z_{(0)}^{i\nu} = |z|^{i\nu}e^{-\nu  \arg_{0}(z)}$, $\arg_{0}(z) \in (0,2\pi)$. Then the RH problem 
\begin{enumerate}[$(a)$]
\item $m^{X,(2)}(\cdot) = m^{X,(2)}(q_{2}, q_{4},q_{5},q_{6}, \cdot) : \C \setminus X \to \mathbb{C}^{3 \times 3}$ is analytic;

\item on $X \setminus \{0\}$, the boundary values of $m^{X,(2)}$ exist, are continuous, and satisfy $m_+^{X,(2)} =  m_-^{X,(2)} v^{X,(2)}$;

\item $m^{X,(2)}(z) = I + O(z^{-1})$ as $z \to \infty$, and $m^{X,(2)}(z) = O(1)$ as $z \to 0$;
\end{enumerate}
has a unique solution $m^{X,(2)}(z)$. This solution satisfies
\begin{align}\label{II mXasymptotics 3 green}
m^{X,(2)}(z) = I + \frac{m_{1}^{X,(2)}}{z} + O\biggl(\frac{1}{z^2}\biggr), \quad z \to \infty,  \quad m_{1}^{X,(2)} := \begin{pmatrix} 
0 & 0 & 0 \\
0 & 0 & \beta_{12}^{(2)} \\ 
0 & \beta_{21}^{(2)} & 0 \end{pmatrix}, 
\end{align}
where the error term is uniform with respect to $\arg z \in [-\pi, \pi]$ and $q_{2},q_{4},q_{5},q_{6}$ in compact subsets of $\{q_{2},q_{4},q_{5},q_{6}\in \mathbb{C} \,|\, 1 + |q_{2}|^2 - |q_{4}|^{2}>0, 1 - |q_{5}|^2 - |q_{6}|^{2}>0, q_{4}-\bar{q}_5-q_{2}\bar{q}_6=0\}$, and 
\begin{align}\label{II betaXdef 3 green}
& \beta_{12}^{(2)} := \frac{e^{\frac{3\pi i}{4}}e^{\frac{\pi\hat{\nu}_2}{2}}e^{2\pi (\nu_{4}-\nu_{2})}\sqrt{2\pi}(\bar{q}_6-\bar{q}_2\bar{q}_5)}{(e^{\pi \hat{\nu}_2}-e^{-\pi \hat{\nu}_2})\Gamma(-i\hat{\nu}_2)}, \qquad \beta_{21}^{(2)} := \frac{e^{-\frac{3\pi i}{4}}e^{\frac{\pi \hat{\nu}_2}{2}}e^{2\pi\nu_{2}}\sqrt{2\pi}(q_{6}-q_{2}q_{5})}{(e^{\pi \hat{\nu}_2}-e^{-\pi \hat{\nu}_2})\Gamma(i\hat{\nu}_2)},
\end{align}
where $\hat{\nu}_2 := \nu_2 + \nu_5 - \nu_4$. (Note that $\beta_{12}^{(2)}\beta_{21}^{(2)} = \hat{\nu}_{2}$.)
 
\end{lemma}
\begin{proof}
Uniqueness of $m^{X,(2)}(z)$ follows from standard arguments.
Let us temporarily assume that $m^{X,(2)}(z)$ exists.
Define $m^{(X1)}(z) := m^{X,(2)}(z) z^{i (\nu_{5}-\frac{\nu_{4}}{2})\tilde{\sigma}} z_{(0)}^{i (\nu_{2}-\frac{\nu_{4}}{2})\tilde{\sigma}} e^{-\frac{iz^{2}}{4}\tilde{\sigma}}$ where $\tilde{\sigma} = \diag (0,1,-1)$. We verify that 
\begin{align*}
m_+^{(X1)}(z) =  m_-^{(X1)}(z) v^{(X1)}(z) \quad \text{for } \ z \in (X \cup \R) \setminus\{0\},  
\end{align*}
where $\R$ is oriented from left to right and $v^{(X1)}$ is given by
\begin{align*}
& \begin{pmatrix} 
1 & 0 & 0 \\
0 & 1 & 0		\\
0 & -\frac{q_{6}-q_{2}q_{5}}{1+|q_{2}|^{2}} & 1 \end{pmatrix} \mbox{ if } z \in X_1, & & \begin{pmatrix} 
1 & 0 & 0 \\
0 & 1 & -\frac{\bar{q}_{6}-\bar{q}_{2}\bar{q}_{5}}{1-|q_{5}|^{2}-|q_{6}|^{2}} \\
0 & 0 & 1  \end{pmatrix} \mbox{ if } z \in X_2, \\
& \begin{pmatrix} 
1 & 0 & 0 \\
0 & 1 & 0 \\
0 & \frac{q_{6}-q_{2}q_{5}}{1-|q_{5}|^{2}-|q_{6}|^{2}} & 1 \end{pmatrix} \mbox{ if } z \in X_3, & & \begin{pmatrix} 1 & 0 & 0 \\
0 & 1 & \frac{\bar{q}_{6}-\bar{q}_{2}\bar{q}_{5}}{1+|q_{2}|^{2}} \\
0 & 0 & 1 \end{pmatrix} \mbox{ if } z \in X_4, \\
& e^{\pi (\nu_{4}-2\nu_{5}) \tilde{\sigma}} \mbox{ if } z \in \R_-, & & e^{\pi (2\nu_{2}-\nu_{4})\tilde{\sigma}} \mbox{ if } z \in \R_+.
\end{align*}
We next collapse the contours $X_1$ and $X_4$ onto $\R_+$ and the contours $X_2$ and $X_3$ onto $\R_-$. To this end we let
\begin{align}\label{II psimX1 3 green}
\psi(z) =  m^{(X1)}(z) B(z), 
\end{align}
where $B(z)$ is equal to
\begin{align*}
\begin{pmatrix} 
1 & 0 & 0 \\
0 & 1 & 0		\\
0 & -\frac{q_{6}-q_{2}q_{5}}{1+|q_{2}|^{2}} & 1 \end{pmatrix}, \; \begin{pmatrix} 
1 & 0 & 0 \\
0 & 1 & \frac{\bar{q}_{6}-\bar{q}_{2}\bar{q}_{5}}{1-|q_{5}|^{2}-|q_{6}|^{2}} \\
0 & 0 & 1  \end{pmatrix}, \; \begin{pmatrix} 
1 & 0 & 0 \\
0 & 1 & 0 \\
0 & \frac{q_{6}-q_{2}q_{5}}{1-|q_{5}|^{2}-|q_{6}|^{2}} & 1 \end{pmatrix}, \; \begin{pmatrix} 
1 & 0 & 0 \\
0 & 1 & -\frac{\bar{q}_{6}-\bar{q}_{2}\bar{q}_{5}}{1+|q_{2}|^{2}} \\
0 & 0 & 1 \end{pmatrix}
\end{align*}
for $\arg z \in (0, \frac{\pi}{4})$, $\arg z \in (\frac{3\pi}{4}, \pi)$, $\arg z \in (-\pi, -\frac{3\pi}{4})$, $\arg z \in (-\frac{\pi}{4},0)$, respectively, and $B(z)=I$ otherwise. The function $\psi$ satisfies $\psi_+(z) =  \psi_-(z) v^\psi(z)$ for $z \in \mathbb{R}\setminus \{0\}$, where $v^\psi$ is piecewise constant and defined by
\begin{align*}
& v^{\psi}(z) = \begin{pmatrix} 1 & 0 & 0 \\
0 & 1 & \frac{\bar{q}_{6}-\bar{q}_{2}\bar{q}_{5}}{1+|q_{2}|^{2}} \\
0 & 0 & 1 \end{pmatrix} e^{\pi (2\nu_{2}-\nu_{4})\tilde{\sigma}} \begin{pmatrix} 1 & 0 & 0 \\
0 & 1 & 0 \\
0 & -\frac{q_{6}-q_{2}q_{5}}{1+|q_{2}|^{2}} & 1  \end{pmatrix} \nonumber \\
& \hspace{0.92cm} = \begin{pmatrix}
1 & 0 & 0 \\
0 & e^{\pi(2\nu_{2}-\nu_{4})} - e^{-\pi(2\nu_{2}-\nu_{4})}\frac{|q_{6}-q_{2}q_{5}|^{2}}{(1+|q_{2}|^{2})^{2}} & e^{-\pi(2\nu_{2}-\nu_{4})}\frac{\bar{q}_{6}-\bar{q}_{2}\bar{q}_{5}}{1+|q_{2}|^{2}} \\
0 & -e^{-\pi(2\nu_{2}-\nu_{4})}\frac{q_{6}-q_{2}q_{5}}{1+|q_{2}|^{2}} & e^{-\pi(2\nu_{2}-\nu_{4})}
\end{pmatrix}, & & z \in \R_{+}, \nonumber \\
& v^{\psi}(z) = \begin{pmatrix} 1 & 0 & 0 \\
0 & 1 & 0 \\
0 & -\frac{q_{6}-q_{2}q_{5}}{1-|q_{5}|^{2}-|q_{6}|^{2}} & 1 \end{pmatrix} e^{\pi (\nu_{4}-2\nu_{5}) \tilde{\sigma}} \begin{pmatrix} 1 & 0 & 0 \\
0 & 1 & \frac{\bar{q}_{6}-\bar{q}_{2}\bar{q}_{5}}{1-|q_{5}|^{2}-|q_{6}|^{2}} \\
0 & 0 & 1  \end{pmatrix}  \nonumber \\
&  = \begin{pmatrix}
1 & 0 & 0 \\
0 & e^{\pi(\nu_{4}-2\nu_{5})} & e^{\pi(\nu_{4}-2\nu_{5})}\frac{\bar{q}_{6}-\bar{q}_{2}\bar{q}_{5}}{1-|q_{5}|^{2}-|q_{6}|^{2}} \\
0 & -e^{\pi(\nu_{4}-2\nu_{5})}\frac{q_{6}-q_{2}q_{5}}{1-|q_{5}|^{2}-|q_{6}|^{2}} & e^{-\pi(\nu_{4}-2\nu_{5})}-e^{\pi(\nu_{4}-2\nu_{5})} \frac{|q_{6}-q_{2}q_{5}|^{2}}{(1-|q_{5}|^{2}-|q_{6}|^{2})^{2}}
\end{pmatrix}, & & z \in \R_{-}.
\end{align*}
Using (\ref{II condition on q2q4q5q6}), it is a simple calculation to verify that $v^{\psi}(z)$ is constant on the whole real line. This implies that $(\partial_{z}\psi) \psi^{-1}$ is an entire function. Note that
\begin{align}\label{psiatinfinity}
\psi(z) & = (I+m_{1}^{X,(2)}z^{-1}+O(z^{-2}))z^{i (\nu_{5}-\frac{\nu_{4}}{2})\tilde{\sigma}} z_{(0)}^{i (\nu_{2}-\frac{\nu_{4}}{2})\tilde{\sigma}} e^{-\frac{iz^{2}}{4}\tilde{\sigma}}B(z) \qquad \mbox{as } z \to \infty.
\end{align} 
Hence, assuming that the asymptotics can be differentiated termwise,
\begin{align}\label{II psi one diff 3 green}
(\partial_{z}\psi) \psi^{-1} = -\frac{i z}{2}\tilde{\sigma} + \begin{pmatrix}
0 & 0 & 0 \\
0 & 0 & i\beta_{12}^{(2)} \\ 
0 & -i\beta_{21}^{(2)} & 0
\end{pmatrix},
\end{align}
for some $\beta_{12}^{(2)},\beta_{21}^{(2)} \in \mathbb{C}$. Differentiating a second time we get
$$\begin{cases}
\partial_{z}^2 \psi_{33} + (\frac{z^2}{4} - \frac{i}{2} - \beta_{12}^{(2)}\beta_{21}^{(2)})\psi_{33} = 0,\\
\partial_{z}^2 \psi_{22} + (\frac{z^2}{4} + \frac{i}{2} - \beta_{12}^{(2)}\beta_{21}^{(2)})\psi_{22} = 0, 
\end{cases} \qquad z \in \C \setminus \R.$$
By \cite[Chapter 12]{NIST}, 
\begin{subequations}\label{psi33psi22}
\begin{align}
& \psi_{33}(z) = \begin{cases} c_1 D_{-i\nu}(e^{\frac{3\pi i}{4}}z) + c_2 D_{-i\nu}(e^{-\frac{\pi i}{4}}z), & \im z > 0, \\
c_3 D_{-i\nu}(e^{\frac{3\pi i}{4}}z) + c_4 D_{-i\nu}(e^{-\frac{\pi i}{4}}z), \quad & \im z < 0,
\end{cases}
	\\
& \psi_{22}(z) = \begin{cases} c_5 D_{i\nu}(e^{-\frac{3\pi i}{4}}z) + c_6 D_{i\nu}(e^{\frac{\pi i}{4}}z), & \im z > 0, \\
c_7 D_{i\nu}(e^{-\frac{3\pi i}{4}}z) + c_8 D_{i\nu}(e^{\frac{\pi i}{4}}z), \quad & \im z < 0.
\end{cases}
\end{align}
\end{subequations}
for some constants $\{c_j\}_1^8$, where $\nu := \beta_{12}^{(2)}\beta_{21}^{(2)}$ and $D_{\nu}$ denotes the parabolic cylinder function. The $23$ and $32$ entries of $\psi$ can be found directly using \eqref{II psi one diff 3 green}:
\begin{align*}
\psi_{23}(z) = \frac{i\psi_{33}'(z)+\frac{z}{2}\psi_{33}(z)}{\beta_{21}^{(2)}}, \qquad \psi_{32}(z) = \frac{-i\psi_{22}'(z)+\frac{z}{2}\psi_{22}(z)}{\beta_{12}^{(2)}},
\end{align*}
where primes denote derivative with respect to $z$. Hence, we find
\begin{align}
\psi = \begin{pmatrix} 1 & 0 & 0 \\
0 &  \psi_{22} & \frac{i\psi_{33}'+\frac{z}{2}\psi_{33}}{\beta_{21}^{(2)}} \\
0 & \frac{-i\psi_{22}'+\frac{z}{2}\psi_{22}}{\beta_{12}^{(2)}}  &  \psi_{33} \end{pmatrix}, \qquad z \in \C\setminus \R.
\end{align}
From (\ref{psiatinfinity}), we obtain
\begin{align*}
& \resizebox{1.03\hsize}{!}{$\psi_{33}(z) = z^{-i (\nu_{5}-\frac{\nu_{4}}{2})} z_{(0)}^{-i (\nu_{2}-\frac{\nu_{4}}{2})}e^{\frac{iz^{2}}{4}}(1+O(z^{-1})) = \begin{cases}
z^{-i (\nu_{2}+\nu_{5}-\nu_{4})} e^{\frac{iz^{2}}{4}}(1+O(z^{-1})), \\
e^{-\pi(\nu_{4}-2\nu_{2})}z^{-i (\nu_{2}+\nu_{5}-\nu_{4})}e^{\frac{iz^{2}}{4}}(1+O(z^{-1})),
\end{cases}$}
	\\
& \psi_{22}(z) = z^{i (\nu_{5}-\frac{\nu_{4}}{2})} z_{(0)}^{i (\nu_{2}-\frac{\nu_{4}}{2})}e^{-\frac{iz^{2}}{4}}(1+O(z^{-1})) = \begin{cases}
z^{i (\nu_{2}+\nu_{5}-\nu_{4})} e^{-\frac{iz^{2}}{4}}(1+O(z^{-1})),  \\
e^{\pi(\nu_{4}-2\nu_{2})}z^{i (\nu_{2}+\nu_{5}-\nu_{4})}e^{-\frac{iz^{2}}{4}}(1+O(z^{-1})), 
\end{cases} 
\end{align*}
as $z \to \infty$, where in each bracket the first and second lines apply for $\arg z \in (\epsilon,\pi-\epsilon)$ and $\arg z \in (-\pi+\epsilon,-\epsilon)$, respectively (for any $\epsilon \in (0,\frac{\pi}{2})$). Since
\begin{align*}
D_{\nu}(z) = z^{\nu}e^{-\frac{z^{2}}{4}}(1+O(z^{-2})),
\end{align*}
as $z \to \infty$, $\arg z \in (-\frac{3\pi}{4}+\epsilon,\frac{3\pi}{4}-\epsilon)$, these asymptotic formulas are consistent with (\ref{psi33psi22}) provided that we choose $\nu=\nu_{2}+\nu_{5}-\nu_{4}$ and
\begin{align*}
& c_{1} = 0, & & c_{2} = e^{\frac{\pi \nu}{4}}, & & c_{3} = e^{-\frac{3\pi \nu}{4}}e^{-\pi (\nu_{4}-2\nu_{2})}, & & c_{4}=0,
	\\
& c_{5}=e^{-\frac{3\pi \nu}{4}}, & & c_{6} = 0, & & c_{7} = 0, & & c_{8} = e^{\frac{\pi \nu}{4}}e^{\pi (\nu_{4}-2\nu_{2})}.
\end{align*}
With this choice, we get
\begin{subequations}\label{psi33psi22final}
\begin{align}
& \psi_{33}(z) = \begin{cases} 
e^{\frac{\pi \nu}{4}} D_{-i\nu}(e^{-\frac{\pi i}{4}}z), & \im z > 0, \\
e^{-\frac{3\pi \nu}{4}}e^{-\pi (\nu_{4}-2\nu_{2})} D_{-i\nu}(e^{\frac{3\pi i}{4}}z), \quad & \im z < 0,
\end{cases}
	\\
& \psi_{22}(z) = \begin{cases} 
e^{-\frac{3\pi \nu}{4}} D_{i\nu}(e^{-\frac{3\pi i}{4}}z), & \im z > 0, \\
e^{\frac{\pi \nu}{4}}e^{\pi (\nu_{4}-2\nu_{2})} D_{i\nu}(e^{\frac{\pi i}{4}}z), \quad & \im z < 0.
\end{cases} 
\end{align}
\end{subequations}
and evaluation of the jump relation $\psi_+ = \psi_- v^\psi$ at $z = 0$ then gives \eqref{II betaXdef 3 green}. The expressions in (\ref{psi33psi22final}) lead to an explicit formula for $m^{X,(2)}$ in terms of parabolic cylinder functions. This expression was established under certain assumptions. However, it can now be verified directly that the constructed function $m^{X,(2)}(z)$ indeed satisfies the stated RH problem and the asymptotic formula (\ref{II mXasymptotics 3 green}).
\end{proof}

\subsection*{Acknowledgements}
Support is acknowledged from the Novo Nordisk Fonden Project, Grant 0064428, the European Research Council, Grant Agreement No. 682537, the Swedish Research Council, Grant No. 2015-05430, Grant No. 2021-04626, and Grant No. 2021-03877, the G\"oran Gustafsson Foundation, and the Ruth and Nils-Erik Stenb\"ack Foundation.

\bibliographystyle{plain}
\bibliography{is}

\end{document}